\documentclass{amsart}
\usepackage{amsmath}
\usepackage{paralist}
\usepackage{amssymb}
\usepackage{amsthm}
\usepackage[colorlinks=true]{hyperref}
\usepackage{appendix}
\hypersetup{urlcolor=blue, citecolor=red}

\numberwithin{equation}{section}

\newtheorem{Thm}{Theorem}[section]
\newtheorem{Lem}[Thm]{Lemma}
\newtheorem{Cor}[Thm]{Corollary}
\newtheorem{Prop}[Thm]{Proposition}

\newtheorem{Def}[Thm]{Definition}
\newtheorem{Rem}[Thm]{Remark}

\begin{document}
\sloppy
\allowdisplaybreaks[4]
\title[Wellposedness of the degenerate Master Equation]{Wellposedness of the Master Equation for Mean Field Games with Grushin Type Diffusion}

\author{Yiming Jiang, Yawei Wei, Yiyun Yang}
\address{School of Mathematical Sciences and LPMC\\ Nankai University\\ Tianjin 300071 China}
\email{ymjiangnk@nankai.edu.cn}
\address{School of Mathematical Sciences and LPMC\\ Nankai University\\ Tianjin 300071 China}
\email{weiyawei@nankai.edu.cn}
\address{School of Mathematical Sciences\\ Nankai University\\ Tianjin 300071 China}
\email{1120210035@mail.nankai.edu.cn}
\thanks{Acknowledgements: This work is supported by the NSFC under the grands 12271269 and supported by the Fundamental Research Funds for the Central Universities.}

\begin{abstract}
We study the wellposedness of the master equation for a second-order mean field games with the Grushin type diffusion. In order to do this, we obtain the properties of its solution by investigating a degenerate mean field games system for which there exists an equivalent characterization with the master equation. The crucial points of this paper are to explore some regularities of solutions to two types of linear degenerate partial differential equations and a kind of degenerate linear coupled system so as to derive the existence of solutions to the master equation.
\end{abstract}

\maketitle
\section{Introduction}
\noindent
In this paper, we study a kind of non-linear first-order partial differential equation (PDE in short) stated on the space of probability measures defined on the two-dimensional torus, called master equation as follows
\begin{equation}\label{ME}
\begin{cases}
-\partial_t U(t, x, m)-\Delta_{\mathcal{X}} U(t, x, m)+H\left(x, D_{\mathcal{X}} U(t, x, m)\right)\\
\quad-\int_{\mathbb{T}^2} \Delta_{\mathcal{X}}^y\frac{\delta U}{\delta m}(t, x, m, y) d m(y)\\
\quad+\int_{\mathbb{T}^2} D_{\mathcal{X}}^y \frac{\delta U}{\delta m} (t, x, m, y) \cdot D_p H\left(y, D_{\mathcal{X}} U(t, y, m)\right) d m(y)\\
=F(x, m),\quad \text{in } [0, T] \times \mathbb{T}^2 \times \mathcal{P}\left(\mathbb{T}^2\right), \\
U(T, x, m)=G(x, m),\quad \text{in } \mathbb{T}^2 \times \mathcal{P}\left(\mathbb{T}^2\right),
\end{cases}
\end{equation}
where $H: \mathbb{T}^2 \times \mathbb{R}^2 \rightarrow \mathbb{R}$ is given by $H(x,p):=\frac{1}{2}|p|^2$, and $\frac{\delta U}{\delta m}$ is the first order derivative of $U$ with respect to the measure $m$ (see Definition \ref{derivative w.r.t. m}). In addition, we call $\mathcal{X}=\{X_1,X_2\}:=\{\partial_{x_1},x_1\partial_{x_2}\}$ as Grushin vector fields, which is a typical class of vector fields with an anisotropic structure satisfying the H\"{o}rmander condition (cf. \cite{70Gr,13Br}). For any function $f:\mathbb{T}^2 \rightarrow \mathbb{R}$, we define the subgradient and the hypoelliptic operator associated to $\mathcal{X}$ respectively as
$$
D_{\mathcal{X}}f:=(X_1f,X_2f)^T,\quad\Delta_{\mathcal{X}}f:=\sum_{i=1}^{2}X_i^2f.
$$
While for any vector-valued function $g:\mathbb{T}^2 \rightarrow \mathbb{R}^2$, the corresponding divergence is defined as
$$
\operatorname{div}_{\mathcal{X}}g=X_1g+X_2g.
$$
For the sake of distinction, we use the superscript $y$ in differential operators to denote the operators acting on the $y$ variable.

Heuristically, equation \eqref{ME} is a Hamilton-Jacobi-Bellman (HJB in short) equation in the space of measures, arising from the limit problem of a differential game with finitely many indistinguishable players, in which the dynamic of player $i$, with $1\leq i \leq N$, is driven by the Grushin type diffusion, which means the player may have a ``forbidden" direction on a vertical line (cf. \cite{23JRWX}), namely
\begin{equation*}
\begin{cases}
d X_{s}^i=\left(\begin{array}{l}\alpha^{i,1}\\ \alpha^{i,2} X_{s}^{i,1} \end{array}\right) d s+\sqrt{2}\left(\begin{array}{ll}1 & 0 \\ 0 & X_{s}^{i,1}\end{array}\right) d B_{s}^{i}, \\
X_{t}^i=(x^{i,1},x^{i,2})^T, \text{ in } \mathbb{T}^2,
\end{cases}
\end{equation*}
where $X_s^i=(X_{s}^{i,1},X_{s}^{i,2})^T$ represents the state of the $i$-th player, $\alpha^i=(\alpha^{i,1},\alpha^{i,2})^T$ is the control chosen from a certain set $A_i$, and $\{B_{s}^{i}=(B_s^{i,1},B_s^{i,2})^T\}_{i=1,\ldots,N}$ are independent two-dimensional standard Brownian motions. The player $i$ choose his own strategy in order to minimize the cost function
\begin{equation*}
J^{N,i}(t, \boldsymbol{x}, (\alpha^i)_{i=1,\ldots,N})=\mathbb{E}\left[\int_{t}^{T} \left(L\left(X_{s}^i, \alpha_{s}^i\right)+F^{N,i}\left(\boldsymbol{X}_{s}\right)\right) d s+G^{N,i}\left(\boldsymbol{X}_{T}\right)\right],
\end{equation*}
where the notation $\boldsymbol{x}$ indicates a vector of $\mathbb{T}^{2 \times N}$ defined by $\boldsymbol{x}=(x^1,\ldots,x^N)$, and the same is for $\boldsymbol{X}$. Also here, the Lagrangian $L\left(X_{s}^i, \alpha_{s}^i\right):=\frac{1}{2}|\alpha_{s}^i|^2 $, while $F^{N,i}$ and $G^{N,i}$ are respectively the running and the final cost function associated with the $i$-th player. Now, denoting $\left(v^{N,i}\right)_{i=1,\ldots,N}$ as the value functions of the players, we say that the controls $\left(\hat{\alpha}^i\right)_{i=1,\ldots,N}$ provide a Nash equilibrium if the following inequality holds for all controls $\left(\alpha^i\right)_{i=1,\ldots,N}$ and for all $i$,
\begin{equation*}
v^{N,i}(t,\boldsymbol{x}):=J^{N,i}\left(t,\boldsymbol{x},\left(\hat{\alpha}^i\right)_{i=1,\ldots,N}\right) \leq J^{N,i}\left(t,\boldsymbol{x},\alpha^i,\left(\hat{\alpha}^j\right)_{j \neq i}\right),
\end{equation*}
this means the $i$-th player takes his own optimal strategy, while the other players have taken the control provided by the Nash equilibrium. By using the It\^{o}'s formula and the dynamic programming principle, we can deduce the HJB equations that $\left(v^{N,i}\right)_{i=1,\ldots,N}$ satisfy, namely the Nash system as follows
\begin{equation}\label{Nash system}
\begin{cases}
-\partial_{t} v^{N,i}(t, \boldsymbol{x})-\sum\limits_{j=1}^{N} \Delta_{\mathcal{X}}^{x^j}v^{N,i}(t, \boldsymbol{x})+H^{i}\left(x^i,D_{\mathcal{X}}^{x^i} v^{N,i}(t,\boldsymbol{x})\right)\\
\quad-\sum\limits_{j \neq i}D_pH^j\left(x^j,D_{\mathcal{X}}^{x^j} v^{N,j}(t,\boldsymbol{x})\right) \cdot D_{\mathcal{X}}^{x^j} v^{N,i}(t,\boldsymbol{x})=F^{N,i}(\boldsymbol{x}),\\
\quad \text {in }[0,T] \times \mathbb{T}^{2 \times N}, i \in\{1, \ldots, N\},\\
v^{N,i}(T,\boldsymbol{x})=G^{N,i}(\boldsymbol{x}), \quad \text {in }\mathbb{T}^{2 \times N}.
\end{cases}
\end{equation}
Note that the Hamiltonian of the system is the Fenchel conjugate of the Lagrangian $L$, namely
\begin{equation*}
H^i(x,p):=\sup_{a \in A_i} \left(-a \cdot p-\frac{1}{2}|a|^2\right)=\frac{1}{2}|p|^2.
\end{equation*}
This leads to the optimal feedback strategies as
\begin{equation*}
\left(\hat{\alpha}^i(t,\boldsymbol{x})=-D_pH^i\left(x^i,D_{\mathcal{X}}^{x^i} v^{N,i}(t,\boldsymbol{x})\right)\right)_{i=1,\ldots,N}.
\end{equation*}
Due to the symmetry of the game, we can suppose that $F^{N,i}$, $G^{N,i}$ and $v^{N,i}$ take the form
\begin{equation*}
F^{N,i}(\boldsymbol{x}) \simeq F(x^i,m_{\boldsymbol{x}}^{N,i}),\quad G^{N,i}(\boldsymbol{x}) \simeq G(x^i,m_{\boldsymbol{x}}^{N,i}),
\end{equation*}
and
\begin{equation*}
v^{N,i}(t,\boldsymbol{x}) \simeq U(t,x^i,m_{\boldsymbol{x}}^{N,i}),
\end{equation*}
where $m_{\boldsymbol{x}}^{N,i}:=\frac{1}{N-1}\sum\limits_{j \neq i}\delta_{x^j}$ is the empirical measure of the players except player $i$, and $U$ maps from $[0,T]\times \mathbb{T}^2 \times \mathcal{P}(\mathbb{T}^2)$ to $\mathbb{R}$. Similar to the method in \cite[Proposition 6.1.1]{19CDLL} to compute the relevant derivatives, we can have
\begin{equation*}
D_{\mathcal{X}}^{x^j} v^{N,i}(t,\boldsymbol{x}) \simeq
\begin{cases}
D_{\mathcal{X}}U\left(t,x^i,m_{\boldsymbol{x}}^{N,i}\right), & \mbox{if } j=i, \\
\frac{1}{N-1}D_{\mathcal{X}}^{x^j}\frac{\delta U}{\delta m}\left(t,x^i,m_{\boldsymbol{x}}^{N,i},x^j\right), & \mbox{otherwise},
\end{cases}
\end{equation*}
and
\begin{equation*}
\Delta_{\mathcal{X}}^{x^j} v^{N,i}(t,\boldsymbol{x}) \simeq
\begin{cases}
\Delta_{\mathcal{X}}U\left(t,x^i,m_{\boldsymbol{x}}^{N,i}\right), & \mbox{if } j=i,\\
\frac{1}{N-1}\Delta_{\mathcal{X}}^{x^j}\frac{\delta U}{\delta m}\left(t,x^i,m_{\boldsymbol{x}}^{N,i},x^j\right)\\
\quad+\left(\frac{1}{N-1}\right)^2 \operatorname{Tr}\left[D_{\mathcal{X}}^{x^j,x^j}\frac{\delta^2 U}{\delta m^2}\right]\left(t,x^i,m_{\boldsymbol{x}}^{N,i},x^j,x^j\right), & \mbox{otherwise}.
\end{cases}
\end{equation*}
Substituting the above relations into the Nash system \eqref{Nash system} and then letting $N \rightarrow \infty$, we thus obtain the master equation of the form in \eqref{ME}.

The mean field games (MFG in short) theory is usually utilized to solve the problem of differential games with infinitely many small and indistinguishable players. This theory was first introduced in 2006 by Lasry and Lions \cite{06LL_I,06LL_II,07LL,08LLG}. In the same years, Huang, Caines and Malham\'{e} \cite{06HCM} also established similar definitions. The limit problem in MFG theory boils down to the study of a coupled system of PDEs, which consists of a backward HJB equation satisfied by the value function $u$ of individual players and a forward Kolmogorov-Fokker-Planck (KFP in short) equation satisfied by the distribution law $m$ of the population. It was proved by Lions in his lectures at Coll\`{e}ge de France \cite{Li} that there exists an equivalent characterization between the solutions to the MFG system and the solutions to the master equation. Thus the study of the Nash equilibria of a MFG can reduce to the analysis of one unique equation. It is worth mentioning that the form of the MFG system considered in this paper is as follows
\begin{equation}\label{MFG system}
\begin{cases}
-\partial_t u-\Delta_{\mathcal{X}}u+H(x, D_{\mathcal{X}} u)=F(x, m(t)), & \mbox{in } [0,T] \times \mathbb{T}^2, \\
\partial_t m-\Delta_{\mathcal{X}} m-\operatorname{div}_{\mathcal{X}}\left(m D_pH(x, D_{\mathcal{X}}u)\right)=0, & \mbox{in } [0,T] \times \mathbb{T}^2, \\
u(T,x)=G(x, m(T)),\quad m(0)=m_0, & \mbox{in } \mathbb{T}^2,
\end{cases}
\end{equation}
which is composed of two degenerate parabolic equations.

In the last decade, different papers have investigated the master equation and presented the most critical issues such as existence, uniqueness and regularity results. For instance, it was recharacterized by Bensoussan, Frehse and Yam in \cite{15BFY,17BFY} as a set of PDEs on $L^{2}$ spaces, while in \cite{14CD} it was interpreted as a decoupled field of infinite-dimensional forward-backward stochastic differential equations by Carmona and Delarue. Furthermore, Gangbo and \'{S}wi\c{e}ch \cite{15GS} proved a small time existence for the first-order master equation. And it is worth noting that Cardaliaguet et al. \cite{19CDLL} provided some very general and well-known good treatment results of the wellposedness of classical solutions to both first and second-order master equation, corresponding to the nonlocal MFG without and with common noise, respectively. Notably, we can refer to \cite{23JR} for new findings and clarifications concerning the results in \cite{19CDLL}. From a probabilistic point of view, Chassagneux et al. \cite{14CCD} proved for the first time the wellposedness of  the first-order master equation using a probabilistic approach. While Buckhdan et al. \cite{17BLPR} proved the existence of a classical solution in the case of no coupling and no common noise using probabilistic arguments. In recent years, Ricciardi \cite{21Ri} studied the wellposedness of the MFG master equation in a framework of Neumann boundary condition. For the first time, Gangbo et al. \cite{22GMMZ} provided the first global in time well-posedness result in the case of non-separable displacement monotone Hamiltonians and non-degenerate idiosyncratic noise, which is the second important breakthrough since the work of \cite{19CDLL}. In addition to classical solutions, recent researches on weak solutions are likewise important. Bertucci presented the notion of monotone solution of master equations in the case of finite or continuous state space in \cite{21Be_F,21Be_C}. Meanwhile, Cardaliaguet and Souganidis \cite{21CS} introduced a notion of weak solution to the master equation without idiosyncratic noise.

Compared to the case of uniform elliptic, the researches on the degenerate MFG problems have emerged more recently. For the studies of hypoelliptic MFG, Dragoni and Feleqi \cite{18DF} studied the ergodic MFG systems with H\"{o}rmander diffusion, which is a class of systems of degenerate elliptic PDEs satisfying H\"{o}rmander condition. While Feleqi et al. \cite{20FGT} considered hypoelliptic MFG systems with quadratic Hamiltonians and proved the existence and uniqueness of the solution using the technique of Hopf-Cole transform. Furthermore, Mimikos-Stamatopoulos \cite{21Mi} considered the hypoelliptic MFG system with local coupling, driven by a KFP diffusion. While Jiang et al. \cite{23JRWX} proved the global wellposedness of the hypoelliptic MFG systems with Grushin structure. As for the general degenerate MFG, Cardaliaguet et al. \cite{15CGPT} proved existence and uniqueness of a suitably defined weak solution to the degenerate parabolic MFG system with local coupling. And Ferreira et al. \cite{21FGT} extended the existence of weak solutions to a wide class of time-dependent degenerate MFG systems. Moreover, Cardaliaguet et al. \cite{22CSS} built a new notion of probabilistically weak solutions for the MFG systems with common noise and degenerate idiosyncratic noise. There are fewer researches on the degenerate MFG master equations. Unlike this paper, Bayraktar et al. \cite{19BCCD} studied the finite state MFG master equation with Wright-Fisher common noise, which is a degenerate parabolic second-order PDE set on the simplex. Bansil et al. \cite{23BMM} constructed global in time classical solutions to degenerate MFG master equations without idiosyncratic noise, and this work was compared to \cite{22GMMZ} for the extension under lower level regularity assumptions on the data, and the weaker version of the displacement monotonicity condition on the Hamiltonians. More recently, Bansil and M\'{e}sz\'{a}ros \cite{24BM} proposed novel monotonicity conditions applicable for MFG and established new global well-posedness results for the associated master equations in the case of potentially degenerate idiosyncratic noise.

\section{Notation, assumptions and main results}\label{Sec_2}
\noindent
Throughout this paper, suppose that $T > 0$ is a fixed time, and $\mathbb{T}^d:=\mathbb{R}^d / \mathbb{Z}^d$ denotes the d-dimensional torus, which is a bounded compact space with distance $d_{\mathbb{T}^d}$.

Given a family of vector fields $\mathcal{X}=\{X_1,\ldots,X_m\}$ defined on $\mathbb{T}^d$, which satisfies the H\"{o}rmander condition. (For more details about H\"{o}rmander vector fields we refer to \cite{06Mo}.) The dual vector fields are defined by $X_i^*:=-X_i-\operatorname{div}X_i$ where $\operatorname{div}X_i$ indicate the standard (Euclidean) divergence of the vector fields $X_i:\mathbb{T}^d\rightarrow\mathbb{R}^d$. For any $x,y \in \mathbb{T}^d$, the Carnot-Carath\'{e}odory distance induced by the family $\mathcal{X}=\{X_1,\ldots,X_m\}$ is defined as
$$
d_{cc}(x, y):=\inf \{l(\gamma) \mid \gamma:[0, T] \rightarrow \mathbb{T}^d \text{ is } \mathcal{X}\text{-subunit}, \gamma(0)=x, \gamma(T)=y\},
$$
where we call $\mathcal{X}$-subunit any absolutely continuous curve $\gamma$ such that
$$
\gamma^{\prime}(t)=\sum_{j=1}^m \lambda_j(t) X_j(\gamma(t)), \text{ a.e. }t \in (0,T),
$$
and $l(\gamma)$ is defined as the length valued as
$$
l(\gamma):=\int_{0}^{T}\sqrt{\sum_{j=1}^{m}\lambda_j^2(t)}dt.
$$
For $x \in \mathbb{T}^d$, we introduce the $d_{cc}$-ball as
$$
B_r(x)=\{y \in \mathbb{T}^d:d_{cc}(x,y)<r\}.
$$
The H\"{o}rmander condition implies that the distance $d_{cc}(x, y)$ is finite and continuous with respect to the original Euclidean topology induced on $\mathbb{T}^d$. We also know that there exists $C>0$ such that
\begin{equation}\label{d_cc VS Euclidean}
C^{-1}d_{\mathbb{T}^d}(x,y) \leq d_{cc}(x, y) \leq Cd_{\mathbb{T}^d}(x,y)^{\frac{1}{k}},
\end{equation}
where $k \in \mathbb{N}$ is the step, i.e. the maximum of the degrees of the iterated brackets occurring in the fulfillment of the H\"{o}rmander condition. In particular, as for the Grushin vector fields $\mathcal{X}=\{X_1,X_2\}=\{\partial_{x_1},x_1\partial_{x_2}\}$, we have $d=2$, $X_i^*=-X_i$ and the step $k=2$.

The set $\mathcal{P}\left(\mathbb{T}^{d}\right)$ of Borel probability measures on $\mathbb{T}^{d}$ is endowed with the Kantorovich-Rubinstein distance
$$
d_1\left(m, m'\right):=\sup_\phi \int_{\mathbb{T}^d} \phi(y) d\left(m-m'\right)(y)
$$
where the supremum is taken over all $d_{cc}$-Lipschitz continuous maps $\phi: \mathbb{T}^d \rightarrow \mathbb{R}$ with a Lipschitz constant bounded by $1$. It can be known that $d_1$ is well-defined and this distance metricizes the weak convergence of measures (cf. \cite[Section 5.1]{18CD}). Also we have the following equivalent definition:
$$
d_1\left(m, m'\right):=\inf_{\gamma \in \Pi\left(m, m'\right)} \left[\int_{\mathbb{T}^d \times \mathbb{T}^d} d_{cc}(x,y) d\gamma(x,y)\right],
$$
where $\Pi\left(m, m'\right)$ is the set of Borel joint probability measures on $\mathbb{T}^d \times \mathbb{T}^d$ such that the marginal probability measures are respectively $m$ and $m'$.

When the probability measure $m$ is absolutely continuous with respect to the Lebesgue measure, we use the same letter $m$ to denote its density. Namely, we write $m: \mathbb{T}^d \ni x \mapsto m(x) \in \mathbb{R}_{+}$. In addition, we often consider flows of time dependent measures of the form $(m(t))_{t \in[0, T]}$, with $m(t) \in \mathcal{P}\left(\mathbb{T}^d\right)$ for any $t \in[0, T]$. For each time $t \in[0, T]$, if $m(t)$ is absolutely continuous with respect to the Lebesgue measure on $\mathbb{T}^d$, we can identify $m(t)$ with its density and we sometimes denote by $m:[0, T] \times \mathbb{T}^d \ni(t, x) \mapsto m(t, x) \in \mathbb{R}_{+}$ the collection of the densities.

We introduce the definition of the first order derivative of a function with respect to the measure (cf. \cite[Definition 2.2.1]{19CDLL}).
\begin{Def}\label{derivative w.r.t. m}
Suppose $U$ maps from $\mathcal{P}\left(\mathbb{T}^d\right)$ into $\mathbb{R}$. We say that $U$ is $\mathcal{C}^1$ if there exists a continuous map $\frac{\delta U}{\delta m}: \mathcal{P}\left(\mathbb{T}^d\right) \times \mathbb{T}^d \rightarrow \mathbb{R}$ such that, for any $m, m^{\prime} \in \mathcal{P}\left(\mathbb{T}^d\right)$,
$$
\lim _{s \rightarrow 0^{+}} \frac{U\left((1-s) m+s m^{\prime}\right)-U(m)}{s}=\int_{\mathbb{T}^d} \frac{\delta U}{\delta m}(m, y) d\left(m^{\prime}-m\right)(y) .
$$
\end{Def}
Note that $\frac{\delta U}{\delta m}$ is defined up to an additive constant. To ensure uniqueness, we add the normalization convention
$$
\int_{\mathbb{T}^d} \frac{\delta U}{\delta m}(m, y) d m(y)=0 .
$$
The integral form is that, for any $m, m^{\prime} \in \mathcal{P}\left(\mathbb{T}^d\right)$,
\begin{equation*}
U \left(m^{\prime}\right)-U(m) =\int_0^1 \int_{\mathbb{T}^d} \frac{\delta U}{\delta m}\left((1-s) m+s m^{\prime}, y\right) d\left(m^{\prime}-m\right)(y) d s .
\end{equation*}

Next we introduce an important class of weighted H\"{o}lder spaces associated to the family of vector fields $\mathcal{X}=\{X_1,\ldots,X_m\}$ (cf. \cite{07BB,10BBLU}).

Let $X^J:=X_{j_1} \cdots X_{j_m}$, where $J$ is any multi-index $J=\left(j_1, \ldots, j_m\right) \in \mathbb{Z}_{+}^m$ with the length $|J|=j_1+\cdots+j_m$, thus $X^J$ is a linear differential operator of order $|J|$. For $n \in \mathbb{N}$ and $\alpha \in(0,1)$ we define the weighted H\"{o}lder spaces
$$
\begin{aligned}
& C_{\mathcal{X}}^{\alpha}\left(\mathbb{T}^d\right):=\left\{\phi \in L^{\infty}\left(\mathbb{T}^d\right): \sup _{\substack{x, y \in \mathbb{T}^d \\
x \neq y}} \frac{|\phi(x)-\phi(y)|}{d_{cc}(x, y)^\alpha}<\infty\right\}, \\
& C_{\mathcal{X}}^{n+ \alpha}\left(\mathbb{T}^d\right):=\left\{\phi \in L^{\infty}\left(\mathbb{T}^d\right): X^J \phi \in C_{\mathcal{X}}^{ \alpha}\left(\mathbb{T}^d\right), \text{ any }|J| \leq n\right\}.
\end{aligned}
$$
For any function $\phi \in C_{\mathcal{X}}^{ \alpha}\left(\mathbb{T}^d\right)$, the H\"{o}lder seminorm can be defined as
$$
[\phi]_{C_{\mathcal{X}}^{ \alpha}\left(\mathbb{T}^d\right)}:=\sup _{\substack{x, y \in \mathbb{T}^d \\ x \neq y}} \frac{|\phi(x)-\phi(y)|}{d_{cc}(x, y)^\alpha}.
$$
Furthermore, for any $\phi \in C_{\mathcal{X}}^{n+ \alpha}\left(\mathbb{T}^d\right)$, the H\"{o}lder norm is defined as
\begin{equation}\label{Holder norm}
\|\phi\|_{C_{\mathcal{X}}^{n+ \alpha}\left(\mathbb{T}^d\right)}:=\|\phi\|_{C_{\mathcal{X}}^n\left(\mathbb{T}^d\right)}+\sum_{0 \leq|J| \leq n}\left[X^J \phi\right]_{C_{\mathcal{X}}^{ \alpha}\left(\mathbb{T}^d\right)},
\end{equation}
where $\|\phi\|_{C_{\mathcal{X}}^n\left(\mathbb{T}^d\right)}:=\sum\limits_{0 \leq |J|\leq n}\|X^J\phi\|_{L^{\infty}\left(\mathbb{T}^d\right)}$.

Endowed with the above norm, $C_{\mathcal{X}}^{n+ \alpha}\left(\mathbb{T}^d\right)$ is a Banach space and it follows form \eqref{d_cc VS Euclidean} that
$$
C^{-1}\|\phi\|_{C^{ \frac{\alpha}{k}}\left(\mathbb{T}^d\right)} \leq\|\phi\|_{C_{\mathcal{X}}^{\alpha}\left(\mathbb{T}^d\right)} \leq C\|\phi\|_{C^{ \alpha}\left(\mathbb{T}^d\right)},
$$
where $\|\phi\|_{C^{ \alpha}\left(\mathbb{T}^d\right)}$ is the standard H\"{o}lder norm, and $C>0$ is a constant depending only on the dimension $d$ and the family of vector fields $\mathcal{X}=\left\{X_1, \ldots, X_m\right\}$.

We can also define the parabolic Carnot-Carath\'{e}odory distance
$$
d_p\left((t,x),(s,y)\right):=\sqrt{d_{cc}(x,y)^2+|t-s|},
$$
which is a well-defined distance on $\mathbb{R} \times \mathbb{T}^d$. Replacing distance $d_{cc}$ with $d_p$, we can define the corresponding parabolic weighted H\"{o}lder spaces on $[0,T] \times \mathbb{T}^d$ as
$$
\begin{aligned}
& C_{\mathcal{X}}^{\frac{\alpha}{2},\alpha}\left([0,T]\times\mathbb{T}^d\right)\\
:=&\left\{\phi \in L^{\infty}\left([0,T]\times\mathbb{T}^d\right): \sup _{\substack{(t,x), (s,y) \in [0,T]\times\mathbb{T}^d \\
(t,x) \neq (s,y)}} \frac{|\phi(t,x)-\phi(s,y)|}{d_p\left((t,x), (s,y)\right)^\alpha}<\infty\right\}, \\
& C_{\mathcal{X}}^{\frac{n+\alpha}{2},n+\alpha}\left([0,T]\times\mathbb{T}^d\right)\\
:=&\left\{\phi \in L^{\infty}\left([0,T]\times\mathbb{T}^d\right): \partial_t^i X^J \phi \in C_{\mathcal{X}}^{\frac{\alpha}{2},\alpha}\left([0,T]\times\mathbb{T}^d\right), \text{ any }|J|+2i \leq n\right\}.
\end{aligned}
$$
with the seminorm
$$
[\phi]_{C_{\mathcal{X}}^{\frac{\alpha}{2},\alpha}\left([0,T]\times\mathbb{T}^d\right)}:=\sup _{\substack{(t,x), (s,y) \in [0,T]\times\mathbb{T}^d \\
(t,x) \neq (s,y)}} \frac{|\phi(t,x)-\phi(s,y)|}{d_p\left((t,x), (s,y)\right)^\alpha},
$$
and the norm
\begin{equation}\label{para. Holder norm}
\!\|\phi\|_{C_{\mathcal{X}}^{\frac{n+\alpha}{2},n+\alpha}\left([0,T]\times\mathbb{T}^d\right)}\!:=\!\sum_{0 \leq|J|+2i \leq n}\left(\!\|\partial_t^i X^J\phi\|_{L^{\infty}\!\left([0,T]\times\mathbb{T}^d\right)} \!+\!\left[\partial_t^i X^J \phi\right]_{C_{\mathcal{X}}^{\frac{\alpha}{2},\alpha}\left([0,T]\times\mathbb{T}^d\right)}\!\right)\!.
\end{equation}
To simplify the notations, we can abbreviate the H\"{o}lder norms in \eqref{Holder norm} and \eqref{para. Holder norm} respectively as $\|\cdot\|_{n+\alpha}$ and $\|\cdot\|_{\frac{n+\alpha}{2},n+\alpha}$.

If a smooth map $\phi$ depends on two space variables, for example, $\phi=\phi(x, y)$, and $m, n \in \mathbb{N}$ are the order of derivation of $\phi$ with respect to $x$ and $y$ respectively, we define
$$
\|\phi\|_{(m, n)}:=\sum_{|J| \leqslant m,\left|J'\right| \leqslant n}\left\|X^{\left(J,J'\right)} \phi\right\|_{\infty}.
$$
Moreover, if the derivatives are $C_{\mathcal{X}}^{\alpha}$ continuous,
$$
\|\phi\|_{(m+\alpha, n+\alpha)}\! :=\! \|\phi\|_{(m, n)}\!+\! \sum_{|J|\leq m,\left|J'\right|\leq n}\! \sup _{(x, y) \neq\left(x^{\prime}, y^{\prime}\right)}\! \frac{\left|X^{\left(J,J'\right)} \phi(x, y)-X^{\left(J,J'\right)} \phi\left(x^{\prime}, y^{\prime}\right)\right|}{d_{cc}\left(x,x'\right)^\alpha +d_{cc}\left(y,y'\right)^\alpha}\!.
$$

Analogously, let $k$ be a non-negative integer and $1\leq p \leq \infty$, we denote by $W_{\mathcal{X}}^{k,p}(\mathbb{T}^d)$ the weighted Sobolev space associated to the family of vector fields $\mathcal{X}=\{X_1,\ldots,X_m\}$ such that
\begin{equation*}
W_{\mathcal{X}}^{k,p}\left(\mathbb{T}^d\right):=\left\{\phi \in L^{p}\left(\mathbb{T}^d\right): X^J \phi \in L^{p}\left(\mathbb{T}^d\right), \text{ any }|J| \leq k\right\}.
\end{equation*}
Endowed with the above norm, $W_{\mathcal{X}}^{k,p}\left(\mathbb{T}^d\right)$ is a Banach space.

The dual space of $C_{\mathcal{X}}^{n+\alpha}$ is denoted by $C_{\mathcal{X}}^{-(n+\alpha)}$ with norm
$$
\|\psi\|_{-(n+\alpha)}:=\sup _{\|\phi\|_{n+\alpha} \leqslant 1}\langle\psi, \phi\rangle_{C_{\mathcal{X}}^{-(n+\alpha)}, C_{\mathcal{X}}^{n+\alpha}}, \text{ for any } \psi \in C_{\mathcal{X}}^{-(n+\alpha)}.
$$
The same for the dual spaces of $C_{\mathcal{X}}^{\frac{n+\alpha}{2},n+\alpha}$ and $W_{\mathcal{X}}^{k,\infty}$, denoting respectively as $C_{\mathcal{X}}^{-\frac{n+\alpha}{2},-(n+\alpha)}$ and $W_{\mathcal{X}}^{-k,\infty}$.

We need the following hypotheses:
\begin{enumerate}[H1)]
\item $F: \mathbb{T}^2 \times \mathcal{P}(\mathbb{T}^2) \rightarrow \mathbb{R}$ satisfies, for some $\alpha \in (0,1)$ and $C_F>0$,
\begin{equation}\label{monotonicity_F}
\int_{\mathbb{T}^2}\left(F(x, m)-F\left(x, m^{\prime}\right)\right) d\left(m-m^{\prime}\right)(x) \geq 0,
\end{equation}
\begin{equation}\label{property_F}
\int_{\mathbb{T}^2}\int_{\mathbb{T}^2} \frac{\delta F}{\delta m}(x,m,y)\rho(x)\rho(y)dxdy \geq 0
\end{equation}
for any $\rho \in C_{\mathcal{X}}^{-(1+\alpha)}\left(\mathbb{T}^2\right)$ and any $m \in \mathcal{P}\left(\mathbb{T}^2\right)$. And we have
$$
\begin{gathered}
\sup _{m \in \mathcal{P}(\mathbb{T}^2)}\left(\|F(\cdot, m)\|_{1+\alpha}+\left\|\frac{\delta F}{\delta m}(\cdot, m, \cdot)\right\|_{(1+\alpha,2+\alpha)}\right)+\operatorname{Lip}\left(\frac{\delta F}{\delta m}\right) \leq C_F, \\
\operatorname{Lip}\left(\frac{\delta F}{\delta m}\right)\!:=\!\sup _{\substack{m_1,m_2 \in \mathcal{P}(\mathbb{T}^2)\\ m_1 \neq m_2}}\left(d_1\left(m_1, m_2\right)^{-1}\left\|\frac{\delta F}{\delta m}\!\left(\!\cdot, m_1, \cdot\right)-\frac{\delta F}{\delta m}\left(\cdot, m_2, \cdot\right)\right\|_{(1+\alpha, 2+\alpha)}\!\right)\!;
\end{gathered}
$$\label{hyp1}
\item $G: \mathbb{T}^2 \times \mathcal{P}(\mathbb{T}^2) \rightarrow \mathbb{R}$ satisfies similar estimates as $F$ with $1+\alpha$ replaced by $2+\alpha$, namely,
\begin{equation}\label{monotonicity_G}
\int_{\mathbb{T}^2}\left(G(x, m)-G\left(x, m^{\prime}\right)\right) d\left(m-m^{\prime}\right)(x) \geq 0,
\end{equation}
\begin{equation}\label{property_G}
\int_{\mathbb{T}^2}\int_{\mathbb{T}^2} \frac{\delta G}{\delta m}(x,m,y)\rho(x)\rho(y)dxdy \geq 0
\end{equation}
for any $\rho \in C_{\mathcal{X}}^{-(1+\alpha)}\left(\mathbb{T}^2\right)$ and any $m \in \mathcal{P}\left(\mathbb{T}^2\right)$.
$$
\sup _{m \in \mathcal{P}(\mathbb{T}^2)}\left(\|G(\cdot, m)\|_{2+\alpha}+\left\|\frac{\delta G}{\delta m}(\cdot, m, \cdot)\right\|_{(2+\alpha, 2+\alpha)}\right)+\operatorname{Lip}\left(\frac{\delta G}{\delta m}\right) \leq C_G,
$$
$$
\operatorname{Lip}\left(\frac{\delta G}{\delta m}\right)\!:=\!\sup _{\substack{m_1,m_2 \in \mathcal{P}(\mathbb{T}^2)\\ m_1 \neq m_2}}\!\left(\!d_1\left(m_1, m_2\right)^{-1}\left\|\frac{\delta G}{\delta m}\left(\cdot, m_1, \cdot\right)-\frac{\delta G}{\delta m}\left(\cdot, m_2, \cdot\right)\right\|_{(2+\alpha, 2+\alpha)}\!\right)\!.
$$\label{hyp2}
\end{enumerate}
\begin{Rem}
The hypotheses \eqref{property_F} and \eqref{property_G} are respectively stronger than the monotonicity conditions \eqref{monotonicity_F} and \eqref{monotonicity_G}. In fact, the latter cannot imply the former, as we can find a counter-example in \cite[Remark 2.23]{23JR}. We can note by \cite[Remark 2.3(i)]{22GMMZ} that \eqref{monotonicity_F} implies property \eqref{property_F} with $D_{x}\frac{\delta F}{\delta m}$ instead of $\frac{\delta F}{\delta m}$.
\end{Rem}

Before giving the main results, let us first state the concept of the solution to the master equation.
\begin{Def}\label{Def_ME sol.}
We say that a map $U:[0, T] \times \mathbb{T}^{2} \times \mathcal{P}\left(\mathbb{T}^{2}\right) \rightarrow \mathbb{R}$ is a solution to the first-order master equation \eqref{ME} if
\begin{enumerate}
\item $U$ is continuous in all its arguments (for the $d_1$ distance on $\mathcal{P}\left(\mathbb{T}^{2}\right)$), and is of class $C_{\mathcal{X}}^{2}$ in $x$ and $C^1$ in $t$ (the corresponding derivatives are continuous in all the arguments);
\item $U$ is of class $C^1$ with respect to $m$, the first-order derivative
    $$
    [0, T] \times \mathbb{T}^{2} \times \mathcal{P}\left(\mathbb{T}^{2}\right) \times \mathbb{T}^{2} \ni (t, x, m, y) \mapsto \frac{\delta U}{\delta m}(t, x, m, y),
    $$
    is continuous in all the arguments. Also $\frac{\delta U}{\delta m}$ is $C_{\mathcal{X}}^2$ in $y$, with the derivatives being continuous in all the arguments;
\item U satisfies the master equation \eqref{ME}.
\end{enumerate}
\end{Def}
With the previous hypotheses we would like to investigate the existence and uniqueness of the solution to the master equation \eqref{ME}.
\begin{Thm}[wellposedness of the master equation]\label{Thm_ME wellposed regularity}
Suppose hypotheses H\ref{hyp1}) and H\ref{hyp2}) are satisfied. Then there exists a unique solution U of the master equation \eqref{ME} in the sense of Definition \ref{Def_ME sol.}.

Moreover, the derivative $\frac{\delta U}{\delta m}$ satisfies
\begin{equation*}
\sup_{(t,m) \in [0,T] \times \mathcal{P}(\mathbb{T}^2)}\left\|\frac{\delta U}{\delta m}\left(t, \cdot, m, \cdot\right)\right\|_{(2+\alpha,2+\alpha)} \leqslant C
\end{equation*}
and is Lipschitz continuous with respect to $m$, namely
\begin{equation*}
\sup_{t \in [0,T]}\sup_{m_1 \neq m_2}\left(d_{1}\left(m_1,m_2\right)\right)^{-1}\left\|\frac{\delta U}{\delta m}\left(t,\cdot,m_1,\cdot\right)-\frac{\delta U}{\delta m}\left(t, \cdot, m_2,\cdot\right)\right\|_{2+\alpha,2+\alpha} < \infty.
\end{equation*}
\end{Thm}
The main idea of the proof stems from \cite{19CDLL}, which is quite classical. More precisely, the uniqueness is obtained by proving that with the solution of the master equation \eqref{ME} we can construct a solution of the MFG system \eqref{MFG system}, and then by uniqueness of the system. Conversely, we consider the unique solution $(u,m) \in C_{\mathcal{X}}^{1,2}\left([t_0,T]\times\mathbb{T}^2\right) \times C\left([t_0,T];\mathcal{P}(\mathbb{T}^2)\right)$ to the MFG system \eqref{MFG system} with initial condition $m(t_0)=m_0$ for any $\left(t_0,m_0\right)\in[0,T]\times\mathcal{P}\left(\mathbb{T}^2\right)$, then we define
\begin{equation*}
U\left(t_0, x, m_0\right):=u\left(t_0, x\right).
\end{equation*}
The existence can be obtained by proving that $U$ is a solution of the master equation \eqref{ME}. In order to do this, we need to prove some preliminary regularities of $U$ (Proposition \ref{Prop_MFG system wellposed regularity}), obtained by an insight into properties of the solution to the MFG system \eqref{MFG system}, which are uniformly in $(t_0,m_{0})$. In addition, we also need to prove the Lipschitz property, $C^1$ differentiability of $U$ and moreover the Lipschitz continuity of the derivative $\frac{\delta U}{\delta m}$ with respect to the measure, these can be seen respectively in Proposition \ref{Prop_Lip. ctn. of U}, Proposition \ref{Prop_U C^1 w.r.t. m} and Proposition \ref{Prop_derivative Lip. w.r.t. m}. In the following way, we give the statements and the idea of proof for each proposition.
\begin{Prop}[Space regularity and continuity of $U$]\label{Prop_MFG system wellposed regularity}
Assume that hypotheses H\ref{hyp1}) and H\ref{hyp2}) hold. Then, for any initial condition $\left(t_0, m_0\right) \in[0, T] \times \mathcal{P}\left(\mathbb{T}^2\right)$, the MFG system (\ref{MFG system}) has a unique solution $(u, m) \in C_{\mathcal{X}}^{1,2}\left([t_0,T]\times\mathbb{T}^2\right) \times C\left([t_0,T];\mathcal{P}(\mathbb{T}^2)\right)$ and this solution satisfies
$$
\sup _{t_1 \neq t_2} \frac{d_1\left(m\left(t_1\right), m\left(t_2\right)\right)}{\left|t_1-t_2\right|^{\frac{1}{2}}}+\left\|u\right\|_{1+\frac{\alpha}{2},2+\alpha} \leqslant C,
$$
where the constant $C$ does not depend on $\left(t_0, m_0\right)$. Moreover, if $m_0$ is absolutely continuous with a smooth positive density, then $m$ is of class $C_{\mathcal{X}}^{1+\frac{\alpha}{2},2+\alpha}\left([t_0,T]\times\mathbb{T}^2\right)$ with a smooth positive density.

Furthermore, the solution is stable: if $m_0^n \rightarrow m_0$ in $\mathcal{P}\left(\mathbb{T}^2\right)$, then the corresponding solutions of system \eqref{MFG system} converge in the sense that $\left(u^n,m^n\right) \rightarrow (u,m) \in C_{\mathcal{X}}^{1+\frac{\alpha}{2},2+\alpha}\left([t_0,T]\times\mathbb{T}^2\right) \times C\left([t_0,T];\mathcal{P}(\mathbb{T}^2)\right)$.

Thus we can get the space regularity of $U$, that is
$$
\sup _{t_0 \in[0, T]} \sup _{m_0 \in \mathcal{P}\left(\mathbb{T}^d\right)}\|U(t_0, \cdot, m_0)\|_{2+\alpha} \leqslant C.
$$
In addition to this, the continuity of $U$ with respect to the time (for a more rigorous argument see \cite[Lemma D.1]{23JR}) and the measure can be deduced from the stability of the solution to the MFG system \eqref{MFG system}.
\end{Prop}
To prove Proposition \ref{Prop_MFG system wellposed regularity}, we apply Schauder fixed point theorem and note that when $m_0$ is sufficiently smooth, the proof of the wellposedness of this degenerate MFG system can be found in \cite{23JRWX}. Therefore, the key is to prove the existence and uniqueness of the weak solution to the degenerate KFP equation in the MFG system \eqref{MFG system} when $m_0 \in \mathcal{P}(\mathbb{T}^2)$. In this paper, we prove a more general result in Lemma \ref{Lem_general KFP wellposed regularity}. That is, we study a more general case and consider the KFP equation with the following general form:
\begin{equation}\label{general KFP}
\begin{cases}
\partial_t\rho-\Delta_\mathcal{X}\rho-\operatorname{div}_{\mathcal{X}}(\rho b)=f, & \text{ in } [0,T] \times \mathbb{T}^2, \\
\rho(0)=\rho_{0}, & \text{ in } \mathbb{T}^2.
\end{cases}
\end{equation}
We refer to the idea of duality in \cite{21Ri} to state a suitable definition of the distributional solution.
\begin{Def}\label{Def_general KFP weak sol.}
Let $b \in C_\mathcal{X}^{\frac{\alpha}{2},\alpha}([0,T] \times \mathbb{T}^2)$, $f \in L^1\left([0,T];W_\mathcal{X}^{-1,\infty}(\mathbb{T}^2)\right)$ and $\rho_{0} \in C_\mathcal{X}^{-(1+\alpha)}(\mathbb{T}^2)$. A function $\rho \in C\left([0,T];C_\mathcal{X}^{-(1+\alpha)}(\mathbb{T}^2)\right)$ is said to be a weak solution to equation \eqref{general KFP} if, for all $\xi \in C\left([0,t];C_\mathcal{X}^{1+\alpha}(\mathbb{T}^2)\right)$, $\psi \in C_\mathcal{X}^{1+\alpha}(\mathbb{T}^2)$ and $\phi$ solution in $[0,t] \times \mathbb{T}^2$ of the linear equation as follows
\begin{equation}\label{dual KFP}
\begin{cases}
-\partial_t\phi-\Delta_\mathcal{X}\phi+b D_{\mathcal{X}}\phi=\xi, & \text{ in } [0,t) \times \mathbb{T}^2, \\
\phi(t)=\psi, & \text{ in } \mathbb{T}^2,
\end{cases}
\end{equation}
the following weak formulation holds true:
$$
\langle\rho(t),\psi\rangle+\int_{0}^{t} \langle \rho(s),\xi(s,\cdot) \rangle ds=\left\langle\rho_0,\phi(0,\cdot)\right\rangle+\int_{0}^{t}\langle f(s),\phi(s,\cdot)\rangle ds,
$$
where $\langle\cdot,\cdot\rangle$ denotes the duality between $\mathcal{C}_{\mathcal{X}}^{-(1+\alpha)}$ and $\mathcal{C}_{\mathcal{X}}^{1+\alpha}$ in the first to third case or between $W_{\mathcal{X}}^{-1,\infty}$ and $W_{\mathcal{X}}^{1,\infty}$ in the last case.
\end{Def}
The existence, uniqueness and regularities of the weak solution to the KFP equation are given as follows.
\begin{Lem}\label{Lem_general KFP wellposed regularity}
Let $b \in C_{\mathcal{X}}^{\frac{\alpha}{2},\alpha}([0,T] \times \mathbb{T}^2)$, $f \in L^1\left([0,T];W_{\mathcal{X}}^{-1,\infty}(\mathbb{T}^2)\right)$ and $\rho_0 \in C_{\mathcal{X}}^{-1}(\mathbb{T}^2)$. Then there exists a unique weak solution $\rho$ in the sense of Definition \ref{Def_general KFP weak sol.} to the KFP equation \eqref{general KFP}.

This solution satisfies
\begin{equation}\label{general KFP sol. regularity}
\sup_{t \in [0,T]}\left\|\rho(t)\right\|_{-(1+\alpha)} + \left\|\rho\right\|_{-\frac{\alpha}{2},-\alpha} \leq C\left(\left\|\rho_0\right\|_{-(1+\alpha)} + \left\|f\right\|_{L^1\left([0,T];W_{\mathcal{X}}^{-1,\infty}(\mathbb{T}^2)\right)}\right).
\end{equation}

Finally, the solution is stable: if $\rho_{0}^{k} \rightarrow \rho_{0}$ in $C_{\mathcal{X}}^{-(1+\alpha)}(\mathbb{T}^2)$, $\left\{b^{k}\right\}_{k}$ is uniformly bounded and $b^{k} \rightarrow b$ in $C_{\mathcal{X}}^{\frac{\alpha}{2},\alpha}([0,T] \times \mathbb{T}^2)$, $f^k \rightarrow f$ in $L^1\left([0,T];W_{\mathcal{X}}^{-1,\infty}(\mathbb{T}^2)\right)$, then, calling $\rho^k$ and $\rho$ the solution related, respectively, to $\left(\rho_{0}^{k},b^{k},f^{k}\right)$ and $\left(\rho_{0},b,f\right)$, we have $\rho^{k} \rightarrow \rho$ in $C\left([0,T];C_{\mathcal{X}}^{-(1+\alpha)}(\mathbb{T}^2)\right)$.
\end{Lem}
\begin{Prop}[Lipschitz continuity of $U$ w.r.t. the measure]\label{Prop_Lip. ctn. of U}
Assume that hypotheses H\ref{hyp1}) and H\ref{hyp2}) hold. Let $t_0 \in[0, T]$, $m_0^1$, $m_0^2 \in \mathcal{P}\left(\mathbb{T}^2\right)$, and $\left(u^1, m^1\right)$, $\left(u^2, m^2\right)$ be the solutions to the MFG system \eqref{MFG system} with initial condition $\left(t_0, m_0^1\right)$ and $\left(t_0, m_0^2\right)$ respectively. Then
$$
\sup _{t \in[0, T]}\left\{d_1\left(m^1(t), m^2(t)\right)+\left\|u^1(t, \cdot)-u^2(t, \cdot)\right\|_{2+\alpha}\right\} \leqslant C d_1\left(m_0^1, m_0^2\right),
$$
for a constant $C$ independent of $t_0, m_0^1$ and $m_0^2$. In particular,
$$
\left\|U\left(t_0, \cdot, m_0^1\right)-U\left(t_0, \cdot, m_0^2\right)\right\|_{2+\alpha} \leqslant C d_1\left(m_0^1, m_0^2\right).
$$
\end{Prop}
To prove Proposition \ref{Prop_Lip. ctn. of U}, we note that the estimates of $u^1-u^2$ and $d_1\left(m^1, m^2\right)$ need the results respectively in Lemma \ref{Lem_LHJB wellposed regularity} and Lemma \ref{Lem_homo. LHJB Lipschitz}, both of which rely on the study of the linearized HJB equation. Namely, we consider a linear degenerate backward parabolic equation as follows (see also in \eqref{dual KFP})
\begin{equation}\label{LHJB}
\begin{cases}
-\partial_t z-\Delta_{\mathcal{X}} z+V(t, x) \cdot D_{\mathcal{X}} z=f(t, x),& \text {in }[0, T) \times \mathbb{T}^2, \\
z(T, x)=z_T(x),& \text {in }\mathbb{T}^2.
\end{cases}
\end{equation}
We use the ``lifting and approximation" technique and the theory of singular integrals or fractional integrals to prove one of the most crucial lemmas in this paper.
\begin{Lem}\label{Lem_LHJB wellposed regularity}
Let $V \in C_{\mathcal{X}}^{\frac{\alpha}{2},1+\alpha}\left([0, T] \times \mathbb{T}^{2} \right)$ and $f \in C \left([0, T]; C_{\mathcal{X}}^{1+\alpha}\left(\mathbb{T}^2\right)\right)$. Then, for any $z_T \in C_{\mathcal{X}}^{2+\alpha}(\mathbb{T}^2)$, equation \eqref{LHJB} has a unique solution $z$ which belongs to $C_{\mathcal{X}}^{1,2} \left([0, T) \times \mathbb{T}^{2} \right) \bigcap C([0, T] \times \mathbb{T}^2)$ and satisfies
$$
\sup _{t \in[0, T]}\|z(t, \cdot)\|_{2+\alpha} \leqslant C\left\{\left\|z_T\right\|_{2+\alpha}+\sup _{t \in[0, T]}\|f(t, \cdot)\|_{1+\alpha}\right\},
$$
where $C$ depends on $\sup\limits _{t \in[0, T]}\|V(t, \cdot)\|_{1+\alpha}$ and $\alpha$ only.

Furthermore, for any constant $T' \in (0,T)$, $z$ satisfies
$$
\sup _{\substack{t \neq t'\\t,t' \in [0,T']}} \frac{\left\|z\left(t^{\prime}, \cdot\right)-z(t, \cdot)\right\|_{2+\alpha}}{\left|t^{\prime}-t\right|^\beta} \leqslant C\left\{\left\|z_T\right\|_{2+\alpha}+\sup _{t \in[0, T]}\|f(t, \cdot)\|_{1+\alpha}\right\},
$$
where $\beta \in (0,\frac{1}{2})$, and $C$ depends on $T'$, $\sup\limits_{t \in[0, T]}\|V(t, \cdot)\|_{1+\alpha}$ and $\alpha$ only.
\end{Lem}
In addition, when the data $f=0$, the space regularity result of Lemma \ref{Lem_LHJB wellposed regularity} can be generalized if the terminal condition is only a Lipschitz function.
\begin{Lem}\label{Lem_homo. LHJB Lipschitz}
Suppose $f(t,x)=0$, $V(t,x) \in C^{\frac{\alpha}{2},\alpha}_{\mathcal{X}}\left([0,T]\times\mathbb{T}^2\right)$, and $z_T$ is $d_{cc}$-Lipschitz with Lipschitz constant bounded by $1$. Then the unique solution $z$ of the equation \eqref{homo. LHJB} satisfies a $d_{cc}$-Lipschitz condition, uniformly in $t$, namely there exists $C>0$, independent of $t$, such that
$$
|z(t,x)-z(t,y)| \leq C d_{cc}(x,y).
$$
\end{Lem}
\begin{Prop}[$C^1$ differentiability of $U$ w.r.t. the measure]\label{Prop_U C^1 w.r.t. m}
Assume that H\ref{hyp1}) and H\ref{hyp2}) hold. Fix $t_{0} \in [0, T]$ and $m_{0}, \hat{m}_{0} \in \mathcal{P}\left(\mathbb{T}^{2}\right)$, let $(u,m)$ and $(\hat{u}, \hat{m})$ be the solutions of the MFG system \eqref{MFG system} beginning from $\left(t_{0}, m_{0}\right)$ and $\left(t_{0}, \hat{m}_{0}\right)$ respectively and let $(z,\rho)$ be the solution to the system \eqref{linear MFG system} with initial condition $\left(t_{0}, \hat{m}_{0}-m_{0}\right)$. Then
\begin{align*}
\sup _{t \in\left[t_{0}, T\right]} & \left(\|\hat{u}(t, \cdot)-u(t, \cdot)-z(t,\cdot)\|_{2+\alpha}+\left\|\hat{m}(t, \cdot)-m(t, \cdot)\right. -\rho(t, \cdot) \|_{-(1+\alpha)}\right) \\
\leqslant & C d_{1}^{2}\left(m_{0}, \hat{m}_{0}\right).
\end{align*}
Finally, it is immediate for one to obtain the $C^1$ differentiability of $U$ with respect to $m$, namely
\begin{equation*}
\frac{\delta U}{\delta m}\left(t_0, x, m_0, y\right)=K\left(t_0, x, m_0, y\right),
\end{equation*}
satisfying
\begin{equation*}
\sup_{(t_0,m_0) \in [0,T] \times \mathcal{P}(\mathbb{T}^2)}\left\|\frac{\delta U}{\delta m}\left(t_0, \cdot, m_0, \cdot\right)\right\|_{(2+\alpha,2+\alpha)} \leqslant C.
\end{equation*}
Moreover,
\begin{align*}
& \left\|U\left(t_0, \cdot, \hat{m}_0\right)-U\left(t_0, \cdot, m_0\right)-\int_{\mathbb{T}^{2}} \frac{\delta U}{\delta m}\left(t_0, \cdot, m_0, y\right) d\left(\hat{m}_0-m_0\right)(y)\right\|_{2+\alpha}\\
\leqslant & C d_{1}^{2}\left(m_0, \hat{m}_0\right).
\end{align*}
\end{Prop}
\begin{Prop}[Lipschitz continuity of $\frac{\delta U}{\delta m}$ w.r.t. the measure]\label{Prop_derivative Lip. w.r.t. m}
Under the assumptions of H\ref{hyp1}) and H\ref{hyp2}), then
\begin{equation*}
\sup _{t \in[0, T]} \sup _{m_{1} \neq m_{2}}\left(d_{1}\left(m_{1}, m_{2}\right)\right)^{-1}\left\|\frac{\delta U}{\delta m}\left(t, \cdot, m_{1}, \cdot\right)-\frac{\delta U}{\delta m}\left(t, \cdot, m_{2}, \cdot\right)\right\|_{(2+\alpha,1+\alpha)} \leqslant C
\end{equation*}
where $C$ depends on $F$, $G$, $H$ and $T$.
\end{Prop}
To prove Proposition \ref{Prop_U C^1 w.r.t. m} and Proposition \ref{Prop_derivative Lip. w.r.t. m}, we shall construct the $C^1$ derivative of $U$ with respect to the measure in Lemma \ref{Lem_relation of z and rho_0}. To this aim, we differentiate the MFG system \eqref{MFG system} to get the linearized MFG system, which is a coupled system of a degenerate linear backward equation and a KFP equation as follows
\begin{equation}\label{linear MFG system}
\begin{cases}
-\partial_t z-\Delta_{\mathcal{X}} z+D_pH(x,D_{\mathcal{X}}u) \cdot D_{\mathcal{X}} z=\frac{\delta F}{\delta m}(x,m(t))(\rho(t)),& \text {in }[t_0, T] \times \mathbb{T}^2, \\
\partial_t\rho-\Delta_\mathcal{X}\rho-\operatorname{div}_{\mathcal{X}}(\rho D_pH(x,D_{\mathcal{X}}u))\\
\quad -\operatorname{div}_{\mathcal{X}}(m D_{pp}^2H(x,D_{\mathcal{X}}u) D_{\mathcal{X}}z)=0,& \text {in }[t_0, T] \times \mathbb{T}^2, \\
z(T, x)=\frac{\delta G}{\delta m}(x,m(T))(\rho(T)),\quad \rho(t_0)=\rho_0,& \text {in }\mathbb{T}^2,
\end{cases}
\end{equation}
where $(u,m)$ is the solution to the MFG system \eqref{MFG system} with initial condition $m(t_0)=m_0$ for any fixed $(t_0,m_0) \in [0,T] \times \mathcal{P}(\mathbb{T}^2)$, and $\rho_0$ can be supposed in a suitable space.

We aim at proving that $U$ is of class $C^1$ with respect to $m$ satisfying
\begin{equation*}
z(t_0,x)=\int_{\mathbb{T}^2}\frac{\delta U}{\delta m}(t_0,x,m_0,y)d \rho_0(y)=:\frac{\delta U}{\delta m}(t_0,x,m_0)(\rho_0).
\end{equation*}
That is the following lemma.
\begin{Lem}\label{Lem_relation of z and rho_0}
Under the assumptions of H\ref{hyp1}) and H\ref{hyp2}), there exists, for any  $\left(t_{0}, m_{0}\right) \in [0,T]\times\mathcal{P}\left(\mathbb{T}^2\right)$, a $C_{\mathcal{X}}^{2+\alpha}(\mathbb{T}^2) \times C_{\mathcal{X}}^{2+\alpha}(\mathbb{T}^2)$ map $(x, y) \mapsto K\left(t_{0}, x, m_{0}, y\right)$ such that, for any $\rho_{0} \in C_{\mathcal{X}}^{-(1+\alpha)}(\mathbb{T}^2)$, the $z$ component of the solution of system \eqref{linear MFG system} is given by
\begin{equation}\label{linear MFG system_rep.formula}
z\left(t_{0}, x\right)=\int_{\mathbb{T}^2}K\left(t_{0}, x, m_{0}, y\right)d\rho_{0}(y)=\left\langle\rho_{0}, K\left(t_{0}, x, m_{0}, \cdot\right)\right\rangle_{2+\alpha}.
\end{equation}
Moreover
\begin{equation*}
\left\|K\left(t_{0}, \cdot, m_{0}, \cdot\right)\right\|_{(2+\alpha,2+\alpha)} \leqslant C,
\end{equation*}
also $K$ and its derivatives in $(x,y)$ are continuous on $[0,T] \times \mathbb{T}^{2} \times \mathcal{P}\left(\mathbb{T}^{2}\right) \times \mathbb{T}^{2}$.
\end{Lem}

The contribution of this paper is mainly of three points. First, the master equation studied in this paper is a kind of degenerate PDE stated on the space of probability measures. It is related to a second-order MFG with the Grushin type diffusion, where the generic player may have a ``forbidden'' direction on a vertical line. Second, we prove the existence and uniqueness of the classical solution in the scale of weighted H\"{o}lder spaces for the master equation \eqref{ME}, which can describe the Nash equilibria in the MFG. As a byproduct, we rigorously illustrate the equivalent characterization of this master equation with respect to the MFG system \eqref{MFG system}. Third, due to the degenerate nature of the equations, new problems arise that need to be solved. In particular, in order to obtain the Lipschitz continuity of $U$, we study the linearized HJB equation \eqref{LHJB}, which is a linear degenerate parabolic equation satisfying the H\"{o}rmander condition. Owing to the lack of H\"{o}lder continuity of the non-homogeneous term with respect to the variable $t$, we obtain some Schauder estimates (see Lemma \ref{Lem_LHJB wellposed regularity}) in the scale of weighted H\"{o}lder spaces, which is completely new in the literature. As another application, by duality, we prove the existence and uniqueness of the weak solution to the degenerate KFP equation \eqref{general KFP} and obtain the regularities of the solution.

The rest of the paper is organized as follows. In Section \ref{Sec_3}, we devote to the linearized HJB equation \eqref{LHJB}, proving Lemma \ref{Lem_LHJB wellposed regularity} and Lemma \ref{Lem_homo. LHJB Lipschitz}. Before that, we first give some preliminaries and known results concerning the sub-Laplacian with drift in Heisenberg group and the theory of singular integrals or fractional integrals. Lastly, we obtain the existence, uniqueness and regularities of the weak solution to the KFP equation \eqref{general KFP} in Lemma \ref{Lem_general KFP wellposed regularity}. In Section \ref{Sec_4}, we give the proof of Proposition \ref{Prop_MFG system wellposed regularity} and then we prove the Lipschitz continuity of $U$ with respect to the measure in Proposition \ref{Prop_Lip. ctn. of U}. In Section \ref{Sec_5}, we can prove Lemma \ref{Lem_relation of z and rho_0} which regards the linearized MFG system \eqref{linear MFG system}, once we have proved the relevant results for linearized HJB equation \eqref{LHJB} and KFP equation \eqref{general KFP}. Next we prove the $C^1$ character of $U$ with respect to the measure in Proposition \ref{Prop_U C^1 w.r.t. m} and Proposition \ref{Prop_derivative Lip. w.r.t. m}. Finally, in Section \ref{Sec_6}, we complete the proof of Theorem \ref{Thm_ME wellposed regularity}.

\section{Technical results for two types of linear degenerate equations}\label{Sec_3}
\noindent
In this section, we will give some technical results for linear equation \eqref{LHJB} and KFP equation \eqref{general KFP}. These results play a key role in the proof of the subsequent propositions.

Let us start with some preliminaries and known results.
\subsection{Fundamental solution for sub-Laplacian with drift in Heisenberg group and Gaussian estimate}\label{preliminary1}

Let $\mathbb{H}^n=\mathbb{T}^n \times \mathbb{T}^n \times \mathbb{T}$ be the Heisenberg group of $2n+1$ real dimensions on the torus, with points written $\xi=(x,y,z)$. The group law is
$$(x,y,z) \circ (x',y',z')=(x+x',y+y',z+z'+2\sum_{j=1}^{n}(y_j {x_j}'-x_j {y_j}')),
$$
and it can be proved that the lebesgue measure in $\mathbb{T}^{2n+1}$ is the Haar measure of $\mathbb{H}^n$.

The vector fields
$$
Y_{1,j}=\partial_{x_j}+2 y_j \partial_{z},\quad Y_{2,j}=\partial_{y_j}-2 x_j \partial_{z}, \quad 1 \leq j \leq n,
$$
are homogeneous and left invariant on $\mathbb{H}^n$ and generate its Lie algebra with homogeneous dimension of $Q=2n+2$. The corresponding sub-Laplacian is
$$
\Delta_{\mathcal{Y}}=\sum_{j=1}^n (Y_{1,j}^2+Y_{2,j}^2),
$$
Further, $D_{\mathcal{Y}}=(Y_1,Y_2)=(Y_{1,1},\ldots,Y_{1,n},Y_{2,1},\ldots,Y_{2,n})$ is the horizontal gradient.

We define the sub-Laplacian operator with a drift $v:=(a,b,0) \in \mathbb{R}^n \times \mathbb{R}^n \times \mathbb{R}$ as follows
$$
\mathcal{L}_v:=\Delta_{\mathcal{Y}}-\sum_{j=1}^n (a_j Y_{1,j}+b_j Y_{2,j}).
$$
Suppose the nonzero vector $v$ is related to time but not space, that is to say $v=(a(t),b(t),0)$ for $t \in \mathbb{R}$, and then consider the corresponding parabolic operator $\mathcal{H}_v=-\partial_t-\mathcal{L}_v$, which is a space-left invariant, H\"{o}rmander's operator on $\mathbb{R} \times \mathbb{H}^n$. We can now state the main results used in this paper.

\begin{Lem}[Fundamental solution for $\mathcal{H}_v$]\label{FS}
  There exists a global fundamental solution $\Gamma_v(t,s,\xi,\eta)$ for $\mathcal{H}_v$ in $\mathbb{R} \times \mathbb{H}^n$, with the properties listed below.
  \begin{enumerate}
  \item $\Gamma_v \geq 0$ and vanishes for $t \geq s$, or else for $t < s$, it can be explicitly given by
      \begin{equation*}
        \Gamma_v(t,s,\xi,\eta)=\exp(-\frac{1}{4}\int_{t}^{s}|v(\tau)|^2 d \tau)\chi(t,(\eta^{-1} \circ \xi)^{-1})\Gamma(s-t,\eta^{-1} \circ \xi),
      \end{equation*}
      where $\chi$ is a homomorphism from $\mathbb{R} \times \mathbb{H}^n$ to the multiplicative group $\mathbb{R}_+$ (that is, $\chi(t,\xi \circ \eta)=\chi(t,\xi)\chi(t,\eta),\xi,\eta \in \mathbb{H}^n, t \in \mathbb{R}$) defined as
      $$
      \chi(t,\xi):=\exp[\frac{1}{2} \sum_{j=1}^{n}(a_j(t) x_j+b_j(t) y_j)],
      $$
      and $\Gamma$ is the fundamental solution for the operator $\mathcal{H}$ without drift, namely when $v \equiv 0$.
  \item $\Gamma_v$ is smooth with respect to the space variable $(\xi,\eta) \in \mathbb{H}^n \times \mathbb{H}^n$ for any $(t,s) \in \mathbb{R}\times\mathbb{R}$ such that $t \neq s$.
  \item For every $(s,\eta) \in \mathbb{R} \times \mathbb{H}^n$, $
      \Gamma_v(\cdot,s,\cdot,\eta)$ is locally integrable and
      $$
      \mathcal{H}_v \Gamma_v(\cdot,s,\cdot,\eta)=\delta_{(s,\eta)}
      $$
      (the Dirac measure supported at ${(s,\eta)}$).
  \item For every test function $u \in C_0^\infty(\mathbb{R} \times \mathbb{H}^n)$, we have
      $$
      \mathcal{H}_v(\int_{\mathbb{R} \times \mathbb{H}^n} \Gamma_v(\cdot,s,\cdot,\eta) u(s,\eta)ds d \eta)=\int_{\mathbb{R} \times \mathbb{H}^n} \Gamma_v(\cdot,s,\cdot,\eta) \mathcal{H}_v u(s,\eta)ds d \eta=u.
      $$
  \item $\Gamma_v^*(t,s,\xi,\eta)=\Gamma_v(s,t,\eta,\xi)$ is a fundamental solution for the adjoint operator $\mathcal{H}_v^*=\partial_t-\Delta_{\mathcal{Y}}-\sum_{j=1}^n (a_j Y_{1,j}+b_j Y_{2,j})$ and it satisfies the dual statements of (3) and (4).
  \item For every $t < s$, we have
      $$
      \int_{\mathbb{H}^n} \Gamma_v(t,s,\xi,\eta)d \eta=1.
      $$
  \item $\Gamma_v(t,s,\xi,\eta)$ is space-left invariant, which means it depends on $\xi,\eta$ only through $\eta^{-1} \circ \xi$. Hence, from now on we will always write $\Gamma_v(t,s,\xi,\eta)=\Gamma_v(t,s,\eta^{-1} \circ \xi)$.
  \end{enumerate}
\end{Lem}
\begin{proof}
  To prove (1), we use a technique to rephrase the problem about $\mathcal{H}_v$ to a problem about the one without drift. We have (cf. \cite[Lemma 1.3.1]{02Al})
  \begin{equation}\label{Lv}
   \mathcal{L}_v=\chi^{-1}(t,\xi)(\Delta_{\mathcal{Y}}-\frac{1}{4}|v(t)|^2)\chi(t,\xi),
  \end{equation}
  the multiplicative function $\chi$ can be written as $\chi(t,\xi)=\exp(\frac{1}{2}v \cdot \pi(\xi))$, where $\pi$ is the canonical projection $\pi:\mathbb{H}^n \rightarrow \mathbb{T}^{2n+1} \cong \mathbb{H}^n / Ker(\pi)$. It is well known that the fundamental solution of $\mathcal{H}$ has the form (cf. \cite{77Ga,76Hu} or \cite{03Lu})
  $$
  \Gamma(s-t,\eta^{-1} \circ \xi)=\frac{1}{2(4\pi(s-t))^{n+1}} \int_{\mathbb{R}}\exp(\frac{\lambda}{4(s-t)}(it-|\eta^{-1} \circ \xi|^2 \cosh \lambda)){(\frac{\lambda}{\sinh \lambda})}^n d \lambda,
  $$
  for $s-t > 0$, and vanishes for $s-t \leq 0$.
  Then it follows from \eqref{Lv} that (1) holds.

  Once (1) has been proved, the rest can be easily obtained and we omit the procedure here.
\end{proof}

We denote the Carnot-Carath\'{e}odory distance on $\mathbb{H}^n$ by $d(\cdot,\cdot)$, and write $d(\xi)=d(\xi,o)$ with $o$ defined as the origin of $\mathbb{H}^n$, where $\xi \in \mathbb{H}^n$. Moreover, it can be observed that
$d(\xi,\eta)=d(\eta^{-1} \circ \xi)$. Now we denote $\|\cdot\|$ as the homogeneous norm in $\mathbb{H}^n$, it is a well-known fact that for all $\xi \in \mathbb{H}^n$,
$$
d(\xi) \simeq \|\xi\|,
$$
see \cite[pp. 10]{22LS} or \cite[pp. 98-99]{09Li} for example.

To simplify the content, hereafter we only consider the case of $n=1$.
\begin{Lem}[Gaussian estimates of $\Gamma_v$]\label{GE}
  The following sharp upper estimates hold for every $k \in \mathbb{N}, u \in \mathbb{H}^1, s-t >0$:
  \begin{align*}
  |D_{\mathcal{Y}}^k (\Gamma_v(t,s,\cdot))(u)| \leq & C(v)\frac{1}{(s-t)^{\frac{k+Q}{2}}} \exp(-\frac{{\|u\|}^2}{C(s-t)}), \\
  |\partial_t\Gamma_v(t,s,u)| \leq & C(v)\frac{1}{(s-t)^{\frac{2+Q}{2}}} \exp(-\frac{{\|u\|}^2}{C(s-t)}),
  \end{align*}

  where $C(v)$ means a constant only depends on $v$.
\end{Lem}
\begin{proof}
  From Lemma \ref{FS}(1), it's easy to get the Gaussian estimates of $\Gamma_v$ by the Gaussian estimates of $\Gamma$ (cf. \cite[pp. 10]{22LS}), as we note that $d(u) \simeq \|u\|$.
\end{proof}

\subsection{Singular integrals on spaces of homogeneous type and continuity on weighted H\"{o}lder spaces}\label{preliminary2}

Let $X$ be a set. A function $q: X \times X \rightarrow \mathbb{R}$ is called a quasidistance on $X$ if there exists a constant $c_q \geqslant 1$ such that for any $x, y, z \in X$:
$$
\begin{gathered}
  q(x, y) \geqslant 0 \quad \text { and } \quad q(x, y)=0 \Longleftrightarrow x=y ; \\
  q(x, y)=q(y, x) ; \\
  q(x, y) \leqslant c_q(q(x, z)+q(z, y)) .
\end{gathered}
$$
We will say that two quasidistances $q, q^{\prime}$ on $X$ are equivalent, and we will write $q \simeq q^{\prime}$, if there exist two positive constants $c_1, c_2$ such that $c_1 q^{\prime}(x, y) \leqslant q(x, y) \leqslant c_2 q^{\prime}(x, y)$ for any $x, y \in X$.

For $r>0$, let $B_r(x)=\{y \in X: q(x, y)<r\}$. These ``balls" satisfy the axioms of a complete system of neighborhoods in $X$, and therefore induce a (separated) topology. With respect to this topology, the balls $B_r(x)$ need not be open. And we will explicitly exclude this kind of pathology.
\begin{Def}
  Let $(X, q)$ be a set endowed with a quasidistance $q$ such that the $q$-balls are open with respect to the topology induced by $q$, and let $\mu$ be a positive Borel measure on $X$ satisfying the doubling condition: there exists a positive constant $c_\mu$ such that
  $$
  \mu\left(B_{2 r}(x)\right) \leqslant c_\mu \cdot \mu\left(B_r(x)\right) \text { for any } x \in X, r>0.
  $$
  Then $(X, q, \mu)$ is called a space of homogeneous type.
\end{Def}

To simplify notation, the measure $d \mu(x)$ will be denoted simply by $d x$, and $\mu(A)$ will be written $|A|$. We will also set
$$
B(x ; y):=B_{q(x, y)}(x).
$$
\begin{Def}[H\"{o}lder spaces]
  For any $\alpha>0, u: X \rightarrow \mathbb{R}$, let:
  $$
  \begin{gathered}
  |u|_{C^\alpha(X)}=\sup \left\{\frac{|u(x)-u(y)|}{q(x, y)^\alpha}: x, y \in X, x \neq y\right\}, \\
  \|u\|_{C^\alpha(X)}=|u|_{C^\alpha(X)}+\|u\|_{L^{\infty}(X)}, \\
  C^\alpha(X)=\left\{u: X \rightarrow \mathbb{R}:\|u\|_{C^\alpha(X)}<\infty\right\}.
  \end{gathered}
  $$
\end{Def}
\begin{Def}
  Let $(X,q,dx)$ be a space of homogeneous type. We will say that a measurable function $k(x, y): X \times X \rightarrow \mathbb{R}$ is a standard kernel on $X$ if $k$ satisfies the following properties:
  \begin{enumerate}
  \item (``growth condition")
  $$
  |k(x, y)| \leqslant \frac{c}{|B(x ; y)|} \text { for any } x, y \in X;
  $$
  \item (``mean value inequality")
  \begin{equation}\label{MVI}
    \left|k(x, y)-k\left(x_0, y\right)\right| \leqslant \frac{c}{\left|B\left(x_0 ; y\right)\right|}\left(\frac{d\left(x_0, x\right)}{d\left(x_0, y\right)}\right)^\beta
  \end{equation}
  for any $x_0, x, y \in X$, with $d\left(x_0, y\right) \geqslant M d\left(x_0, x\right), M>1, c, \beta>0$.
  \end{enumerate}
\end{Def}
The following lemma is the key to proving the subsequent main theorems, and its proof can be found in \cite[Lemma 2.8]{05BB}.
\begin{Lem}
Let $X$ be any space of homogeneous type. Then
\begin{enumerate}
\item $\int_{d(x, y)<r} \frac{d(x, y)^\beta}{|B(x ; y)|} d y \leqslant c r^\beta$ for any $\beta>0$;
\item $\quad \int_{d(x, y)>r} \frac{d(x, y)^{-\beta}}{|B(x ; y)|} d y \leqslant c r^{-\beta}$ for any $\beta>0$.
\end{enumerate}
\end{Lem}
\begin{Lem}[$C^{\alpha}$ continuity of singular integral operator, cf. \mbox{\cite[Theorem 2.7]{07BB}}]\label{HCSIO}
  Let $(X, q, d x)$ be a bounded space of homogeneous type, and let $k(x, y)$ be a standard kernel. Let
  \begin{equation}\label{K_varepsilon}
    K_{\varepsilon} f(x)=\int_{q^{\prime}(x, y)>\varepsilon} k(x, y) f(y) d y
  \end{equation}
  where $q'$ is any quasidistance on $X$, equivalent to $q$, and fixed once and for all. Assume that for every $f \in C^\alpha(X)$ and $x \in X$ the following limit exists:
  $$
  K f(x)=P V \int_X k(x, y) f(y) d y=\lim _{\varepsilon \rightarrow 0} K_{\varepsilon} f(x).
  $$
  Also, assume that (``cancellation properties"):
  \begin{equation}\label{IP}
    \left|\int_{q^{\prime}(x, y)>r} k(x, y) d y\right| \leqslant c_K
  \end{equation}
  for any $r>0$ (with $c_K$ independent of $r$) and
  \begin{equation}\label{CP}
    \lim _{\varepsilon \rightarrow 0} |\int_{q^{\prime}(x, y)>\varepsilon} k(x, y) d y-\int_{q^{\prime}(x_0, y)>\varepsilon} k\left(x_0, y\right) d y|\leqslant c_K q(x, x_0)^\gamma
  \end{equation}
  for some $\gamma \in(0,1]$, where $q^{\prime}$ is the same quasidistance appearing in \eqref{K_varepsilon}. Then the integral operator $K$ is continuous on $C^\alpha(X)$; more precisely:
  $$
  |K f|_{C^\alpha(X)} \leqslant c_K\|f\|_{C^\alpha(X)} \text { for every } \alpha \leqslant \gamma, \alpha<\beta,
  $$
where $\gamma$ is the number in \eqref{CP} and $\beta$ is the number in \eqref{MVI}. Moreover,
  \begin{equation}\label{L^inf E}
    \|K f\|_{\infty} \leqslant c_{K, R, \alpha}\|f\|_\alpha,\text { where } R=\operatorname{diam} X.
  \end{equation}
\end{Lem}
\begin{Rem}\label{Rem_Bdd.SIO}
  If we replace condition \eqref{IP} by
  $$
  \int_{X} |k(x,y)|d y \leq c_K,
  $$
  which means the integral operator is absolutely integrable, then we have the $L^\infty$ estimate as
  $$
  \|K f\|_{\infty} \leqslant c_{K}\|f\|_\infty
  $$
  instead of \eqref{L^inf E}.
\end{Rem}
\begin{Lem}[$C^{\alpha}$ continuity of fractional integral operator, cf. \mbox{\cite[Theorem 2.11]{07BB}}]\label{HCFIO}
  Let $(X, q, d x)$ be a bounded space of homogeneous type, and assume that $X$ does not contain atoms (that is, points of positive measure). Let $k_\delta(x, y)$ be a ``fractional integral kernel", that is,
  \begin{equation}\label{GC-F}
    |k_\delta(x, y)| \leqslant \frac{c q(x, y)^\delta}{|B(x ; y)|}
  \end{equation}
  for any $x, y \in X$, some $c, \delta>0$ (``growth condition");
  \begin{equation}\label{MVI-F}
    \left|k_\delta(x, y)-k_\delta\left(x_0, y\right)\right| \leqslant \frac{c q\left(x_0, y\right)^\delta}{\left|B\left(x_0 ; y\right)\right|}\left(\frac{q\left(x_0, x\right)}{q\left(x_0, y\right)}\right)^\beta
  \end{equation}
for any $x_0, x, y \in X$, with $q\left(x_0, y\right) \geqslant 2 q\left(x_0, x\right)$, some $c, \beta>0$ (``mean value inequality"). Then the integral operator
$$
I_\delta f(x)=\int_X k_\delta(x, y) f(y) d y
$$
is continuous on $C^\alpha(X)$ for any $\alpha<\min (\beta, \delta)$.
\end{Lem}
\begin{Rem}
  In Lemma \ref{HCFIO}, we do not need to satisfy \eqref{CP} and \eqref{IP}, that is because properties \eqref{GC-F} and \eqref{MVI-F} can imply these cancellation properties.
\end{Rem}
\subsection{Results for the linearized HJB equation}
Before proving Lemma \ref{Lem_LHJB wellposed regularity}, we introduce some symbolic notations and then prove the a-prior estimates of the solution (see Lemma \ref{Lem_LHJB prior estimate}).

Let us start from the H\"{o}rmander's operator
$$
H:=-\partial_t-\sum_{i=1}^2 X_i^{2}+\sum_{i=1}^2 V_i(t, x) X_i,
$$
where $\{X_1, X_2\}$ are Grushin's vector fields, recall that
$$
X_1=\partial_{x_{1}}, X_2=x_{1} \partial_{x_{2}} \text { in } \mathbb{T}^2.
$$
They generate a Lie algebra which has not the same structure at any point, since $X_1, X_2$ are independent if and only if $x_{1} \neq 0$. In the light of the Rothschild-Stein's ``lifting and approximation" technique introduced in \cite{76RS}, we lift the above vector fields to the new ones
$$
\widetilde{X}_1 = \partial_{x_{1}}, \widetilde{X}_2=X_2+\partial_{x_{3}}=x_{1} \partial_{x_{2}}+\partial_{x_{3}} \text { in } \mathbb{T}^3 .
$$
Note that $\widetilde{X}_1, \widetilde{X}_2,\left[\widetilde{X}_1, \widetilde{X}_2\right]$ are independent at any point of $\mathbb{T}^3$. Their Lie algebra is the same as that of the Heisenberg group $\mathbb{H}^1$ in $\mathbb{T}^3$ with homogeneous dimension of $Q=4$, and actually a smooth change of variables ``$x=x_{3}, y=x_{1}, z=4x_{2}-2x_{1}x_{3}$" in $\mathbb{T}^3$ can turn these vector fields into the ``canonical form" $Y_{1}=\partial_{x}+2 y \partial_{z}, Y_{2}=\partial_{y}-2 x \partial_{z}$ of $\mathbb{H}^1$ which are left invariant with respect to the translation assigned by the group law
$$
\left(x^{\prime}, y^{\prime}, z^{\prime}\right) \circ (x, y, z)=\left(x+x^{\prime}, y+y^{\prime}, z+z^{\prime}+2\left(x y^{\prime}-x^{\prime} y\right)\right).
$$
Moreover, for any smooth function $f:\mathbb{H}^1 \rightarrow \mathbb{R},$
\begin{equation}\label{approximation}
  \widetilde{X}_i (f\left(\Theta(\cdot,\eta)\right))(\xi)=Y_{i} f \left(\Theta(\xi,\eta)\right),i \in \{1,2\},
\end{equation}
where $\mathbb{T}^3 \ni \xi \mapsto \Theta(\xi,\eta)$ for any $\eta \in \mathbb{T}^3$ is a smooth diffeomorphism defined as
\begin{equation}\label{diffeomorphism}
  \Theta(\xi,\eta):=(\xi_{3}-\eta_{3}, \xi_{1}-\eta_{1}, 4(\xi_{2}-\eta_{2})-2(\xi_{3}-\eta_{3})(\xi_{1}+\eta_{1})).
\end{equation}
Moreover, we can deduce that
\begin{equation}\label{approximation_1}
  \widetilde{X}_i (f\left(\Theta(\xi,\cdot)\right))(\eta)=-(Y_{i} f) \left(-\Theta(\xi,\eta)\right)=-(Y_{i} f) \left(\Theta(\eta,\xi)\right),i \in \{1,2\}.
\end{equation}
As it was introduced by Rothschild and Stein in \cite{76RS},
$$
\widetilde{d}^{\prime}(\xi,\eta):=\|\Theta(\xi,\eta)\|
$$
is a quasidistance equivalent to $\widetilde{d}_{cc}$, where $\|\cdot\|$ is the homogeneous norm in $\mathbb{H}^1$.

The vector fields $\widetilde{X}_1, \widetilde{X}_2$ satisfy the required condition, moreover they project onto the original $X_1, X_2$, in the sense that for any function $\widetilde{f}(x_{1},x_{2},x_{3})=f(x_{1},x_{2})$, we have
$$
X_1 \widetilde{f}=\widetilde{X}_1 \widetilde{f} ; X_2 \widetilde{f}=\widetilde{X}_2 \widetilde{f}.
$$
And once we have proved a similar estimate for $\widetilde{X}_i \widetilde{X}_j \widetilde{f}$ in a high-dimensional space, this property should allow us to easily obtain the required a-priori estimate for $X_i X_j f$.

With a slight abuse of notation, we will add the hat ``$\sim$" to denote function on $\mathbb{T}^3$ independent of $x_{3}$, namely $\widetilde{f}(\xi)=f(x), \xi \in \mathbb{T}^3 \text { while } x \in \mathbb{T}^2$.

The corresponding lifted operator is
$$
\widetilde{H}:=-\partial_t-\sum_{i=1}^2 \widetilde{X}_i^{2}+\sum_{i=1}^2 \widetilde{V}_i(t, \xi) \widetilde{X}_i.
$$
Let us now freeze the space variable of $\widetilde{V}(t, \xi)$ at the point $\xi_0=(x_0,0) \in \mathbb{T}^3$, and then consider the frozen lifted differential operator
\begin{equation}\label{widetilde H_0}
  \widetilde{H}_0:=-\partial_t-\sum_{i=1}^2 \widetilde{X}_i^2+\sum_{i=1}^2 \widetilde{V}_i^0 \widetilde{X}_i,
\end{equation}
where $\widetilde{V}_i^{0}$ denotes $\widetilde{V}_i\left(t, \xi_0\right)$ for simplification. Since the vector fields $\widetilde{X}_i$ can be globally approximated by left invariant vector fields $Y_{i}$ defined on Heisenberg group $\mathbb{H}^1$, we will consider the corresponding approximating operator
\begin{equation}\label{mathcal H_0}
  \mathcal{H}_0:=-\partial_t-\sum_{i=1}^2 Y_i^2+\sum_{i=1}^2 \widetilde{V}_i^0 Y_i.
\end{equation}
$\mathcal{H}_0$ is a nonhomogeneous but space-left invariant H\"{o}rmander's operator on $\mathbb{R} \times \mathbb{H}^1$. So it admits a space-left invariant fundamental solution $\Gamma_0$ (that is, $\Gamma_0(t,s,\xi,\eta)=\Gamma_0(t,s,\eta^{-1} \circ \xi)$). Starting with $\Gamma_0$, we build a parametrix for $\widetilde{H}_0$. Let us consider kernel $\Gamma_0\left(t,s, \Theta(\xi,\eta)\right)$ and then compute
$$
\widetilde{H}_0\left[\Gamma_0\left(\cdot,s, \Theta(\cdot,\eta)\right)\right](t, \xi).
$$
Recall the approximation relation in \eqref{approximation}, hence
$$
\widetilde{H}_0\left[\Gamma_0\left(\cdot,s, \Theta(\cdot,\eta)\right)\right](t, \xi)=\mathcal{H}_0\left[\Gamma_0\left(\cdot,s,\cdot\right)\right](t,
\Theta(\xi,\eta))=\delta_{(s,0)}(t, \Theta(\xi,\eta)),
$$
where $\delta_{(s,0)}$ is the Dirac mass at the point $(s,0)$. More precisely, we already get the parametrix (fundamental solution actually) for $\widetilde{H}_0$ as $k(t,s,\xi,\eta):=\Gamma_0(t,s,\Theta_\eta(\xi))$, which is defined on the whole space.
\begin{Rem}
In view of the fact that (see Lemma \ref{FS}(6)), when $t<s$, $\int_{\mathbb{T}^3} \Gamma_0 (t,s,y^{-1} \circ x)d y=1$, so by change of variables we can compute that $\int_{\mathbb{T}^3} \Gamma_0 (t,s,\Theta(\xi,\eta))d \eta=\frac{1}{4}$ when $t<s$, for which the Jacobian determinant is equal to
$$
|J|=\begin{vmatrix}
0 & 0 & 1\\
1 & 0 & 0\\
2(\xi_3-\eta_3) & 4 & -2(\xi_1+\eta_1)\\
\end{vmatrix}=4.
$$
\end{Rem}
The link between the fundamental solution for differential operator $\widetilde H_0$ and the abstract theory of singular integrals or fractional integrals is contained in the following propositions.
\begin{Lem}\label{FS type 0}
  Suppose $t<s$, the kernel
  $$
  k(t,s,\xi,\eta)=\Gamma_0(t,s,\Theta(\xi,\eta))
  $$
  is a singular integral kernel satisfying the conditions in Lemma \ref{HCSIO} as follows:
    \begin{enumerate}
    \item (growth condition)
    \begin{equation}\label{k-GC}
      |k(t,s,\xi,\eta)| \leqslant \frac{c}{\widetilde{d}_{cc}(\xi, \eta)^Q} \leqslant \frac{c}{|\widetilde{B}(\xi;\eta)|};
    \end{equation}
    \item (mean value inequality)
    \begin{equation}\label{k-MVI}
      \begin{aligned}
      |k(t,s,\xi,\eta)-k(t,s,\xi_1,\eta)| & \leqslant c \frac{\widetilde{d}_{cc}(\xi_1,\xi)}{\widetilde{d}_{cc}(\xi_1, \eta)^{Q+1}}\\
      & \leqslant c \frac{\widetilde{d}_{cc}(\xi_1, \xi)}{|\widetilde{B}(\xi_1;\eta)|} \cdot(\frac{\widetilde{d}_{cc}(\xi_1,\xi)}{\widetilde{d}_{cc}(\xi_1,\eta)})
      \end{aligned}
    \end{equation}
    when $\widetilde{d}_{cc}(\xi_1,\eta) > 2 \widetilde{d}_{cc}(\xi_1,\xi)$;
    \item (cancellation properties)
    \begin{equation}\label{k-CP1}
      |\int_{r<\widetilde{d}_{cc}^{\prime}(\xi,\eta)<R} k(t,s,\xi,\eta)d \eta| \leqslant c
    \end{equation}
    with $c$ independent of $r, R$, and
    \begin{equation}\label{k-CP2}
      \begin{aligned}
      \lim _{\varepsilon \rightarrow 0} | & \int_{\widetilde{d}_{cc}^{\prime}(\xi,\eta)>\varepsilon} k(t,s,\xi,\eta)d \eta-\int_{\widetilde{d}_{cc}^{\prime}(\xi_1, \eta)>\varepsilon} k(t,s,\xi_1,\eta)d \eta|\\
      \leqslant & c \widetilde{d}_{cc}(\xi_1,\xi)^\gamma
      \end{aligned}
    \end{equation}
    for any constant $\gamma$.
  \end{enumerate}
  It is worth noting that constants c above depends only on the frozen drift coefficient $\widetilde{V}^0(t)$.
\end{Lem}
\begin{proof}
  It is similar to the proof of \cite[Proposition 6.4]{07BB} that we can prove (1) and (2) by Gaussian estimates for the fundamental solution of $\mathcal{H}_0$, referring to Lemma \ref{GE}. As for (3), let us note that $k(t,s,\xi,\eta)$ is actually a non-singular integral kernel (since $t<s$) with
  $$
  \int_{\mathbb{T}^3} k(t,s,\xi,\eta)d \eta=\frac{1}{4}.
  $$
   Hence $k$ obviously satisfies the cancellation properties in (3) for any $\gamma$.
\end{proof}
\begin{Lem}\label{FS type lambda}
For every $m \in \{0,1\}$, we define a kernel as
$$
k_m(t,s,\xi,\eta) :=
\begin{cases}
\Gamma_0(t,s,\Theta(\xi,\eta)), & \mbox{when } m=0,\\
Y_i(\Gamma_0(t,s,\cdot))(\Theta(\xi,\eta)),\mbox{ for }i \in \{1,2\}, & \mbox{when } m=1.
\end{cases}
$$

Then, for any $\lambda \in (1,2)$, we have that $k_m$ is a fractional integral kernel satisfying the conditions in Lemma \ref{HCFIO} as follows:
\begin{enumerate}
\item (growth condition)
\begin{equation}\label{k_m-GC}
|k_m(t,s,\xi,\eta)| \leqslant \frac{c}{|s-t|^{\frac{\lambda}{2}}} \frac{1}{\widetilde{d}_{cc}(\xi, \eta)^{Q+m-\lambda}} \leqslant \frac{c}{|s-t|^{\frac{\lambda}{2}}} \frac{\widetilde{d}_{cc}(\xi, \eta)^{\lambda-m}}{|B(\xi;\eta)|};
\end{equation}
\item (mean value inequality)
\begin{equation}\label{k_m-MVI}
\begin{aligned}
|k_m(t,s,\xi,\eta)-k_m(t,s,\xi_1,\eta)| & \leqslant \frac{c}{|s-t|^{\frac{\lambda}{2}}} \frac{\widetilde{d}_{cc}(\xi_1,\xi)}{\widetilde{d}_{cc}(\xi_1, \eta)^{Q+m-\lambda+1}}\\
& \leqslant \frac{c}{|s-t|^{\frac{\lambda}{2}}} \frac{\widetilde{d}_{cc}(\xi_1, \xi)^{\lambda-m}}{|B(\xi_1;\eta)|} \cdot(\frac{\widetilde{d}_{cc}(\xi_1,\xi)}{\widetilde{d}_{cc}(\xi_1,\eta)})
\end{aligned}
\end{equation}
 when $\widetilde{d}_{cc}(\xi_1,\eta) > 2 \widetilde{d}_{cc}(\xi_1,\xi)$.
\end{enumerate}
The constants c above depends only on the frozen drift coefficient $\widetilde{V}^0(t)$.
\end{Lem}
\begin{proof}
We omit the proof because it is the same as the one for the above proposition and notice that $\lambda-m>0$.
\end{proof}
\begin{Def}\label{kernel type}
We call that the integral kernel satisfying conditions \eqref{k-GC}, \eqref{k-MVI}, \eqref{k-CP1} and \eqref{k-CP2} in Lemma \ref{FS type 0} is a kernel of type 0. In the similar way, the integral kernel satisfying conditions \eqref{k_m-GC} and \eqref{k_m-MVI} in Lemma \ref{FS type lambda} is a kernel of type $\lambda-m$ with constant $\frac{c}{|s-t|^{\frac{\lambda}{2}}}$. In addition, the corresponding integral operator is called an operator of type 0 or an operator of type $\lambda-m$ with constant $\frac{c}{|s-t|^{\frac{\lambda}{2}}}$, respectively.
\end{Def}
The next corollary is also essential for proving Lemma \ref{Lem_LHJB Duhamel formula} and Lemma \ref{Lem_LHJB prior estimate}.
\begin{Cor}\label{widetilde Xi T}
The following two statements hold true.
\begin{enumerate}
\item Suppose $t<s$, define integral operator
$$
T_0 f:=PV \int_{\mathbb{T}^3} k(t,s,\xi,\eta)f(\eta)d \eta,
$$
where $k$ is the kernel from Lemma \ref{FS type 0}. Then there exist two operators of type 0 over $\mathbb{T}^3$, denoted as $T_0^h$ for $h=1,2$, such that for any $f \in C_\mathcal{X}^1(\mathbb{T}^3)$ one has:
$$
\widetilde{X}_i T_0 f=\sum_{h=1}^{2}T_0^h \widetilde{X}_h f,\quad i \in \{1,2\},
$$
where operators $T_0^h$ depend on the selection of the vector field $\widetilde{X}_i$.
\item Suppose $t \leq s, m \in \{0,1\}$, for every fixed $\lambda \in (1,2)$, define integral operator
$$
T_{\lambda-m} f:=\int_{\mathbb{T}^3} k_m(t,s,\xi,\eta)f(\eta)d \eta.
$$
where $k_m$ is the kernel from Lemma \ref{FS type lambda}. Then there exist two operators of type $\lambda-m$ with constant $\frac{c}{(s-t)^{\frac{\lambda}{2}}}$ over $\mathbb{T}^3$, denoted as $T_{\lambda-m}^h$ for $h=1,2$, such that for any $f \in C_\mathcal{X}^1(\mathbb{T}^3)$ one has:
$$
\widetilde{X}_i T_{\lambda-m}f=\sum_{h=1}^{2}T_{\lambda-m}^h \widetilde{X}_h f,\quad i \in \{1,2\},
$$
where operators $T_{\lambda-m}^h$ depend on the selection of the vector field $\widetilde{X}_i$.
\end{enumerate}
\end{Cor}
\begin{proof}
We refer primarily to \cite[pp. 292]{76RS} for these two conclusions. Here we prove only the first one, while the proof of the second one is similar.

The relation in \eqref{approximation} can be simplified as $\widetilde{X}_i^{\xi}=Y_i$, where the superscript $\xi$ means applying $\widetilde{X}_i$ to the $\xi$ variable of $f(\Theta(\xi,\eta))$. Similarly by calculating we can get
$$
\widetilde{X}_1^\eta=-Y_1-(\xi_3-\eta_3)[Y_1,Y_2], \quad \widetilde{X}_2^\eta=-Y_2+(\xi_1-\eta_1)[Y_1,Y_2].
$$
Since $[\widetilde{X}_1^\eta,\widetilde{X}_2^\eta]=\partial_{x_2} =-4\partial_z=-[Y_1,Y_2]$, we then have
$$
\widetilde{X}_i^\xi=-\widetilde{X}_i^\eta +\Theta_i[\widetilde{X}_i^\eta,\widetilde{X}_j^\eta],\quad i \neq j.
$$
More concretely, we can know
\begin{align*}
\widetilde{X}_i^\xi k(t,s,\xi,\eta)= & (-\widetilde{X}_i^\eta +\Theta_i[\widetilde{X}_i^\eta,\widetilde{X}_j^\eta]) k(t,s,\xi,\eta) \\
= & -\widetilde{X}_i^\eta k(t,s,\xi,\eta)+(\widetilde{X}_i^\eta \widetilde{X}_j^\eta-\widetilde{X}_j^\eta \widetilde{X}_i^\eta)(\Theta_i k)(t,s,\xi,\eta) \\
= & -\sum_{h=1}^{2} \widetilde{X}_h^\eta k^{h}(t,s,\xi,\eta),
\end{align*}
where the new kernels $k^h$ are defined to be
$$
k^{i}=(2-\Theta_i \widetilde{X}_j^\eta)k(t,s,\xi,\eta), k^{j}=\Theta_i \widetilde{X}_i^\eta k(t,s,\xi,\eta),\quad i \neq j.
$$
Next it is sufficient to verify that the integral kernels $k^h$ satisfy the three conditions in Lemma \ref{FS type 0}. The first two are easy to get, as for the cancellation properties, we likewise know that the corresponding integral kernel is non-singular since $t \neq s$, that is, integrating by parts we have
\begin{align*}
\int_{\mathbb{T}^3} \Theta_i \widetilde{X}_j^\eta k(t,s,\xi,\eta)d \eta= & 1,\\
\int_{\mathbb{T}^3} \Theta_i \widetilde{X}_i^\eta k(t,s,\xi,\eta)d \eta= & 0.
\end{align*}
So the cancellation properties hold trivially.
\end{proof}
We can prove the following lemma which is useful in the proof of Lemma \ref{Lem_LHJB prior estimate}.
\begin{Lem}\label{Lem_LHJB Duhamel formula}
Let $\widetilde{H}_0$ be as the frozen lifted operator in \eqref{widetilde H_0} with $\widetilde{V}^0 \in C^1\left([0,T]\right)$. Suppose $z_T \in C\left(\mathbb{T}^3\right)$ and $f \in C\left([0,T]\times\mathbb{T}^3\right)$, then the function
$$
z(t,\xi)=\int_{\mathbb{T}^3}\Gamma_0(t,T,\Theta(\xi,\eta))z_T(\eta)d \eta + \int_{t}^{T}\int_{\mathbb{T}^3}\Gamma_0(t,s,\Theta(\xi,\eta))f(s,\eta)d \eta d s
$$
belongs to the class
$$
C_{\mathcal{X}}^{1,2} \left([0, T) \times \mathbb{T}^{3} \right) \bigcap C\left([0, T] \times \mathbb{T}^3\right)
$$
and is the unique solution to the following backward Cauchy problem
\begin{equation}\label{FLE}
\begin{cases}
\widetilde{H}_0 z(t,\xi)=f(t,\xi),& \text {in }[0, T) \times \mathbb{T}^3, \\
z(T,\xi)=z_T(\xi),& \text {in }\mathbb{T}^3.
\end{cases}
\end{equation}

\end{Lem}
\begin{proof}
The uniqueness can be obtained by weak maximum principle (cf. \cite[Theorem 13.1]{10BBLU}) since the domain is bounded. So it is sufficient to check that $z(t,\xi)$ satisfies the equation \eqref{FLE}, which will be proved in two steps.

\emph{Step 1: The statement holds if $\widetilde{f} \equiv 0$.} Since $\widetilde{V}_0(t) \in C^1([0,T])$, we have
$$
\Gamma_0(\cdot,T,\Theta(\cdot,\eta)) \in C_{\widetilde{\mathcal{X}}}^{1,2} \left([0, T-\epsilon] \times \mathbb{T}^{3} \right), \text{ for each } \epsilon > 0,
$$
with uniformly bounded derivatives $\partial_t \Gamma_0, D_{\widetilde{\mathcal{X}}} \Gamma_0$ and $D_{\widetilde{\mathcal{X}}}^2 \Gamma_0$, then we see that $z \in C_{\widetilde{\mathcal{X}}}^{1,2} \left([0, T) \times \mathbb{T}^{3} \right)$. Furthermore, differentiate under the integral sign and obtain
$$
\widetilde{H}_0z(t,\xi)=0 \text{ for } t \in [0,T).
$$

Let $\xi^* \in \mathbb{T}^3$ be fixed, next we prove that $z(t,\xi) \rightarrow \widetilde{z}_T(\xi^*)$, as $(t,\xi) \rightarrow (T,\xi^*)$. Actually, we have
\begin{align*}
|z(t,\xi)-\widetilde{z}_T(\xi^*)| \leq & \int_{|\eta-\xi^*| \leq \delta} \Gamma_0(t,T,\Theta(\xi,\eta))|\widetilde{z}_T(\eta)-\widetilde{z}_T(\xi^*)|d \eta \\
& + 2\sup |\widetilde{z}_T| \int_{|\eta-\xi^*| > \delta} \Gamma_0(t,T,\Theta(\xi,\eta))d \eta.
\end{align*}
The first integral can be made small by choosing $\delta>0$ small enough, since $\widetilde{z}_T$ is continuous and $\int_{\mathbb{T}^3} \Gamma_0(t,T,\Theta(\xi,\eta))d \eta=\frac{1}{4}$ for $t<T$. Once $\delta$ is fixed, by dominated convergence theorem, the second integral tends to zero as $(t,\xi) \rightarrow (T,\xi^*)$, noting that $\Gamma_0(t,T,\Theta(\xi,\eta)) \geq 0$ and vanishes for $t \geq T$.

\emph{Step 2: The statement holds if $\widetilde{z}_T \equiv 0$.} Let $\widetilde{f}_n(t,\xi)$ be the standard (Euclidean) mollified version of $\widetilde{f}(t,\xi)$, that is
$$
\widetilde{f}_n(t,\xi):=\int_{0}^{T}\int_{\mathbb{T}^3}\phi_n(t-s,\xi-\eta) \widetilde{f}(s,\eta)d\eta ds,
$$
where the mollifier
$$
\phi_n(t,\xi):=\begin{cases}
Cn^4\exp\left(\frac{1}{n^2 t^2-1}+\frac{1}{n^2|\xi|^2-1}\right), & \mbox{if } (t,\xi) \in (-\frac{1}{n},\frac{1}{n})\times B_\frac{1}{n}^E, \\
0, & \mbox{otherwise},
\end{cases}
$$
where the constant $C>0$ is selected so that $\int_{-\infty}^{+\infty}\int_{\mathbb{R}^3}\phi_n d\xi dt =1$ and $B^E$ denotes the Euclidean ball. Then $\widetilde{f}_n(t,\xi) \in C^\infty([0,T] \times \mathbb{T}^3)$, and $\lim\limits_{n \rightarrow \infty}\|\widetilde{f}_n(t,\xi)-\widetilde{f}(t,\xi)\|_\infty=0$. Define
$$
z_n(t,\xi):=\int_{t}^{T}\int_{\mathbb{T}^3}\Gamma_0(t,s,\Theta(\xi,\eta)) \widetilde{f}_n(s,\eta)d \eta d s,
$$
and easily we have $z_n(t,\xi) \rightarrow z(t,\xi)$ as $n \rightarrow \infty$ by the dominated convergence theorem.

Because of the singularity of $\Gamma_0$ at the diagonal, we cannot directly justify differentiating under the integral sign. To overcome this obstacle, we split
\begin{align*}
& z_n(t,\xi)\\
& \quad = \int_{t}^{t+\epsilon}\int_{\mathbb{T}^3} \Gamma_0(t,s,\Theta(\xi,\eta))\widetilde{f}_n(s,\eta)d \eta d s + \int_{t+\epsilon}^{T}\int_{\mathbb{T}^3} \Gamma_0(t,s,\Theta(\xi,\eta))\widetilde{f}_n(s,\eta)d \eta d s \\
&\quad = : z_n^1(t,\xi)+z_n^2(t,\xi).
\end{align*}
As for $z_n^1(t,\xi)$, since by Lemma \ref{FS type lambda}, $\Gamma_0(t,s,\Theta(\xi,\eta))$ is a frozen kernel of type $\lambda$, where $\lambda$ can be chosen arbitrarily in the interval $(1,2)$, we denote the corresponding frozen integral operator as $T_{\lambda}(\xi_0)$. Then, by Corollary \ref{widetilde Xi T}, we compute
$$
\widetilde{X}_i z_n^1(t,\xi)=\int_{t}^{t+\epsilon} \widetilde{X}_i T_{\lambda}(\xi_0) \widetilde{f}_n(s,\cdot)(\xi)d s=\int_{t}^{t+\epsilon} \sum_{h=1}^{2} T_{\lambda}^h(\xi_0) (\widetilde{X}_h \widetilde{f}_n(s,\cdot))(\xi)d s.
$$
Furthermore,
$$
\widetilde{X}_j \widetilde{X}_i z_n^1(t,\xi)=\int_{t}^{t+\epsilon} \sum_{h,l=1}^{2} T_{\lambda}^{h,l}(\xi_0) (\widetilde{X}_l \widetilde{X}_h \widetilde{f}_n(s,\cdot))(\xi)d s.
$$
In addition, we can write (see Lemma \ref{FS})
\begin{align*}
z_n^1(t,\xi)= & \int_{t}^{t+\epsilon} \int_{\mathbb{T}^3} \exp\big(-\frac{1}{4}\int_{t}^{s}|\widetilde{V}^0(\tau)|^2d \tau-\frac{1}{2} \langle\widetilde{V}^0(t),\Theta(\xi,\eta)\rangle\big)\\
&\qquad\qquad \Gamma\left(s-t,\Theta(\xi,\eta)\right) \widetilde{f}_n(s,\eta)d \eta d s\\
= & \int_{0}^{\epsilon} \int_{\mathbb{T}^3} \exp\big(-\frac{1}{4}\int_{t}^{t+h}|\widetilde{V}^0(\tau)|^2d \tau-\frac{1}{2}\langle\widetilde{V}^0(t),\Theta(\xi,\eta)\rangle\big)\\
&\qquad\qquad \Gamma\left(h,\Theta(\xi,\eta)\right) \widetilde{f}_n(t+h,\eta)d \eta d h.
\end{align*}
Then we compute
\begin{align*}
\partial_t z_n^1(t,\xi)= & \int_{0}^{\epsilon} \int_{\mathbb{T}^3} \big(\frac{1}{4}|\widetilde{V}^0(\tau)|^2\bigg|_{t+h}^t-\frac{1}{2} \langle{{}\widetilde{V}^0}^{\prime}(t),\Theta(\xi,\eta)\rangle\big)\\
&\qquad\qquad \exp\big(-\frac{1}{4}\int_{t}^{t+h}|\widetilde{V}^0(\tau)|^2 d \tau-\frac{1}{2}\langle\widetilde{V}^0(t),\Theta(\xi,\eta)\rangle\big)\\
&\qquad\qquad \Gamma\left(h,\Theta(\xi,\eta)\right) \widetilde{f}_n(t+h,\eta)d \eta dh\\
& + \int_{0}^{\epsilon} \int_{\mathbb{T}^3} \exp\big(-\frac{1}{4}\int_{t}^{t+h}|\widetilde{V}^0(\tau)|^2 d \tau-\frac{1}{2}\langle\widetilde{V}^0(t),\Theta(\xi,\eta)\rangle\big)\\ &\qquad\qquad \Gamma\left(h,\Theta(\xi,\eta)\right) \partial_t \widetilde{f}_n(t+h,\eta)d \eta d h.
\end{align*}
Combining above equalities, from Remark \ref{Rem_Bdd.SIO}, we can estimate
\begin{align*}
  |\widetilde{H}_0 z_n^1(t,\xi)| \leq & c(\mathbb{T}^3)\int_{t}^{t+\epsilon}\frac{1}{(t-s)^{\frac{\lambda}{2}}}d s(\|D_{\widetilde{\mathcal{X}}}^2\widetilde{f}_n\|_\infty+\|\widetilde{V}^0\|_\infty \|D_{\widetilde{\mathcal{X}}}\widetilde{f}_n\|_\infty) \\
  & + c(\mathbb{T}^3,\widetilde{V}^0)\int_{0}^{\epsilon}\frac{1}{h^{\frac{\lambda}{2}}}d h(\|\widetilde{f}_n\|_\infty+\|\partial_t \widetilde{f}_n\|_\infty) \\
  \leq & \epsilon^{1-\frac{\lambda}{2}}c(\mathbb{T}^3,\widetilde{V}^0,\widetilde{f}_n),
\end{align*}
and it is noted that $c(\mathbb{T}^3)$ means the constant is dependent of the diam of $\mathbb{T}^3$, we will omit it after this.

We also find
\begin{align*}
  \widetilde{X}_i z_n^2(t,\xi)= & \int_{t+\epsilon}^T \int_{\mathbb{T}^3} \widetilde{X}_i \Gamma_0(t,s,\Theta(\cdot,\eta))(\xi)\widetilde{f}_n(s,\eta)d \eta d s,\\
  \widetilde{X}_j \widetilde{X}_i z_n^2(t,\xi)= & \int_{t+\epsilon}^T \int_{\mathbb{T}^3} \widetilde{X}_j \widetilde{X}_i \Gamma_0(t,s,\Theta(\cdot,\eta))(\xi)\widetilde{f}_n(s,\eta)d \eta d s,\\
  \partial_t z_n^2(t,\xi)= & \int_{t+\epsilon}^T \int_{\mathbb{T}^3} \partial_t \Gamma_0(t,s,\Theta(\xi,\eta))\widetilde{f}_n(s,\eta)d \eta d s\\
  & -\int_{\mathbb{T}^3} \Gamma_0(t,t+\epsilon,\Theta(\xi,\eta))\widetilde{f}_n(t+\epsilon,\eta)d \eta.
\end{align*}
So
\begin{align*}
& \widetilde{H}_0 z_n^2(t,\xi)\\
= & \int_{t+\epsilon}^T \int_{\mathbb{T}^3} \widetilde{H}_0 \Gamma_0(\cdot,s,\Theta(\cdot,\eta))(t,\xi)\widetilde{f}_n(s,\eta)d \eta d s+\!\int_{\mathbb{T}^3} \Gamma_0(t,t+\epsilon,\Theta(\xi,\eta))\widetilde{f}_n(t+\epsilon,\eta)d \eta \\
= & \int_{\mathbb{T}^3} \Gamma_0(t,t+\epsilon,\Theta(\xi,\eta))\widetilde{f}_n(t+\epsilon,\eta)d \eta.
\end{align*}
Based on the above calculations, it is worth to note that, for any $n,m \in \mathbb{N}$, and $|J| \in \{0,1,2\}$,
\begin{align*}
  \|\widetilde{X}^J z_n-\widetilde{X}^J z_m\|_\infty+\|\partial_t z_n-\partial_t z_m\|_\infty \leq & c\|\widetilde{f}_n-\widetilde{f}_m\|_\infty +\epsilon^{1-\frac{\lambda}{2}}c(\widetilde{V}^0,\widetilde{f}_n-\widetilde{f}_m).
\end{align*}
Choosing $\epsilon$ small enough and letting $n,m \rightarrow \infty$, we can get that $\widetilde{X}^J z_n,\partial_t z_n$ uniformly converge to $\widetilde{X}^J z,\partial_t z$ by Cauchy criterion for uniform convergence.

Apart from this, we have
$$
\widetilde{H}_0 z_n(t,\xi)=\lim_{\epsilon \rightarrow 0} \int_{\mathbb{T}^3} \Gamma_0(t,t+\epsilon,\Theta(\xi,\eta))\widetilde{f}_n(t+\epsilon,\eta)d \eta=\widetilde{f}_n(t,\xi),\quad (t,\xi) \in [0,T) \times \mathbb{T}^3.
$$
The limit as $\epsilon \rightarrow 0$ has been computed as in Step 1. Then, we let $n \rightarrow \infty$ to find that $\widetilde{H}_0 z(t,\xi)=\widetilde{f}(t,\xi)$, with $(t,\xi) \in [0,T) \times \mathbb{T}^3$. Finally note that $\|z(t,\cdot)\|_\infty \leq \frac{1}{4}(T-t)\|\widetilde{f}\|_\infty \rightarrow 0$ when $t \rightarrow T$, thus we obtain the conclusion.

Combining the above two steps, we get the final results.
\end{proof}
\begin{Lem}\label{Lem_LHJB prior estimate}
Let $V \in C^{1,1+\alpha}_{\mathcal{X}}\left([0,T]\times\mathbb{T}^2\right)$, $f \in C\left([0,T];C^{1+\alpha}_{\mathcal{X}}\left(\mathbb{T}^2\right)\right)$ and $z_T \in C^{2+\alpha}_{\mathcal{X}}\left(\mathbb{T}^2\right)$. Suppose that $z$ is a solution to equation \eqref{LHJB} in $[0, T] \times \mathbb{T}^{2}$, belonging to $C \left([0, T]; C_{\mathcal{X}}^{2+\alpha}\left(\mathbb{T}^2\right)\right)$. Then $z$ satisfies
\begin{equation}\label{MaxE}
  \|z\|_\infty \leqslant \left\|z_T\right\|_\infty+C\|f\|_\infty
\end{equation}
for some positive constant C depending on $T$ only.

Furthermore, $z$ satisfies
\begin{equation}\label{prior USE}
  \sup _{t \in[0, T]}\|z(t, \cdot)\|_{2+\alpha} \leqslant C\left(\left\|z_T\right\|_{2+\alpha}+\sup _{t \in[0, T]}\|f(t, \cdot)\|_{1+\alpha}\right)
\end{equation}
for some positive constant $C$, dependent of $\sup\limits _{t \in[0, T]}\|V(t, \cdot)\|_{1+\alpha}$ and $\alpha$ only.

In addition, for any constant $T' \in (0,T)$ and $\beta \in \left(0,\frac{1}{2}\right)$,
\begin{equation}\label{prior HSE}
\sup _{\substack{t \neq t'\\t,t' \in [0,T']}} \frac{\left\|z\left(t^{\prime}, \cdot\right)-z(t, \cdot)\right\|_{2+\alpha}}{\left|t^{\prime}-t\right|^\beta} \leqslant C\left(\left\|z_T\right\|_{2+\alpha}+\sup _{t \in[0, T]}\|f(t, \cdot)\|_{1+\alpha}\right)
\end{equation}
for some positive constant $C$, dependent of $T'$, $\sup\limits_{t \in[0, T]}\|V(t, \cdot)\|_{1+\alpha}$ and $\alpha$ only.
\end{Lem}

\begin{proof}
Estimate \eqref{MaxE} can be proved by weak maximum principle. Indeed, let
$$
v(t):=\|z_T\|_{L^\infty(\mathbb{T}^2)}+ e^t\|f\|_{L^\infty([0,T]\times\mathbb{T}^2)}, t \in [0,T],
$$
then
$$
H v(t,x)=e^t\|f\|_{L^\infty([0,T]\times\mathbb{T}^2)} \geq \|f\|_{L^\infty([0,T]\times\mathbb{T}^2)} \geq f(t,x)=H z(t,x) \text{ in } [0,T]\times\mathbb{T}^2,
$$
While $v \geq z_T$ on $\{T\} \times \mathbb{T}^d$, hence by the weak maximum principle in \cite[Theorem 13.1]{10BBLU},
\begin{align*}
  \|z\|_{L^\infty([0,T]\times\mathbb{T}^2)} \leq \|v\|_{L^\infty([0,T]\times\mathbb{T}^2)} \leq & \|z_T\|_{L^\infty(\mathbb{T}^2)}+ e^T\|f\|_{L^\infty([0,T]\times\mathbb{T}^2)} \\
  \leq & \|z_T\|_{L^\infty(\mathbb{T}^2)}+ C\|f\|_{L^\infty([0,T]\times\mathbb{T}^2)}.
\end{align*}

In order to prove estimates \eqref{prior USE} and \eqref{prior HSE}, we need to refer to the technique of ``lifting and approximation" mentioned before.

Writing $\widetilde{H}=\widetilde{H}-\widetilde{H}_0+\widetilde{H}_0=\sum\limits_{i=1}^2 \left(\widetilde{V}_i(s, \xi)-\widetilde{V}_i(s, \xi_0)\right) \widetilde{X}_i+\widetilde{H}_0$, and defining $z_{\varepsilon,\delta}(t,\xi):=\widetilde{z}(t, \xi)\psi_{\varepsilon,\delta}(t, \xi)$, where $\psi_{\varepsilon,\delta}$ is a cutoff function satisfying
$$
(-\frac{\varepsilon}{2}+t_0,\frac{\varepsilon}{2}+t_0) \times \widetilde{B}_{\frac{\delta}{2}}(\xi_0) \prec \psi_{\varepsilon,\delta} \prec (-\varepsilon+t_0,\varepsilon+t_0) \times \widetilde{B}_{\delta}(\xi_0),
$$
for any $t_0 \in [0,T]$ and any $\varepsilon, \delta > 0$. This means $0 \leq \psi_{\varepsilon,\delta} \leq 1$, $\psi_{\varepsilon,\delta} \equiv 1$ on $(-\frac{\varepsilon}{2}+t_0,\frac{\varepsilon}{2}+t_0) \times \widetilde{B}_{\frac{\delta}{2}}(\xi_0)$ and $\operatorname{sprt}\psi_{\varepsilon,\delta} \subseteq (-\varepsilon+t_0,\varepsilon+t_0) \times \widetilde{B}_{\delta}(\xi_0)$. Observing that $\psi_{\varepsilon,\delta}$ is the solution to the following backward Cauchy problem:
$$
\begin{cases}
\widetilde{H}_0 z_{\varepsilon,\delta}+\left(\widetilde{V}-\widetilde{V}^0\right) \cdot D_{\widetilde{\mathcal{X}}} z_{\varepsilon,\delta}=\widetilde{f} \psi_{\varepsilon,\delta}-g_{\varepsilon,\delta}, & \text {in }[0, T) \times \mathbb{T}^3, \\
z_{\varepsilon,\delta}(T, \xi)=\widetilde{z}_T \psi_{\varepsilon,\delta}(T,\xi), & \text {in } \mathbb{T}^3,
\end{cases}
$$
where
\begin{align*}
& g_{\varepsilon,\delta}(t, \xi)\\
:= & \widetilde{z}(t, \xi) \partial_t \psi_{\varepsilon,\delta}(t, \xi)+\sum_{i=1}^2 \widetilde{X}_i \widetilde{z}(t, \xi) \widetilde{X}_i \psi_{\varepsilon,\delta}(t, \xi) \\
& + \widetilde{z}(t, \xi) \Delta_{\widetilde{\mathcal{X}}} \psi_{\varepsilon,\delta}(t, \xi)-\widetilde{z}(t, \xi) \widetilde{V}(t, \xi) \cdot D_{\widetilde{\mathcal{X}}} \psi_{\varepsilon,\delta}(t, \xi), \text{ any }(t, \xi) \in (0, T) \times \mathbb{T}^3.
\end{align*}
According to Lemma \ref{Lem_LHJB Duhamel formula}, Lemma \ref{FS type 0} as well as Lemma \ref{FS type lambda}, and using notations $T_{\lambda}(\xi_0)$, where $\lambda \in (1,2)$ is fixed, and $T_0(\xi_0)$ to respectively denote the frozen integral operator of type ${\lambda}$ with constant $\frac{c}{|s-t|^{\frac{\lambda}{2}}}$ and the frozen integral operator of type 0, thus when $0 \leq t<T$, we have
\begin{align*}
z_{\varepsilon,\delta}(t, \xi)=& \int_{t}^{T} T_{\lambda}(\xi_0) \widetilde{f} \psi_{\varepsilon,\delta}(s,\xi)d s - \int_{t}^{T} T_{\lambda}(\xi_0)g_{\varepsilon,\delta}(s,\xi)d s \\
& + T_0(\xi_0) \widetilde{z}_T \psi_{\varepsilon,\delta}(T,\xi) + \int_{t}^{T} T_{\lambda}(\xi_0)\left(\widetilde{V}^0-\widetilde{V}\right) \cdot D_{\widetilde{\mathcal{X}}} z_{\varepsilon,\delta}(s,\xi)d s,
\end{align*}

Take derivatives of both sides of the above equality and apply Corollary \ref{widetilde Xi T}(1) to get
\begin{align*}
\widetilde{X}_i z_{\varepsilon,\delta}(t, \xi)=& \int_{t}^{T} T_{\lambda-1}(\xi_0) \widetilde{f} \psi_{\varepsilon,\delta}(s,\xi)d s - \int_{t}^{T} T_{\lambda-1}(\xi_0) g_{\varepsilon,\delta}(s,\xi)d s\\
& + \sum_{h=1}^{2} T_0^h(\xi_0) \widetilde{X}_h (\widetilde{z}_T \psi_{\varepsilon,\delta})(T,\xi)\\
& + \int_{t}^{T} T_{\lambda-1}(\xi_0) \left(\widetilde{V}^0-\widetilde{V}\right) \cdot D_{\widetilde{\mathcal{X}}} z_{\varepsilon,\delta}(s,\xi)d s.
\end{align*}
Furthermore, apply Corollary \ref{widetilde Xi T} (2) to have
\begin{align*}
& \widetilde{X}_j \widetilde{X}_i z_{\varepsilon,\delta}(t, \xi)\\
=& \int_{t}^{T} \sum_{h=1}^{2} T_{\lambda-1}^h(\xi_0) \widetilde{X}_h(\widetilde{f} \psi_{\varepsilon,\delta})(s,\xi)d s - \int_{t}^{T} \sum_{h=1}^{2} T_{\lambda-1}^h(\xi_0) \widetilde{X}_h g_{\varepsilon,\delta}(s,\xi)d s\\
& + \sum_{h,l=1}^{2} T_0^{h,l}(\xi_0) \widetilde{X}_l \widetilde{X}_h (\widetilde{z}_T \psi_{\varepsilon,\delta})(T,\xi)\\
& + \int_{t}^{T} \sum_{h=1}^{2} T_{\lambda-1}^h(\xi_0)\left(\left(\widetilde{V}^0-\widetilde{V}\right) \cdot \widetilde{X}_h D_{\widetilde{\mathcal{X}}} z_{\varepsilon,\delta}-\widetilde{X}_h \widetilde{V} \cdot D_{\widetilde{\mathcal{X}}} z_{\varepsilon,\delta}\right)(s,\xi)d s.
\end{align*}

Let us prove the estimate \eqref{prior USE}. Apply the H\"{o}lder continuity of singular or fractional integral operator (see Lemma \ref{HCSIO} and Lemma \ref{HCFIO}) to the space $(\widetilde{B}_{\delta}(\xi_0),\widetilde{d}_{cc},d\xi)$, then  for any $t_0 \in [0,T)$,
\begin{align*}
& \|z_{\varepsilon,\delta}(t_0, \cdot)\|_{C_{\widetilde{\mathcal{X}}}^{\alpha}(\widetilde{B}_{\delta}(\xi_0))}\\
\leq & \int_{t_0}^{T} \frac{C}{(s-t_0)^\frac{\lambda}{2}} \|\widetilde{f}\psi_{\varepsilon,\delta}(s,\cdot)\|_ {C_{\widetilde{\mathcal{X}}}^{\alpha}(\widetilde{B}_{\delta}(\xi_0))}d s + C\| \widetilde{z}_T \psi_{\varepsilon,\delta}(T,\cdot)\|_ {C_{\widetilde{\mathcal{X}}}^{\alpha}(\widetilde{B}_{\delta}(\xi_0))}\\
& + \int_{t_0}^{T} \frac{C}{(s-t_0)^\frac{\lambda}{2}} \|g_{\varepsilon,\delta}(s, \cdot)\|_{C_{\widetilde{\mathcal{X}}}^{\alpha}(\widetilde{B}_{\delta}(\xi_0))}d s\\
& + \int_{t_0}^{T} \frac{C}{(s-t_0)^\frac{\lambda}{2}} \|(\widetilde{V}^0-\widetilde{V}) \cdot D_{\widetilde{\mathcal{X}}} z_{\varepsilon,\delta}(s,\cdot)\|_{C_{\widetilde{\mathcal{X}}}^{\alpha}(\widetilde{B}_{\delta}(\xi_0))}d s\\
\leq & \varepsilon^{1-\frac{\lambda}{2}} C \sup _{s \in [t_0,T]} \|\widetilde{f}\psi_{\varepsilon,\delta}(s, \cdot)\|_{C_{\widetilde{\mathcal{X}}}^{\alpha}(\widetilde{B}_{\delta}(\xi_0))} + C\| \widetilde{z}_T \psi_{\varepsilon,\delta}(T,\cdot)\|_{C_{\widetilde{\mathcal{X}}}^{\alpha}(\widetilde{B}_{\delta}(\xi_0))}\\
& + \varepsilon^{1-\frac{\lambda}{2}} C \sup _{s \in [t_0,T]} \|g_{\varepsilon,\delta}(s, \cdot)\|_{C_{\widetilde{\mathcal{X}}}^{\alpha}(\widetilde{B}_{\delta}(\xi_0))}\\
& + \varepsilon^{1-\frac{\lambda}{2}} C \sup _{s \in [t_0,T]} \|(\widetilde{V}^0-\widetilde{V}) \cdot D_{\widetilde{\mathcal{X}}} z_{\varepsilon,\delta}(s, \cdot)\|_{C_{\widetilde{\mathcal{X}}}^{\alpha}(\widetilde{B}_{\delta}(\xi_0))},
\end{align*}
in the same way,
\begin{align*}
& \|\widetilde{X}_i z_{\varepsilon,\delta}(t_0, \cdot)\|_{C_{\widetilde{\mathcal{X}}}^{\alpha}(\widetilde{B}_{\delta}(\xi_0))}\\
\leq & \varepsilon^{1-\frac{\lambda}{2}} C \sup _{s \in [t_0,T]} \|\widetilde{f}\psi_{\varepsilon,\delta}(s, \cdot)\|_{C_{\widetilde{\mathcal{X}}}^{\alpha}(\widetilde{B}_{\delta}(\xi_0))}+C\| \widetilde{z}_T \psi_{\varepsilon,\delta}(T,\cdot)\|_ {C_{\widetilde{\mathcal{X}}}^{1+\alpha}(\widetilde{B}_{\delta}(\xi_0))}\\
& + \varepsilon^{1-\frac{\lambda}{2}} C \sup _{s \in [t_0,T]} \|g_{\varepsilon,\delta}(s, \cdot)\|_{C_{\widetilde{\mathcal{X}}}^{\alpha}(\widetilde{B}_{\delta}(\xi_0))}\\
& + \varepsilon^{1-\frac{\lambda}{2}} C \sup _{s \in [t_0,T]} \|(\widetilde{V}^0-\widetilde{V}) \cdot D_{\widetilde{\mathcal{X}}} z_{\varepsilon,\delta}(s, \cdot)\|_{C_{\widetilde{\mathcal{X}}}^{\alpha}(\widetilde{B}_{\delta}(\xi_0))},
\end{align*}
\begin{align*}
& \|\widetilde{X}_j \widetilde{X}_i z_{\varepsilon,\delta}(t_0,\cdot)\|_ {C_{\widetilde{\mathcal{X}}}^{\alpha}(\widetilde{B}_{\delta}(\xi_0))}\\
\leq & \varepsilon^{1-\frac{\lambda}{2}} C \sup _{s \in [t_0,T]} \|\widetilde{f} \psi_{\varepsilon,\delta}(s,\cdot)\|_ {C_{\widetilde{\mathcal{X}}}^{1+\alpha}(\widetilde{B}_{\delta}(\xi_0))} +C\| \widetilde{z}_T \psi_{\varepsilon,\delta}(T,\cdot) \|_{C_{\widetilde{\mathcal{X}}}^{2+\alpha}(\widetilde{B}_{\delta}(\xi_0))}\\
& + \varepsilon^{1-\frac{\lambda}{2}} C \sup _{s \in [t_0,T]} \|g_{\varepsilon,\delta}(s,\cdot)\|_ {C_{\widetilde{\mathcal{X}}}^{1+\alpha}(\widetilde{B}_{\delta}(\xi_0))}\\
& + \varepsilon^{1-\frac{\lambda}{2}} C \sup _{s \in [t_0,T]} \sum_{h=1}^{2} \left( \| (\widetilde{V}^0-\widetilde{V}) \cdot \widetilde{X}_h D_{\widetilde{\mathcal{X}}} z_{\varepsilon,\delta}(s, \cdot)\|_{C_{\widetilde{\mathcal{X}}}^{\alpha}(\widetilde{B}_{\delta}(\xi_0))}\right.\\
&\left. + \|\widetilde{X}_h \widetilde{V} \cdot D_{\widetilde{\mathcal{X}}} z_{\varepsilon,\delta}(s, \cdot)\|_{C_{\widetilde{\mathcal{X}}}^{\alpha}(\widetilde{B}_{\delta}(\xi_0))} \right),
\end{align*}
where $C$ is independent of $t_0,\varepsilon,\delta$, and also uniformly bounded with respect to $\xi_{0}$ since $\widetilde{V}(t,\xi) \in L^{\infty}([0,T] \times \mathbb{T}^3)$.

We exploit the fact that $D_{\widetilde{\mathcal{X}}} z_{\varepsilon,\delta}$, $\widetilde{X}_h D_{\widetilde{\mathcal{X}}} z_{\varepsilon,\delta}$ and $\widetilde{V}^0-\widetilde{V}$ vanish at some point of $\widetilde{B}_{\delta}(\xi_0)$, then from \cite[Proposition 4.2(i)]{07BB}, we easily deduce that
\begin{align*}
& \left[\left(\widetilde{V}^0-\widetilde{V}\right) \cdot D_{\widetilde{\mathcal{X}}} z_{\varepsilon,\delta}(s,\cdot)\right]_
{C_{\widetilde{\mathcal{X}}}^{\alpha}(\widetilde{B}_{\delta}(\xi_0))}\\
\leq & \delta^{\alpha} C \left[\widetilde{V}(s,\cdot)\right]_{C_{\widetilde{\mathcal{X}}}^{\alpha} (\widetilde{B}_{\delta}(\xi_0))} \left[D_{\widetilde{\mathcal{X}}} z_{\varepsilon,\delta}(s, \cdot)\right]_{C_{\widetilde{\mathcal{X}}}^{\alpha} (\widetilde{B}_{\delta}(\xi_0))},
\end{align*}
\begin{align*}
& \left[\left(\widetilde{V}^0-\widetilde{V}\right) \cdot \widetilde{X}_h D_{\widetilde{\mathcal{X}}} z_{\varepsilon,\delta}(s, \cdot)\right]_{C_{\widetilde{\mathcal{X}}}^{\alpha}(\widetilde{B}_{\delta}(\xi_0))}\\
\leq & \delta^{\alpha} C \left[\widetilde{V}(s,\cdot)\right]_{C_{\widetilde{\mathcal{X}}}^{\alpha} (\widetilde{B}_{\delta}(\xi_0))} \left[\widetilde{X}_h D_{\widetilde{\mathcal{X}}} z_{\varepsilon,\delta}(s,\cdot)\right]_{C_{\widetilde{\mathcal{X}}}^{\alpha} (\widetilde{B}_{\delta}(\xi_0))}.
\end{align*}
While obviously,
$$
\|\left(\widetilde{V}^0-\widetilde{V}\right) \cdot D_{\widetilde{\mathcal{X}}} z_{\varepsilon,\delta}(s,\cdot)\|_\infty \leq \delta^{\alpha} C \left[\widetilde{V}(s,\cdot)\right]_{C_{\widetilde{\mathcal{X}}}^{\alpha} (\widetilde{B}_{\delta}(\xi_0))} \|D_{\widetilde{\mathcal{X}}} z_{\varepsilon,\delta}(s,\cdot)\|_\infty,
$$
$$
\|\left(\widetilde{V}^0-\widetilde{V}\right) \cdot \widetilde{X}_h D_{\widetilde{\mathcal{X}}} z_{\varepsilon,\delta}(s, \cdot)\|_\infty \leq \delta^{\alpha} C \left[\widetilde{V}(s,\cdot)\right]_{C_{\widetilde{\mathcal{X}}}^{\alpha} (\widetilde{B}_{\delta}(\xi_0))} \|\widetilde{X}_h D_{\widetilde{\mathcal{X}}} z_{\varepsilon,\delta}(s, \cdot)\|_\infty.
$$
In addition to that, according to properties of the cutoff function (cf. \cite[Lemma 6.2]{07BB}), it is known that for $\delta$ small enough,
\begin{align*}
  \|\widetilde{f} \psi_{\varepsilon,\delta}(s,\cdot)\|_{C_{\widetilde{\mathcal{X}}}^{1+\alpha} (\widetilde{B}_{\delta}(\xi_0))} \leq & \delta^{-2} C \|\widetilde{f}(s,\cdot)\|_{C_{\widetilde{\mathcal{X}}}^{1+\alpha} (\widetilde{B}_{\delta}(\xi_0))},\\
  \|\widetilde{z}_T \psi_{\varepsilon,\delta}(T,\cdot)\|_{C_{\widetilde{\mathcal{X}}}^{2+\alpha} (\widetilde{B}_{\delta}(\xi_0))} \leq & \delta^{-3} C \|\widetilde{z}_T\|_{C_{\widetilde{\mathcal{X}}}^{2+\alpha} (\widetilde{B}_{\delta}(\xi_0))},\\
  \|\widetilde{g}_{\varepsilon,\delta}(s,\cdot)\|_ {C_{\widetilde{\mathcal{X}}}^{1+\alpha} (\widetilde{B}_{\delta}(\xi_0))} \leq & \left(\varepsilon^{-2} \delta^{-2}+\delta^{-4}\right) C \|\widetilde{z}(s,\cdot)\|_{C_{\widetilde{\mathcal{X}}}^{1+\alpha} (\widetilde{B}_{\delta}(\xi_0))}\\
  & + \delta^{-3} C \|\widetilde{z}(s,\cdot)\|_{C_{\widetilde{\mathcal{X}}}^{2+\alpha} (\widetilde{B}_{\delta}(\xi_0))}\\
  & + \delta^{-3} C \|\widetilde{V}(s,\cdot)\|_{C_{\widetilde{\mathcal{X}}}^{1+\alpha} (\widetilde{B}_{\delta}(\xi_0))} \|\widetilde{z}(s,\cdot)\|_{C_{\widetilde{\mathcal{X}}}^{1+\alpha} (\widetilde{B}_{\delta}(\xi_0))}.
\end{align*}
Then we can now go on to obtain that, for any $t \in [0,T)$,
\begin{align*}
& \|z_{\varepsilon,\delta}(t,\cdot)\|_ {C_{\widetilde{\mathcal{X}}}^{2+\alpha}(\widetilde{B}_{\delta}(\xi_0))}\\
\leq & \varepsilon^{1-\frac{\lambda}{2}} \delta^{-2} C \sup _{s \in [t,T]} \|\widetilde{f}(s,\cdot)\|_{C_{\widetilde{\mathcal{X}}}^{1+\alpha}(\widetilde{B}_{\delta}(\xi_0))}+ \delta^{-3} C\|\widetilde{z}_T\|_{C_{\widetilde{\mathcal{X}}}^{2+\alpha}(\widetilde{B}_{\delta}(\xi_0))}\\
& + \varepsilon^{1-\frac{\lambda}{2}} C \left( \varepsilon^{-2} \delta^{-2} \sup _{s \in [t,T]} \|\widetilde{z}(s,\cdot)\|_{C_{\widetilde{\mathcal{X}}}^{1+\alpha} (\widetilde{B}_{\delta}(\xi_0))}+ \delta^{-4}\sup _{s \in [t,T]}\|\widetilde{z}(s,\cdot)\|_{C_{\widetilde{\mathcal{X}}}^{2+\alpha} (\widetilde{B}_{\delta}(\xi_0))} \right.\\
&\left. + \delta^{-3} \sup _{s \in [t,T]} \|\widetilde{V}(s,\cdot)\|_{C_{\widetilde{\mathcal{X}}}^{1+\alpha} (\widetilde{B}_{\delta}(\xi_0))} \sup _{s \in [t,T]} \|\widetilde{z}(s,\cdot)\|_{C_{\widetilde{\mathcal{X}}}^{1+\alpha} (\widetilde{B}_{\delta}(\xi_0))} \right)\\
& + \varepsilon^{1-\frac{\lambda}{2}} \delta^{\alpha-3} C \sup _{s \in [t,T]} \|\widetilde{V}(s,\cdot)\|_ {C_{\widetilde{\mathcal{X}}}^{\alpha}(\widetilde{B}_{\delta}(\xi_0))} \sup_{s \in [t,T]} \|\widetilde{z}(s,\cdot)\|_ {C_{\widetilde{\mathcal{X}}}^{2+\alpha}(\widetilde{B}_{\delta}(\xi_0))}\\
& + \varepsilon^{1-\frac{\lambda}{2}} \delta^{-2} C \sup_{s \in [t,T]} \|\widetilde{V}(s, \cdot)\|_{C_{\widetilde{\mathcal{X}}}^{1+\alpha}(\widetilde{B}_{\delta}(\xi_0))} \sup_{s \in [t,T]} \| \widetilde{z}(s,\cdot)\|_{C_{\widetilde{\mathcal{X}}}^{1+\alpha}(\widetilde{B}_{\delta}(\xi_0))},
\end{align*}
where $C$ is independent of $t$, $\varepsilon$, and $\delta$. Recalling that $z_{\varepsilon,\delta} \equiv \widetilde{z}$ in $(-\frac{\varepsilon}{2}+t,\frac{\varepsilon}{2}+t) \times \widetilde{B}_{\frac{\delta}{2}}(\xi_0)$, set $\varepsilon^{1-\frac{\lambda}{2}}=\delta^{5}$, we have
\begin{equation}\label{key inequality}
\begin{aligned}
  \|\widetilde{z}(t,\cdot)\|_{C_{\widetilde{\mathcal{X}}}^{2+\alpha}(\widetilde{B}_{\frac{\delta}{2}}(\xi_0))} \leq & \delta^{3} C \sup _{s \in [t,T]} \|\widetilde{f}(s,\cdot)\|_{C_{\widetilde{\mathcal{X}}}^{1+\alpha}(\widetilde{B}_{\delta}(\xi_0))} +\delta^{-3}C\|\widetilde{z}_T\|_{C_{\widetilde{\mathcal{X}}}^{2+\alpha}(\widetilde{B}_{\delta}(\xi_0))}\\
  & + \delta^{-17} C \sup _{s \in [t,T]} \|\widetilde{z}(s,\cdot)\|_ {C_{\widetilde{\mathcal{X}}}^{1+\alpha}(\widetilde{B}_{\delta}(\xi_0))}+\widetilde{A}\\
  =: & I_1+I_2+I_3+\widetilde{A}.
\end{aligned}
\end{equation}
where
\begin{align*}
\widetilde{A} := & C \left( \delta \sup _{s \in [t,T]} \|\widetilde{z}(s,\cdot)\|_ {C_{\widetilde{\mathcal{X}}}^{2+\alpha}(\widetilde{B}_{\delta}(\xi_0))}\right.\\
& +\delta^{2} \sup _{s \in [t,T]} \|\widetilde{V}(s, \cdot)\|_{C_{\widetilde{\mathcal{X}}}^{1+\alpha}(\widetilde{B}_{\delta}(\xi_0))} \sup _{s \in [t,T]} \|\widetilde{z}(s,\cdot)\|_{C_{\widetilde{\mathcal{X}}}^{1+\alpha}(\widetilde{B}_{\delta}(\xi_0))}\\
& \left. + \delta^{\alpha+2} \sup _{s \in [t,T]} \|\widetilde{V}(s, \cdot)\|_{C_{\widetilde{\mathcal{X}}}^{\alpha}(\widetilde{B}_{\delta}(\xi_0))} \sup _{s \in [t,T]} \|\widetilde{z}(s,\cdot)\|_{C_{\widetilde{\mathcal{X}}}^{2+\alpha}(\widetilde{B}_{\delta}(\xi_0))}\right).
\end{align*}
Compare to the interpolation inequality in \cite[Theorem 7.4]{07BB}, we can prove that, for any $s \in [0,T]$, there exist positive constants $c$, $r$ and $\gamma$ such that for any $\sigma > 0$ small enough, the following inequality holds
$$
\|\widetilde{z}(s,\cdot)\|_ {C_{\widetilde{\mathcal{X}}}^{1+\alpha}(\widetilde{B}_{\delta}(\xi_0))} \leq \sigma \|\widetilde{z}(s,\cdot)\|_ {C_{\widetilde{\mathcal{X}}}^{2+\alpha}(\widetilde{B}_r(\xi_0))} + \frac{c}{\sigma^{\gamma}(r-\delta)^{2\gamma}}\|\widetilde{z}(s,\cdot)\|_ {L^{\infty}(\widetilde{B}_r(\xi_0))}.
$$
Combining with the equivalent relation between $z$ and $\widetilde{z}$ regarding the H\"{o}lder norm (cf. \cite[Proposition 8.3]{07BB}), we immediately get
\begin{align*}
  &\sup _{t \in [0,T]} \|z(t, \cdot)\|_{C_{\mathcal{X}}^{2+\alpha}(B_{\frac{\delta}{2}}(x_0))}\\
   \leq & C \sup _{t \in [0,T]} \|\widetilde{z}(t,\cdot)\|_ {C_{\widetilde{\mathcal{X}}}^{2+\alpha}(\widetilde{B}_{\frac{\delta}{2}}(\xi_0))}\\ \leq & C \left( \sup _{t \in [0,T]} \|\widetilde{f}(t,\cdot)\|_{C_{\widetilde{\mathcal{X}}}^{1+\alpha}(\widetilde{B}_{\delta}(\xi_0))} +\|\widetilde{z}_T\|_{C_{\widetilde{\mathcal{X}}}^{2+\alpha}(\widetilde{B}_{\delta}(\xi_0))}\right.\\
  & + \sigma \sup _{t \in [0,T]} \| \widetilde{z}(t,\cdot)\|_ {C_{\widetilde{\mathcal{X}}}^{2+\alpha}(\widetilde{B}_{r}(\xi_0))}\\
  &\left. + \frac{c}{\sigma^{\gamma}(r-\delta)^{2\gamma}} \|\widetilde{z}\|_{L^\infty([0,T] \times \widetilde{B}_{r}(\xi_0))} \right)+\widetilde{A}\\
  \leq & C \left( \sup _{t \in [0,T]} \|f(t,\cdot)\|_{C_{\mathcal{X}}^{1+\alpha}(B_{\delta}(x_0))} +\|z_T\|_{C_{\mathcal{X}}^{2+\alpha}(B_{\delta}(x_0))}\right.\\
  & + \sigma \sup _{t \in [0,T]} \|z(t,\cdot)\|_
  {C_{\mathcal{X}}^{2+\alpha}(B_{r}(x_0))}\\
  &\left. + \frac{c}{\sigma^{\gamma}(r-\delta)^{2\gamma}}
  \|z\|_{L^\infty([0,T] \times B_{r}(x_0))} \right)+A,
\end{align*}
where
\begin{align*}
A:= & C \left( \delta \sup _{t \in [0,T]} \|z(t,\cdot)\|_ {C_{\mathcal{X}}^{2+\alpha}(B_{\delta}(x_0))}\right.\\
& +\delta^{2} \sup _{t \in [0,T]} \|V(t, \cdot)\|_{C_{\mathcal{X}}^{1+\alpha}(B_{\delta}(x_0))} \sup _{t \in [0,T]} \|z(t,\cdot)\|_{C_{\mathcal{X}}^{1+\alpha}(B_{\delta}(x_0))}\\
& \left. + \delta^{\alpha+2} \sup _{t \in [0,T]} \|V(t,\cdot)\|_
{C_{\mathcal{X}}^{\alpha}(B_{\delta}(x_0))} \sup _{t \in [0,T]} \|z(t,\cdot)\|_{C_{\mathcal{X}}^{2+\alpha}(B_{\delta}(x_0))}\right).
\end{align*}
Since $\mathbb{T}^2$ is compact, there exists a finite covering $(B_{\frac{\delta}{4}}^E(x_i))_{1 \leq i \leq k}$ of $\mathbb{T}^2$, where $(x_i)_{1 \leq i \leq k} \subset \mathbb{T}^2$, note that $\bigcup\limits_{i=1}^k B_{\frac{\delta}{4}}(x_i) \supset \bigcup\limits_{i=1}^k B_{\frac{\delta}{4}}^E(x_i)$, and by \cite[Proposition 4.2(v)]{07BB}, we have
\begin{align*}
  \sup _{t \in [0,T]} \|z(t, \cdot)\|_{C_{\mathcal{X}}^{2+\alpha}(\mathbb{T}^2)} \leq & \sup _{t \in [0,T]} \|z(t,\cdot)\|_
  {C_{\mathcal{X}}^{2+\alpha}\left(\bigcup\limits_{i=1}^k B_{\frac{\delta}{4}}(x_i)\right)}\\
  \leq & C \sup _{t \in [0,T]} \sum_{i=1}^k \|z(t,\cdot)\|_
  {C_{\mathcal{X}}^{2+\alpha}\left(B_{\frac{\delta}{2}}(x_i)\right)}\\
  \leq & C \left( \sup _{t \in [0,T]} \|f(t,\cdot)\|_{C_{\mathcal{X}}^{1+\alpha}(\mathbb{T}^2)} +\|z_T\|_{C_{\mathcal{X}}^{2+\alpha}(\mathbb{T}^2)}\right.\\
  & + \sigma \sup _{t \in [0,T]} \|z(t,\cdot)\|_
  {C_{\mathcal{X}}^{2+\alpha}(\mathbb{T}^2)}\\
  &\left. + \frac{c}{\sigma^{\gamma}(r-\delta)^{2\gamma}}
  \|z\|_{L^\infty([0,T] \times \mathbb{T}^2)} \right)+A^*.
\end{align*}
where
\begin{align*}
A^*:= & C \left( \delta \sup _{t \in [0,T]} \|z(t,\cdot)\|_ {C_{\mathcal{X}}^{2+\alpha}(\mathbb{T}^2)}\right.\\
& +\delta^{2} \sup _{t \in [0,T]} \|V(t, \cdot)\|_{C_{\mathcal{X}}^{1+\alpha}(\mathbb{T}^2)} \sup _{t \in [0,T]} \|z(t,\cdot)\|_{C_{\mathcal{X}}^{1+\alpha}(\mathbb{T}^2)}\\
& \left. + \delta^{\alpha+2} \sup _{t \in [0,T]} \|V(t,\cdot)\|_
{C_{\mathcal{X}}^{\alpha}(\mathbb{T}^2)} \sup _{t \in [0,T]} \|z(t,\cdot)\|_{C_{\mathcal{X}}^{2+\alpha}(\mathbb{T}^2)}\right).
\end{align*}
We can choose $\sigma >0$, $\delta >0$ small enough to get
\begin{align*}
& \sup_{t \in [0,T]} \|z(t,\cdot)\|_{C_{\mathcal{X}}^{2+\alpha}(\mathbb{T}^2)}\\
\leq & C \left(\sup_{t \in [0,T]}\|f(t,\cdot)\|_{C_{\mathcal{X}}^{1+\alpha}(\mathbb{T}^2)}+ \|z_T\|_{C_{\mathcal{X}}^{2+\alpha}(\mathbb{T}^2)}+ \|z\|_{L^\infty([0,T] \times \mathbb{T}^2)}\right)
\end{align*}
with $C$ depends on $\sup\limits_{t \in [0,T]} \|V(t,\cdot)\|_{C_{\mathcal{X}}^{1+\alpha}(\mathbb{T}^2)}$. Then combining with the estimate of $\|z\|_{L^\infty([0,T] \times \mathbb{T}^2)}$, we finally get the estimate \eqref{prior USE}.

It remains to prove the time estimate \eqref{prior HSE}. For simplicity, here we only show that the conclusion holds when the coefficient $V$ is frozen, since the unfreezing procedure is similar as before.

Set $t'=t+h$, when $0 \leqslant t+2 h \leqslant T$, we have from Lemma \ref{Lem_LHJB Duhamel formula} that
\begin{align*}
& \widetilde{z}(t+h,\xi)-\widetilde{z}(t,\xi)\\
= & \int_{\mathbb{T}^3}\left(\Gamma_0(t+h,T,\Theta(\xi,\eta))-\Gamma_0(t,T,\Theta(\xi,\eta))\right) \widetilde{z}_T(\eta)d \eta \\
& + \int_{t+2h}^{T}\int_{\mathbb{T}^3}\left(\Gamma_0(t+h,s,\Theta(\xi,\eta))-\Gamma_0(t,s,\Theta(\xi,\eta))\right) \widetilde{f}(s,\eta)d \eta d s \\
& + \int_{t+h}^{t+2h}\int_{\mathbb{T}^3}\Gamma_0(t+h,s,\Theta(\xi,\eta)) \widetilde{f}(s,\eta)d \eta d s\\
& - \int_{t}^{t+2h}\int_{\mathbb{T}^3}\Gamma_0(t,s,\Theta(\xi,\eta)) \widetilde{f}(s,\eta)d \eta d s\\
= : & A_1+A_2+A_3+A_4.
\end{align*}
Since $\Gamma_0$ is $C^1$ with respect to the time $t$, for any $s \in[0, T]$, thus by Lagrange mean value theorem, there exists a $t^* \in (t,t+h)$ such that
$$
\Gamma_0\left(t+h,s,\Theta(\xi,\eta)\right)-\Gamma_0\left(t,s, \Theta(\xi,\eta)\right)= \partial_t\Gamma_0\left(t^*,s,\Theta(\xi,\eta)\right)h.
$$
Recall that
$$
|\partial_t\Gamma_0\left(t^*,s,\Theta(\xi,\eta)\right)| \leq \frac{c_1}{s-t^*} \frac{e^{-\frac{c_2\|\eta^{-1} \circ \xi\|}{s-t^*}}}{(s-t^*)^{\frac{Q}{2}}},
$$
and it can be checked that $\partial_t\Gamma_0\left(t^*,s,\Theta(\xi,\eta)\right)$ fulfills the three conditions for $t^*<s$ in Lemma \ref{FS type 0}, specifically that it is a frozen kernel of type 0 with constant $c$ replaced by $\frac{c}{s-t^*}$. Hence, from Lemma \ref{HCSIO}, we have for any $0\leq t+h \leq T' <T$ that
$$
\|A_1\|_{C_{\widetilde{\mathcal{X}}}^{2+\alpha}\left(\mathbb{T}^{3}\right)} \leq \frac{Ch}{T-t-h} \|\widetilde{z}_T\|_{C_{\widetilde{\mathcal{X}}}^{2+\alpha}\left(\mathbb{T}^{3}\right)} \leq C(T') \|\widetilde{z}_T\|_{C_{\widetilde{\mathcal{X}}}^{2+\alpha}\left(\mathbb{T}^{3}\right)}h^{\frac{1}{2}},
$$
\begin{align}\label{A_2}
  \|A_2\|_{C_{\widetilde{\mathcal{X}}}^{2+\alpha}\left(\mathbb{T}^{3}\right)} \leq & \int_{t+2h}^{T}\frac{Ch}{(s-t-h)^{\frac{3}{2}}}ds \sup_{t \in [0,T]}\|\widetilde{f}(t,\cdot)\|_{C_{\widetilde{\mathcal{X}}}^{1+\alpha}\left(\mathbb{T}^{3}\right)}\\
  \leq & C\sup_{t \in [0,T]}\|\widetilde{f}(t,\cdot)\|_{C_{\widetilde{\mathcal{X}}}^{1+\alpha}\left(\mathbb{T}^{3}\right)}h^{\frac{1}{2}}.\nonumber
\end{align}
Same as before, from Lemma \ref{HCFIO}, we have
\begin{align*}
&\|A_3+A_4\|_{C_{\widetilde{\mathcal{X}}}^{2+\alpha}\left(\mathbb{T}^{3}\right)}\\
\leq & \left(\int_{t+h}^{t+2h}\frac{C}{(s-t-h)^{\frac{\lambda}{2}}}ds+ \int_{t}^{t+2h}\frac{C}{(s-t)^{\frac{\lambda}{2}}}ds\right) \sup_{t \in [0,T]}\|\widetilde{f}(t,\cdot)\|_{C_{\widetilde{\mathcal{X}}}^{1+\alpha}\left(\mathbb{T}^{3}\right)}\\
\leq & C\sup_{t \in [0,T]}\|\widetilde{f}(t,\cdot)\|_{C_{\widetilde{\mathcal{X}}}^{1+\alpha}\left(\mathbb{T}^{3}\right)}h^{1-\frac{\lambda}{2}}, \text{ for any } \lambda \in (1,2).
\end{align*}
Combining with the above estimates, we can obtain that
\begin{align*}
& \sup_{h>0} \frac{\|\widetilde{z}(t+h,\xi)-\widetilde{z}(t,\xi)\|_{C_{\widetilde{\mathcal{X}}}^{2+\alpha}\left(\mathbb{T}^{3}\right)}}{h^\beta}\\
\leq & C(T') \left(\|\widetilde{z}_T\|_{C_{\widetilde{\mathcal{X}}}^{2+\alpha}\left(\mathbb{T}^{3}\right)}+ \sup_{t \in [0,T]}\|\widetilde{f}(t,\cdot)\|_{C_{\widetilde{\mathcal{X}}}^{1+\alpha}\left(\mathbb{T}^{3}\right)}\right)
\end{align*}
for any $\beta \in (0,\frac{1}{2})$, $t \leq T' < T$. Again by \cite[Proposition 8.3]{07BB}, we get the final result.

When $2 h>T-t$, there is no need to consider the integral from $t+2h$ to $T$ in the above formula \eqref{A_2}, and the result follows in the same way.
\end{proof}
\begin{Cor}
Let $V \in C^{1,1+\alpha}_{\mathcal{X}}\left([0,T]\times\mathbb{T}^2\right)$, $f \in C\left([0,T];C^{\alpha}_{\mathcal{X}}\left(\mathbb{T}^2\right)\right)$ and $z_T \in C^{2+\alpha}_{\mathcal{X}}\left(\mathbb{T}^2\right)$. Suppose that $z$ is a solution to equation \eqref{LHJB} in $[0, T] \times \mathbb{T}^{2}$. Then
\begin{enumerate}
\item $z$ satisfies
\begin{equation}\label{LHJB_C^1 estimate}
\sup _{t \in[0, T]}\|z(t, \cdot)\|_{1} \leqslant C\left(\left\|z_T\right\|_{1}+\|f\|_\infty\right)
\end{equation}
for some positive constant $C$, dependent of $\|V\|_\infty$ only.
\item $z$ satisfies
\begin{equation}\label{LHJB_USE_first order}
\sup _{t \in[0, T]}\|z(t, \cdot)\|_{1+\alpha} \leqslant C\left(\left\|z_T\right\|_{1+\alpha}+\sup _{t \in[0, T]}\|f(t, \cdot)\|_{\alpha}\right)
\end{equation}
for some positive constant $C$, dependent of $\sup\limits _{t \in[0, T]}\|V(t, \cdot)\|_{\alpha}$ and $\alpha$ only.
\item For any $\beta \in \left(0,\frac{1}{2}\right)$, $z$ satisfies
\begin{equation}\label{LHJB_HSE_lower order}
\sup _{\substack{t \neq t'\\t,t' \in [0,T]}} \frac{\left\|z\left(t^{\prime}, \cdot\right)-z(t, \cdot)\right\|_{\alpha}}{\left|t^{\prime}-t\right|^\beta} \leqslant C\left(\left\|z_T\right\|_{1+\alpha}+\sup _{t \in[0, T]}\|f(t, \cdot)\|_{\alpha}\right)
\end{equation}
for some positive constant $C$, dependent of $\sup\limits_{t \in[0, T]}\|V(t, \cdot)\|_{1+\alpha}$ and $\alpha$ only.
\end{enumerate}
\end{Cor}
\begin{proof}
We can prove (1) and (2) by repeating the method for proving estimate \eqref{prior USE}. To this aim, we only need to notice that when proving (1), the term $I_3$ in the key inequality \eqref{key inequality} turns into $\delta^{-3}C\|\widetilde{z}\|_\infty$, which can be controlled by \eqref{MaxE}. As for (2), $I_3$ becomes $\delta^{-9} C \sup _{s \in [t,T]} \|\widetilde{z}(s,\cdot)\|_ {C_{\widetilde{\mathcal{X}}}^{\alpha}(\widetilde{B}_{\delta}(\xi_0))}$. This H\"{o}lder estimate can be removed by \eqref{LHJB_C^1 estimate}.

Now, let us prove (3). Set $\widetilde{v}:=\widetilde{z}-\widetilde{z}_T$, then $\widetilde{v}$ is the solution to the problem
\begin{equation*}
\begin{cases}
\widetilde{H}_0\widetilde{v}=\Delta_{\widetilde{\mathcal{X}}}\widetilde{z}_T +\left(\widetilde{V}^0-\widetilde{V}\right)\cdot D_{\widetilde{\mathcal{X}}}\widetilde{z}-\widetilde{V}^0\cdot D_{\widetilde{\mathcal{X}}}\widetilde{z}_T+\widetilde{f}, & \text {in }[0, T] \times \mathbb{T}^3, \\
\widetilde{v}(T,\xi)=0, & \text {in } \mathbb{T}^3.
\end{cases}
\end{equation*}
By Lemma \ref{Lem_LHJB Duhamel formula} and integrating by parts, we have
\begin{align*}
&\widetilde{v}(t,\xi) \\
= & \int_{t}^{T}\int_{\mathbb{T}^3}\Gamma_0(t,s,\Theta(\xi,\eta))\left(\Delta_{\widetilde{\mathcal{X}}}\widetilde{z}_T +\left(\widetilde{V}^0-\widetilde{V}\right)\cdot D_{\widetilde{\mathcal{X}}}\widetilde{z}-\widetilde{V}^0\cdot D_{\widetilde{\mathcal{X}}}\widetilde{z}_T+\widetilde{f}\right)(s,\eta)d \eta d s \\
= &\int_{t}^{T}\int_{\mathbb{T}^3}\sum_{i=1}^{2}\widetilde{X}_i\Gamma_0(t,s,\Theta(\xi,\cdot))(\eta)\left(-\widetilde{X}_i\widetilde{z}_T +\left(\widetilde{V}_i-\widetilde{V}_i^0\right)\widetilde{z}\right)(s,\eta)d \eta d s \\
& +\int_{t}^{T}\int_{\mathbb{T}^3}\Gamma_0(t,s,\Theta(\xi,\eta))\left(\widetilde{X}_i\widetilde{V}_i\cdot\widetilde{z}-\widetilde{V}^0\cdot D_{\widetilde{\mathcal{X}}}\widetilde{z}_T+\widetilde{f}\right)(s,\eta)d \eta d s.
\end{align*}
The remaining steps are similar to the proof of estimates \eqref{prior HSE} so we omit here. Finally, we can obtain that
\begin{align*}
& \sup_{h>0} \frac{\|v(t+h,x)-v(t,x)\|_{C_{\mathcal{X}}^{\alpha}\left(\mathbb{T}^{2}\right)}}{h^\beta}\\
\leq & C \left(\|z_T\|_{C_{\mathcal{X}}^{1+\alpha}\left(\mathbb{T}^{2}\right)}+ \sup_{t \in [0,T]}\|f(t,\cdot)\|_{C_{\mathcal{X}}^{\alpha}\left(\mathbb{T}^{2}\right)}+\sup_{t \in [0,T]}\|z(t,\cdot)\|_{C_{\mathcal{X}}^{\alpha}\left(\mathbb{T}^{2}\right)}\right)
\end{align*}
for any $\beta \in \left(0,\frac{1}{2}\right)$, $t\leq T$, where $C$ depends on $\sup\limits_{t \in [0,T]}\|V(t,\cdot)\|_{1+\alpha}$ and $\alpha$ only. Combining the above estimate with \eqref{LHJB_C^1 estimate}, we get the conclusion.
\end{proof}
We are now in a position to prove Lemma \ref{Lem_LHJB wellposed regularity}.
\begin{proof}[Proof of Lemma \ref{Lem_LHJB wellposed regularity}]
The existence of the solutions can found in \cite[Theorem 12.1]{10BBLU}, while the uniqueness can be assured by the weak maximum principle in \cite[Theorem 13.1]{10BBLU}. Let $V^n, f^n, z_T^n$ be the standard (Euclidean) mollified versions of $V, f$ and $z_T$. One can easily check that $V^n \rightarrow V$ in $C_{\mathcal{X}}^{1,1+\alpha}\left([0, T] \times \mathbb{T}^{2} \right)$, $f^n \rightarrow f$ in $C \left([0, T]; C_{\mathcal{X}}^{1+\alpha}\left(\mathbb{T}^2\right)\right)$, and $z_T^n \rightarrow z_T$ in $C_{\mathcal{X}}^{2+\alpha}\left(\mathbb{T}^2\right)$ as $n \rightarrow \infty$. Set
$$
H^n=-\partial_t-\Delta_{\mathcal{X}}+V^n(t, x)\cdot D_{\mathcal{X}}.
$$
Since $H^n$ is hypoelliptic for every $n \in \mathbb{N}$, the solution $z^n$ belongs to $C^{\infty}\left([0,T] \times \mathbb{T}^2\right)$, then we can apply our a-priori estimate \eqref{prior USE} in Lemma \ref{Lem_LHJB prior estimate}, writing
\begin{equation}\label{z_n_Schauder estimate}
\begin{aligned}
\sup _{t \in [0,T]} \left\|z^n(t, \cdot)\right\|_{2+\alpha}
\leq & C\left(\sup _{t \in [0,T]} \left\|f^n(t,\cdot)\right\|_{1+\alpha} +\left\|z_T^n\right\|_{2+\alpha}\right) \\
\leq & C\left(\sup _{t \in [0,T]} \left\|f(t,\cdot)\right\|_{1+\alpha} +\left\|z_T\right\|_{2+\alpha}\right),
\end{aligned}
\end{equation}
where C can be bounded independently of $n$.
Hence, by compactness lemma (cf. \cite[Lemma 14.3]{10BBLU}), for any $t \in [0,T]$, we can find a subsequence $n_k \rightarrow \infty$ and a function $v(t,\cdot) \in C_{\mathcal{X}}^{2+\alpha}\left(\mathbb{T}^2\right)$ such that
$X^J z^{n_k} \rightarrow X^J v$ uniformly in $\mathbb{T}^2$, for $|J|=0,1,2$. Moreover, since
$$
\partial_t z^{n_k}=-\Delta_{\mathcal{X}} z^{n_k}+ V^{n_k}(t,x)\cdot D_{\mathcal{X}} z^{n_k}-f^{n_k}(t,x),
$$
we combine \eqref{z_n_Schauder estimate} to obtain that, for any $k_1,k_2 \in \mathbb{N}$,
\begin{align*}
\left\|\partial_t z^{n_{k_1}}-\partial_t z^{n_{k_2}}\right\|_\infty \leq & \left\|\Delta_{\mathcal{X}}\left(z^{n_{k_1}}-z^{n_{k_2}}\right)\right\|_\infty
+\left\|\left(V^{n_{k_1}}-V^{n_{k_2}}\right)\cdot D_{\mathcal{X}}z^{n_{k_1}}\right\|_\infty \\
&+\left\|V^{n_{k_2}}\cdot\left(D_{\mathcal{X}}z^{n_{k_1}}-D_{\mathcal{X}}z^{n_{k_2}}\right)\right\|_\infty +\left\|f^{n_{k_1}}-f^{n_{k_2}}\right\|_\infty \\
\leq & C\bigg(\sup _{t \in [0,T]} \left\|\left(f^{n_{k_1}}-f^{n_{k_2}}\right)(t,\cdot)\right\|_{1+\alpha} +\left\|z_T^{n_{k_1}}-z_T^{n_{k_2}}\right\|_{2+\alpha} \\
& +\!\sup _{t \in [0,T]}\left\|\left(V^{n_{k_1}}-V^{n_{k_2}}\right)(t,\cdot)\right\|_{1+\alpha}\sup _{t \in [0,T]}\left\|D_{\mathcal{X}}z^{n_{k_1}}(t,\cdot)\right\|_{1+\alpha}\bigg).
\end{align*}
Let $k_1,k_2 \rightarrow \infty$, the right hand side of the above inequality tends to $0$. Thus we can find by Cauchy criterion for uniform convergence that $\partial_t z^{n_k}$ converges to $\partial_t v$ uniformly in $[0,T] \times \mathbb{T}^2$.

Consequently, $H^{n_k}z^{n_k}=f^{n_k} \rightarrow f$, while also $H^{n_k}z^{n_k} \rightarrow Hv$ and $z^{n_k}_T \rightarrow v_T$. Then we get that $v=z$ in $[0,T] \times \mathbb{T}^2$, by the uniqueness of the equation, noticing the fact that $z_T=v_T$. This will imply $z \in C([0,T];C_{\mathcal{X}}^{2+\alpha}(\mathbb{T}^2))$, and the desired regularity result. Beyond that, by the similar method, the time regularity can be obtained easily from estimate \eqref{prior HSE} in Lemma \ref{Lem_LHJB prior estimate}.
\end{proof}
Next we prove the result for the corresponding homogeneous equation, that is, we have the following problem
\begin{equation}\label{homo. LHJB}
\begin{cases}
-\partial_t z-\Delta_{\mathcal{X}} z+V(t, x) \cdot D_{\mathcal{X}} z=0,& \text {in }[0, T) \times \mathbb{T}^2, \\
z(T, x)=z_T(x),& \text {in }\mathbb{T}^2.
\end{cases}
\end{equation}
\begin{proof}[Proof of Lemma \ref{Lem_homo. LHJB Lipschitz}]
Since $d_{cc}(x,y) \leq C d_{\mathbb{T}^2}(x,y)^{\frac{1}{k}}$, we have that $z_T$ is continuous in $\mathbb{T}^2$. When $z_T$ is only $d_{cc}$-Lipschitz, we consider its standard (Euclidean) mollified version $z_T^n$. It can be easy to check that $z_T^n \rightarrow z_T$ as $n \rightarrow \infty$ in $C_{\mathcal{X}}^{0+1}(\mathbb{T}^2)$. For every $n \in \mathbb{N}$, there exists a unique solution $z^n \in C_{\mathcal{X}}^{1,2} \left((0, T) \times \mathbb{T}^{2} \right) \bigcap C\left([0, T] \times \mathbb{T}^2\right)$ to equation \eqref{homo. LHJB} with terminal condition $z_T^n$, and thanks to \cite[Theorem 14.4]{10BBLU}, $z^n \in C_{\mathcal{X}}^{1+\frac{\alpha}{2},2+\alpha}\left([\delta,T-\delta] \times \mathbb{T}^2\right)$ for any $\delta > 0$, $0<\alpha<1$. Since $C_{\mathcal{X}}^{1+\frac{\alpha}{2},2+\alpha}\left([\delta,T-\delta] \times \mathbb{T}^2\right)$ is compactly embedded into $C_{\mathcal{X}}^{1,2}\left([\delta,T-\delta] \times \mathbb{T}^2\right)$, this implies that (up to a subsequence) $z^n(t,x) \rightarrow z(t,x)$ uniformly in $C_{\mathcal{X}}^{1,2}\left([\delta,T-\delta] \times \mathbb{T}^2\right)$, as for the linearity and uniqueness of the equation. Let $\delta \rightarrow 0$, we then get $z^n(t,x) \rightarrow z(t,x)$ in $C_{\mathcal{X}}^{1,2}\left((0,T) \times \mathbb{T}^2\right)$. Moreover, $z^n(t,x) \rightarrow z(t,x)$ in $C\left([0,T] \times \mathbb{T}^2\right)$.

On the other hand, by \eqref{LHJB_C^1 estimate} we have,
$$
\sup_{t \in [0,T]}\left\|D_\mathcal{X}z^n(t,\cdot)\right\|_\infty \leq C\left\|z_T^n\right\|_{C_\mathcal{X}^1(\mathbb{T}^2)}.
$$

For any $x \in \mathbb{T}^2$, let $\gamma(t)$ be the absolutely continuous integral curve of the vector fields $\{X_1,X_2\}$ such that
$$
\begin{cases}
\gamma^{\prime}(t)=X_i(\gamma(t)), \\
\gamma(0)=x.
\end{cases}
$$
Then
$$
X_i z_T^n(x)=\left[\frac{d}{d t} z_T^n(\gamma(t))\right](0)=\lim_{t \rightarrow 0} \frac{z_T^n(\gamma(t))-z_T^n(\gamma(0))}{t},\quad i=1,2.
$$
Since $\gamma$ is subunit, we can write
$$
\left|z_T^n(\gamma(t))-z_T^n(\gamma(0))\right| \leqslant \left[z_T^n\right]_{C_\mathcal{X}^{0+1}(\mathbb{T}^2)} d_{cc}(\gamma(t), \gamma(0)) \leqslant \left\|z_T^n\right\|_{C_\mathcal{X}^{0+1}(\mathbb{T}^2)}t,
$$
thus
\begin{equation}\label{C^0+1 control C^1}
\left\|D_\mathcal{X}z_T^n\right\|_\infty \leq \left\|z_T^n\right\|_{C_\mathcal{X}^{0+1}(\mathbb{T}^2)}.
\end{equation}
Finally, we obtain from \cite[Proposition 4.2(ii)]{07BB} and \eqref{C^0+1 control C^1} that for any $t \in [0,T]$, $x,y \in \mathbb{T}^2$,
\begin{align*}
  \left|z^n(t,x)-z^n(t,y)\right| \leq & \left\|D_\mathcal{X}z^n\right\|_\infty d_{cc}(x,y) \leq C\left\|z_T^n\right\|_{C_\mathcal{X}^1(\mathbb{T}^2)} d_{cc}(x,y)\\
  \leq & C\left\|z_T^n\right\|_{C_\mathcal{X}^{0+1}(\mathbb{T}^2)} d_{cc}(x,y),
\end{align*}
Let $n \rightarrow \infty$, we have
$$
|z(t,x)-z(t,y)| \leq C\left\|z_T\right\|_{C_\mathcal{X}^{0+1}(\mathbb{T}^2)} d_{cc}(x,y)
$$
for any $x,y \in \mathbb{T}^2$. This completes the proof.
\end{proof}
We would like to give another generalized regularity result for the degenerate homogeneous equation with terminal condition having lower regularities, which is useful for the proof of the lemma in Section \ref{Sec_5}.
\begin{Lem}\label{Lem_homo. LHJB Schauder estimate}
Suppose $V(t,x) \in C^{\frac{\alpha}{2},1+\alpha}_{\mathcal{X}}\left([0,T]\times\mathbb{T}^2\right)$, and $z_T(x) \in C^{1+\alpha}_{\mathcal{X}}\left(\mathbb{T}^2\right)$. Then the unique solution $z$ of the equation \eqref{homo. LHJB} satisfies
$$
\left\|z\right\|_{\frac{\alpha}{2},1+\alpha} \leq C\left\|z_T\right\|_{1+\alpha},
$$
where $C$ depends on $\left\|V\right\|_{\frac{\alpha}{2},1+\alpha}$ and $\alpha$ only.
\end{Lem}
\begin{proof}
We still use the technique of ``lifting and approximation". The idea to prove the Schauder estimate is similar to Lemma \ref{Lem_LHJB prior estimate}. While the difference is that we apply the abstract theory of singular or fractional integrals to the homogeneous space $\left(\widetilde{B}_\delta^p\left(t_0, \xi_0\right), \widetilde{d}_p, dt d \xi\right)$ instead of $\left(\widetilde{B}_\delta\left(\xi_0\right), \widetilde{d}_{cc}, d \xi\right)$, where $d_p$ is the parabolic Carnot-Carath\'{e}odory distance defined in Section \ref{Sec_2}, and $\widetilde{B}_\delta^p\left(t_0, \xi_0\right):=[t_0-\frac{\delta^2}{2},t_0+\frac{\delta^2}{2}]\times\widetilde{B}_\delta\left(\xi_0\right)$ denotes a $\widetilde{d}_p$-ball.

For any $n \in \mathbb{N}$, let $z_T^n$ be the standard (Euclidean) mollified version of $z_T$, and $z^n$ is the solution to the equation \eqref{homo. LHJB} related to $z_T^n$. We now freeze $\widetilde{V}(t, \xi)$ at some point $\left(t_0, \xi_0\right)=\left(t_0,\left(x_0, 0\right)\right) \in [0, T] \times \mathbb{T}^3$, denoting as $\widetilde{V}^{(0,0)}:=\widetilde{V}\left(t_0, \xi_0\right)$. Consider the frozen lifted differential operator
$$
\widetilde{H}_{(0,0)}:=-\partial_t-\Delta_{\widetilde{\mathcal{X}}}+\widetilde{V}^{(0,0)} \cdot D_{\widetilde{\mathcal{X}}},
$$
and for any $\delta>0$, we set $w_\delta^n(t, \xi):=\left(\widetilde{z}^n(t, \xi)-\widetilde{z}_T^n(\xi)\right)\psi_\delta(t, \xi)$, where $\psi_\delta(t, \xi)$ is a cutoff function such that
$$
\widetilde{B}_\frac{\delta}{2}^p\left(t_0, \xi_0\right)\prec\psi_\delta\prec\widetilde{B}_\delta^p\left(t_0, \xi_0\right).
$$
Observing that $w_\delta^n$ is the solution to the problem
\begin{equation*}
\begin{cases}
\widetilde{H}_{(0,0)}w_\delta^n +\left(\widetilde{V}-\widetilde{V}^{(0,0)}\right)\cdot D_{\widetilde{\mathcal{X}}}w_\delta^n=\Delta_{\widetilde{\mathcal{X}}}\widetilde{z}_T^n\psi_\delta\\ \quad-\left(D_{\widetilde{\mathcal{X}}}\widetilde{z}^n-D_{\widetilde{\mathcal{X}}}\widetilde{z}_T^n\right)\cdot D_{\widetilde{\mathcal{X}}}\psi_\delta-g_\delta^n, & \mbox{in } [0,T]\times\mathbb{T}^3, \\
w_\delta^n(T,\xi)=0, & \mbox{in } \mathbb{T}^3,
\end{cases}
\end{equation*}
where
$$
g_\delta^n(t,\xi):=\left(\widetilde{z}^n-\widetilde{z}_T^n\right) \partial_t \psi_\delta+\left(\widetilde{z}^n-\widetilde{z}_T^n\right) \Delta_{\widetilde{\mathcal{X}}}\psi_\delta -\left(\widetilde{z}^n-\widetilde{z}_T^n\right)\widetilde{V} \cdot D_{\widetilde{\mathcal{X}}}\psi_\delta-\widetilde{V}\cdot D_{\widetilde{\mathcal{X}}} \widetilde{z}_T^n \psi_\delta.
$$
We can get from Lemma \ref{Lem_LHJB Duhamel formula} that
\begin{align*}
w_\delta^n= & -\int_{-\infty}^{T}\int_{\mathbb{T}^3}\Gamma_{(0,0)}\left(t-s,\Theta(\xi,\eta)\right) \left(\widetilde{V}(s,\eta)-\widetilde{V}^{(0,0)}\right)\cdot D_{\widetilde{\mathcal{X}}}w_\delta^n(s,\eta)d\eta ds \\
& +\int_{-\infty}^{T}\int_{\mathbb{T}^3}\Gamma_{(0,0)}\left(t-s,\Theta(\xi,\eta)\right) \Delta_{\widetilde{\mathcal{X}}}\widetilde{z}_T^n(\eta)\psi_\delta(s,\eta)d\eta ds \\
& -\int_{-\infty}^{T}\int_{\mathbb{T}^3}\Gamma_{(0,0)}\left(t-s,\Theta(\xi,\eta)\right) \left(D_{\widetilde{\mathcal{X}}}\widetilde{z}^n(s,\eta)-D_{\widetilde{\mathcal{X}}}\widetilde{z}_T^n(\eta)\right)\cdot D_{\widetilde{\mathcal{X}}}\psi_\delta(s,\eta)d\eta ds \\
& -\int_{-\infty}^{T}\int_{\mathbb{T}^3}\Gamma_{(0,0)}\left(t-s,\Theta(\xi,\eta)\right) g_\delta^n(s,\eta)d\eta ds,
\end{align*}
where $\Gamma_{(0,0)}$ vanishes for $s \leq t$ and is the fundamental solution for the corresponding approximating operator $\mathcal{H}_{(0,0)}$, a left invariant H\"{o}rmander's operator on $\mathbb{R} \times \mathbb{H}^1$. For $m \in \{0,1,2\}$, we define a kernel as
$$
k_m(t,s,\xi,\eta):=Y^J(\Gamma_{(0,0)}\left(t-s,\cdot\right))(\Theta(\xi,\eta))
$$
with $m$-order differential operator $Y^J:=Y_{j_1}Y_{j_2}\cdots Y_{j_m}$, for any multi-index $J=\left(j_1, j_2, \cdots, j_m\right), j_i \in\{1,2\}, i=0,1, \cdots,m$. Then we can find that $k_m$ is a kernel of type $2-m$ (cf. \cite[Proposition 6.3]{07BB}). Without abuse of notation, we denote a frozen integral operator of type $\lambda \geq 0$ as $T_{\lambda}(t_0,\xi_0)$.

Due to \eqref{approximation_1}, it follows by integrating by parts that
\begin{align*}
w_\delta^n= & -T_2(t_0,\xi_0)\left(\left(\widetilde{V}-\widetilde{V}^{(0,0)}\right)\cdot D_{\widetilde{\mathcal{X}}}w_\delta^n\right)(t,\xi)-T_1(t_0,\xi_0) \left(D_{\widetilde{\mathcal{X}}}\widetilde{z}_T^n\cdot D_{\widetilde{\mathcal{X}}}\psi_\delta\right)(t,\xi)\\
& +T_1(t_0,\xi_0)\left(\left(\widetilde{z}^n-\widetilde{z}_T^n\right)\cdot \Delta_{\widetilde{\mathcal{X}}}\psi_\delta\right)(t,\xi)-T_2(t_0,\xi_0)g_\delta^n(t,\xi).
\end{align*}
Then, we compute
\begin{align*}
\widetilde{X}_i w_\delta^n= & -T_1(t_0,\xi_0)\left(\left(\widetilde{V}-\widetilde{V}^{(0,0)}\right)\cdot D_{\widetilde{\mathcal{X}}}w_\delta^n\right)(t,\xi)-T_0(t_0,\xi_0) \left(D_{\widetilde{\mathcal{X}}}\widetilde{z}_T^n\cdot D_{\widetilde{\mathcal{X}}}\psi_\delta\right)(t,\xi)\\
& +T_0(t_0,\xi_0)\left(\left(\widetilde{z}^n-\widetilde{z}_T^n\right)\cdot \Delta_{\widetilde{\mathcal{X}}}\psi_\delta\right)(t,\xi)-T_1(t_0,\xi_0)g_\delta^n(t,\xi).
\end{align*}
We extend the domain of $\widetilde{z}^n$ as
\begin{equation*}
\widetilde{z}^n(t,\xi):=
\begin{cases}
\widetilde{z}^n(t,\xi), & \mbox{if } t \in [0,T], \\
\widetilde{z}^n(0,\xi), & \mbox{if } t \in (-\infty,0), \\
\widetilde{z}^n(T,\xi), & \mbox{if } t \in (T,+\infty),
\end{cases}
\end{equation*}
for any $\xi \in \mathbb{T}^3$. And the same is for $\widetilde{V}$.

Apply the H\"{o}lder continuity of singular or fractional integral operator to the homogeneous space $\left(\widetilde{B}_\delta^p\left(t_0,\xi_0\right), \widetilde{d}_p, d t d \xi\right)$, we can obtain
\begin{align*}
& \left\|w_\delta^n\right\|_ {C_{\widetilde{\mathcal{X}}}^{\frac{\alpha}{2},1+\alpha}\left(\widetilde{B}_\delta^p\left(t_0,\xi_0\right)\right)}\\ \leq & C\left(\left\|\left(\widetilde{V}-\widetilde{V}^{(0,0)}\right)\cdot D_{\widetilde{\mathcal{X}}}w_\delta^n\right\|_ {C_{\widetilde{\mathcal{X}}}^{\frac{\alpha}{2},\alpha}\left(\widetilde{B}_\delta^p\left(t_0,\xi_0\right)\right)}+\left\|D_{\widetilde{\mathcal{X}}}\widetilde{z}_T^n\cdot D_{\widetilde{\mathcal{X}}}\psi_\delta\right\|_ {C_{\widetilde{\mathcal{X}}}^{\frac{\alpha}{2},\alpha}\left(\widetilde{B}_\delta^p\left(t_0,\xi_0\right)\right)}\right.\\
& \left.+\left\|\left(\widetilde{z}^n-\widetilde{z}_T^n\right)\cdot \Delta_{\widetilde{\mathcal{X}}}\psi_\delta\right\|_ {C_{\widetilde{\mathcal{X}}}^{\frac{\alpha}{2},\alpha}\left(\widetilde{B}_\delta^p\left(t_0,\xi_0\right)\right)} +\left\|g_\delta^n\right\|_ {C_{\widetilde{\mathcal{X}}}^{\frac{\alpha}{2},\alpha}\left(\widetilde{B}_\delta^p\left(t_0,\xi_0\right)\right)}\right),
\end{align*}
where $C$ is independent of $\delta$, $n$ and $\left(t_0, \xi_0\right)$. Similar to the proof of Lemma \ref{Lem_LHJB prior estimate}, we use \cite[Proposition 4.2(i)]{07BB} and \cite[Lemma 6.2]{07BB} to further obtain
\begin{align*}
& \left\|w_\delta^n\right\|_ {C_{\widetilde{\mathcal{X}}}^{\frac{\alpha}{2},1+\alpha}\left(\widetilde{B}_\delta^p\left(t_0,\xi_0\right)\right)}\\ \leq & C\left(\delta^\alpha\left\|\widetilde{V}\right\|_ {C_{\widetilde{\mathcal{X}}}^{\frac{\alpha}{2},\alpha}\left(\widetilde{B}_\delta^p\left(t_0,\xi_0\right)\right)}\left\|D_{\widetilde{\mathcal{X}}}w_\delta^n\right\|_ {C_{\widetilde{\mathcal{X}}}^{\frac{\alpha}{2},\alpha}\left(\widetilde{B}_\delta^p\left(t_0,\xi_0\right)\right)}+\delta^{-2}\left\|D_{\widetilde{\mathcal{X}}}\widetilde{z}_T^n\right\|_ {C_{\widetilde{\mathcal{X}}}^{\alpha}\left(\widetilde{B}_\delta\left(\xi_0\right)\right)}\right.\\
& \left.+\delta^{-3}\left\|\widetilde{z}^n-\widetilde{z}_T^n\right\|_ {C_{\widetilde{\mathcal{X}}}^{\frac{\alpha}{2},\alpha}\left(\widetilde{B}_\delta^p\left(t_0,\xi_0\right)\right)} +\delta^{-2}\left\|\left(\widetilde{z}^n-\widetilde{z}_T^n\right)\widetilde{V}\right\|_ {C_{\widetilde{\mathcal{X}}}^{\frac{\alpha}{2},\alpha}\left(\widetilde{B}_\delta^p\left(t_0,\xi_0\right)\right)}\right.\\
& \left.+\delta^{-1}\left\|\widetilde{V}\cdot D_{\widetilde{\mathcal{X}}}\widetilde{z}_T^n\right\|_ {C_{\widetilde{\mathcal{X}}}^{\frac{\alpha}{2},\alpha}\left(\widetilde{B}_\delta^p\left(t_0,\xi_0\right)\right)}\right).
\end{align*}
Choosing $\delta$ small enough, and recalling that $w_\delta^n \equiv \widetilde{z}^n-\widetilde{z}_T^n$ in $\widetilde{B}_{\frac{\delta}{2}}^p\left(t_0, \xi_0\right)$, we have
\begin{equation}\label{first order Schauder estimate_local}
\begin{aligned}
& \left\|\widetilde{z}^n\right\|_ {C_{\widetilde{\mathcal{X}}}^{\frac{\alpha}{2},1+\alpha}\left(\widetilde{B}_{\frac{\delta}{2}}^p\left(t_0,\xi_0\right)\right)}\\ \leq & \left\|\widetilde{z}_T^n\right\|_ {C_{\widetilde{\mathcal{X}}}^{1+\alpha}\left(\widetilde{B}_{\frac{\delta}{2}}\left(\xi_0\right)\right)} +\left\|\widetilde{z}^n-\widetilde{z}_T^n\right\|_ {C_{\widetilde{\mathcal{X}}}^{\frac{\alpha}{2},1+\alpha}\left(\widetilde{B}_{\frac{\delta}{2}}^p\left(t_0,\xi_0\right)\right)}\\
\leq & C\left(\left\|\widetilde{z}_T^n\right\|_ {C_{\widetilde{\mathcal{X}}}^{1+\alpha}\left(\widetilde{B}_{\delta}\left(\xi_0\right)\right)} +\delta^{-3}\left\|\widetilde{z}^n\right\|_ {C_{\widetilde{\mathcal{X}}}^{\frac{\alpha}{2},\alpha}\left(\widetilde{B}_\delta^p\left(t_0,\xi_0\right)\right)}\right),
\end{aligned}
\end{equation}
where $C$ depends on $\left\|\widetilde{V}\right\|_ {C_{\widetilde{\mathcal{X}}}^{\frac{\alpha}{2},\alpha}\left(\widetilde{B}_\delta^p\left(t_0,\xi_0\right)\right)}$ and $\alpha$ only.

By \eqref{LHJB_HSE_lower order} we can get the estimate of $\left\|\widetilde{z}^n\right\|_ {C_{\widetilde{\mathcal{X}}}^{\frac{\alpha}{2},\alpha}\left(\widetilde{B}_{\delta}^p\left(t_0,\xi_0\right)\right)}$. Putting this into \eqref{first order Schauder estimate_local} and using the compactness of $[0,T]\times\mathbb{T}^3$, we obtain the result
\begin{equation}\label{first order Schauder estimate_lifted}
\left\|\widetilde{z}^n\right\|_ {C_{\widetilde{\mathcal{X}}}^{\frac{\alpha}{2},1+\alpha}\left([0,T]\times\mathbb{T}^3\right)} \leq C\left\|\widetilde{z}_T^n\right\|_ {C_{\widetilde{\mathcal{X}}}^{1+\alpha}\left(\mathbb{T}^3\right)},
\end{equation}
where $C$ depends on $\left\|\widetilde{V}\right\|_ {C_{\widetilde{\mathcal{X}}}^{\frac{\alpha}{2},1+\alpha}\left([0,T]\times\mathbb{T}^3\right)}$ and $\alpha$ only.

According to the above estimate and the linearity of equation \eqref{homo. LHJB}, we know that $\left\{\widetilde{z}^n\right\}_n$ is a Cauchy sequence in $C_{\widetilde{\mathcal{X}}}^{\frac{\alpha}{2},1+\alpha}\left([0,T]\times\mathbb{T}^3\right)$, thus it converges to some $\hat{z} \in C_{\widetilde{\mathcal{X}}}^{\frac{\alpha}{2},1+\alpha}\left([0,T]\times\mathbb{T}^3\right)$. Also, since for $t<T$,
\begin{align*}
& \hat{z}(t,\xi)=\lim_{n\rightarrow\infty}\widetilde{z}^n(t,\xi)\\
= & \lim_{n\rightarrow\infty}\int_{\mathbb{T}^3} \Gamma_{(0,0)}(t,T,\Theta(\xi,\eta))\widetilde{z}_T^n(\eta)d \eta \\
& - \lim_{n\rightarrow\infty}\int_{-\infty}^{T}\int_{\mathbb{T}^3}\Gamma_{(0,0)}(t,s,\Theta(\xi,\eta))\left(\widetilde{V}(s,\eta)-\widetilde{V}^{(0,0)}\right)\cdot D_{\widetilde{\mathcal{X}}}\widetilde{z}^n(s,\eta) d \eta d s\\
= & \int_{\mathbb{T}^3} \Gamma_{(0,0)}(t,T,\Theta(\xi,\eta))\widetilde{z}_T(\eta)d \eta \\
& - \int_{-\infty}^{T}\int_{\mathbb{T}^3}\Gamma_{(0,0)}(t,s,\Theta(\xi,\eta))\left(\widetilde{V}(s,\eta)-\widetilde{V}^{(0,0)}\right)\cdot D_{\widetilde{\mathcal{X}}}\hat{z}(s,\eta) d \eta d s,
\end{align*}
and $\lim\limits_{n\rightarrow\infty}\widetilde{z}^n(T,\xi)=\lim\limits_{n\rightarrow\infty}\widetilde{z}_T^n=\widetilde{z}_T$, then we get $\hat{z}=\widetilde{z}$ by the uniqueness of the equation.

Finally, letting $n\rightarrow\infty$ in estimate \eqref{first order Schauder estimate_lifted}, and likewise using the equivalent relation between $z$ and $\widetilde{z}$ (cf. \cite[Proposition 8.3]{07BB}), we finish the proof.
\end{proof}

\subsection{Results for the KFP equation}
We start by recalling the definition of the weak solution to the KFP equation \eqref{general KFP} in Definition \ref{Def_general KFP weak sol.}. Noting that the definition is well-posed. In fact, thanks to Lemma \ref{Lem_LHJB wellposed regularity}, we have $\phi(s,\cdot) \in C_{\mathcal{X}}^{1+\alpha}(\mathbb{T}^2)$ for any $s \in [0,T]$, thus $\left\langle\rho_0,\phi(0,\cdot)\right\rangle$ is well defined. In addition, we know that
$$
\|\phi(s,\cdot)\|_{W_{\mathcal{X}}^{1,\infty}\left(\mathbb{T}^2\right)} \leq C.
$$
Hence, the last integral is also well defined as $f \in L^1\left([0,T];W_\mathcal{X}^{-1,\infty}(\mathbb{T}^2)\right)$.
\begin{Rem}\label{Rem_f=div(c)}
What we are emphatically interested in is a special case of distribution f. Suppose there exists an integrable function $c:[0,T] \times \mathbb{T}^2 \rightarrow \mathbb{R}^2$ such that for any $\phi \in W_\mathcal{X}^{1, \infty}(\mathbb{T}^2)$,
$$
\langle f(t),\phi\rangle=\int_{\mathbb{T}^2} c(t,x) \cdot D_\mathcal{X}\phi(x)dx,
$$
this means that equation \eqref{general KFP} can be written in the following form
$$
\begin{cases}
\partial_t\rho-\Delta_\mathcal{X}\rho+\operatorname{div}_{\mathcal{X}}(\rho b)=\operatorname{div}_{\mathcal{X}}(c), & \text{ in } [0,T] \times \mathbb{T}^2, \\
\rho(0)=\rho_{0}, & \text{ in } \mathbb{T}^2.
\end{cases}
$$
Under this particular case, We can guarantee condition $f \in L^{1}\left([0,T];W_{\mathcal{X}}^{-1,\infty}(\mathbb{T}^2)\right)$ by simply requiring $c \in L^{1}\left([0,T] \times \mathbb{T}^2\right)$. Indeed, according to Jensen's inequality,
\begin{align*}
\left\|f\right\|_{L^{1}\left([0,T];W_{\mathcal{X}}^{-1,\infty}(\mathbb{T}^2)\right)} = & \int_{0}^{T}\sup_{\left\|\phi\right\|_{W_{\mathcal{X}}^{1,\infty}} \leq 1}\left(\int_{\mathbb{T}^2}c(t,x) \cdot D_{\mathcal{X}}\phi(x)dx\right)dt \\
\leq & C\int_{0}^{T}\int_{\mathbb{T}^2}\left|c(t,x)\right|dxdt=C\left\|c\right\|_{L^1\left([0,T] \times \mathbb{T}^2\right)}.
\end{align*}
\end{Rem}
We now prove the existence, uniqueness and regularities of the weak solution.
\begin{proof}[Proof of Lemma \ref{Lem_general KFP wellposed regularity}]
\emph{Step 1: Existence.} We begin by assuming that
$$
b \in C_{\mathcal{X}}^{\frac{\alpha}{2},1+\alpha}([0,T] \times \mathbb{T}^2),\quad f \in C_{\mathcal{X}}^{\frac{\alpha}{2},\alpha}([0,T] \times \mathbb{T}^2),\quad \rho_0 \in C_{\mathcal{X}}^{2+\alpha}(\mathbb{T}^2),
$$
and proving estimate \eqref{general KFP sol. regularity}.

In this case, splitting the divergence terms in \eqref{general KFP}, we can directly obtain that $\rho$ is a classical solution of this linear equation with smooth coefficients (cf. \cite[Theorem 12.1]{10BBLU}). Consider the unique solution $\phi$ of equation \eqref{dual KFP} with $\xi=0$ and $\psi \in C_{\mathcal{X}}^{1+\alpha}(\mathbb{T}^2)$. Multiplying the equation of $\rho$ for $\phi$ and integrating by parts in $[0,t] \times \mathbb{T}^2$ we get
\begin{equation}\label{weak formula}
\langle \rho(t),\psi \rangle = \langle \rho_0,\phi(0,\cdot) \rangle + \int_{0}^{t}\langle f(s),\phi(s,\cdot) \rangle ds.
\end{equation}
We can know from \eqref{LHJB_USE_first order} that
$$
\sup_{s \in [0,T]} \left\|\phi(s,\cdot)\right\|_{C_{\mathcal{X}}^{1+\alpha}(\mathbb{T}^2)} \leq C\left\|\psi\right\|_{C_{\mathcal{X}}^{1+\alpha}(\mathbb{T}^2)},
$$
where $C$ depends on $\sup\limits_{t \in [0,T]}\left\|b(t,\cdot)\right\|_\alpha$.

Thus the right hand side of \eqref{weak formula} has the estimate
\begin{align*}
& \langle \rho_0,\phi(0,\cdot) \rangle + \int_{0}^{t}\langle f(s),\phi(s,\cdot) \rangle ds\\
\leq & C\left\|\psi\right\|_{C_{\mathcal{X}}^{1+\alpha}(\mathbb{T}^2)} \left(\left\|\rho_0\right\|_{C_{\mathcal{X}}^{-(1+\alpha)}\left(\mathbb{T}^2\right)} +\int_{0}^{t}\left\|f(s)\right\|_{W_{\mathcal{X}}^{-1,\infty}\left(\mathbb{T}^2\right)}ds\right).
\end{align*}
Taking the sup when $\psi \in C_{\mathcal{X}}^{1+\alpha}(\mathbb{T}^2)$ and $\left\|\psi\right\|_{C_{\mathcal{X}}^{1+\alpha}(\mathbb{T}^2)} \leq 1$ for \eqref{weak formula}, we obtain
$$
\sup_{t \in [0,T]}\left\|\rho(t)\right\|_{C_{\mathcal{X}}^{-(1+\alpha)}(\mathbb{T}^2)} \leq C\left(\left\|\rho_0\right\|_{C_{\mathcal{X}}^{-(1+\alpha)}(\mathbb{T}^2)} + \left\|f\right\|_{L^1([0,T];W_{\mathcal{X}}^{-1,\infty}(\mathbb{T}^2))}\right).
$$
Next we have to prove the $C_{\mathcal{X}}^{-\frac{\alpha}{2},-\alpha}$ estimate. Consider the solution of equation \eqref{dual KFP} with $t=T$, $\psi=0$ and $\xi \in C_{\mathcal{X}}^{\frac{\alpha}{2},\alpha}([0,T] \times \mathbb{T}^2)$, we can have (cf. \cite[Theorem 14.1]{10BBLU})
\begin{equation}\label{phi_1+alpha}
\left\|\phi\right\|_{C_{\mathcal{X}}^{\frac{\alpha}{2},1+\alpha}([0,T]\times\mathbb{T}^2)} \leq C\left\|\xi\right\|_{C_{\mathcal{X}}^{\frac{\alpha}{2},\alpha}([0,T]\times\mathbb{T}^2)},
\end{equation}
where $C$ depends on $\left\|b(t,\cdot)\right\|_{\frac{\alpha}{2},\alpha}$.

Integrating for the equation of $\rho$ in $[0,T] \times \mathbb{T}^2$ one has
$$
\int_{0}^{T} \int_{\mathbb{T}^2} \rho\xi dxds=\left\langle\rho_{0},\phi(0, \cdot)\right\rangle+\int_{0}^{T}\langle f(s),\phi(s,\cdot)\rangle ds.
$$
Combining with \eqref{phi_1+alpha}, we obtain
$$
\int_{0}^{T}\int_{\mathbb{T}^2} \rho\xi dxds \leq C\left\|\xi\right\|_{C_{\mathcal{X}}^{\frac{\alpha}{2},\alpha}([0,T]\times\mathbb{T}^2)} \left(\left\|\rho_0\right\|_{C_{\mathcal{X}}^{-(1+\alpha)}(\mathbb{T}^2)} + \left\|f\right\|_{L^1([0,T];W_{\mathcal{X}}^{-1,\infty})}\right).
$$
Similarly, taking the sup for $\left\|\xi\right\|_{C_{\mathcal{X}}^{\frac{\alpha}{2},\alpha}([0,T]\times\mathbb{T}^2)} \leq 1$, we eventually get
$$
\left\|\rho\right\|_{C_{\mathcal{X}}^{-\frac{\alpha}{2},-\alpha}([0,T]\times\mathbb{T}^2)} \leq C\left(\left\|\rho_0\right\|_{C_{\mathcal{X}}^{-(1+\alpha)}(\mathbb{T}^2)} + \left\|f\right\|_{L^1([0,T];W_{\mathcal{X}}^{-1,\infty}(\mathbb{T}^2))}\right).
$$

In the general case, we consider the mollified versions $\rho_0^n$, $f^n$, $b^n$ converging to $\rho_0$, $f$, $b$ respectively in $C_{\mathcal{X}}^{-(1+\alpha)}(\mathbb{T}^2)$ (cf. \cite[Lemma 2.3]{23JR}), $L^1([0,T];W_{\mathcal{X}}^{-1,\infty}(\mathbb{T}^2))$ and $C_{\mathcal{X}}^{\frac{\alpha}{2},\alpha}([0,T] \times \mathbb{T}^2)$. And we call $\rho^n$ the related solution of equation \eqref{general KFP}. We can find that
\begin{gather*}
\left\|\rho_0^n\right\|_{C_{\mathcal{X}}^{-(1+\alpha)}(\mathbb{T}^2)} \leq C\left\|\rho_0\right\|_{C_{\mathcal{X}}^{-(1+\alpha)}(\mathbb{T}^2)},\\ \left\|f^n\right\|_{L^1\left([0,T];W_{\mathcal{X}}^{-1,\infty}(\mathbb{T}^2)\right)} \leq C\left\|f\right\|_{L^1\left([0,T];W_{\mathcal{X}}^{-1,\infty}(\mathbb{T}^2)\right)},\\
\left\|b^n\right\|_{C_{\mathcal{X}}^{\frac{\alpha}{2},\alpha}([0,T] \times \mathbb{T}^2)} \leq C\left\|b\right\|_{C_{\mathcal{X}}^{\frac{\alpha}{2},\alpha}([0,T] \times \mathbb{T}^2)},
\end{gather*}
where $C$ is independent of $n$. Now we apply \eqref{general KFP sol. regularity} in the regular case to have
\begin{equation}\label{approx. sol. regularity}
\begin{aligned}
& \sup_{t \in [0,T]}\left\|\rho^n(t)\right\|_{C_{\mathcal{X}}^{-(1+\alpha)}(\mathbb{T}^2)} + \left\|\rho^n\right\|_{C_{\mathcal{X}}^{-\frac{\alpha}{2},-\alpha}([0,T] \times \mathbb{T}^2)}\\
\leq & C\left(\left\|\rho_0\right\|_{C_{\mathcal{X}}^{-(1+\alpha)}(\mathbb{T}^2)}+ \left\|f\right\|_{L^1\left([0,T];W_{\mathcal{X}}^{-1,\infty}(\mathbb{T}^2)\right)}\right).
\end{aligned}
\end{equation}
Note that $b^n$ is bounded uniformly in $n$ and $C$ is therefore independent of $n$. Based on the linearity of the equation, the function $\rho^{n,m}:=\rho^n-\rho^m$ also satisfies equation \eqref{general KFP} with $b=b^n$, $f=f^n-f^m+\operatorname{div}_{\mathcal{X}}(\rho^m(b^n-b^m))$, $\rho_0=\rho_0^n-\rho_0^m$. Then we also have
\begin{equation}\label{Cauchy seq. regularity}
\begin{aligned}
& \sup_{t \in [0,T]}\left\|\rho^{n,m}(t)\right\|_{C_{\mathcal{X}}^{-(1+\alpha)}(\mathbb{T}^2)} + \left\|\rho^{n,m}\right\|_{C_{\mathcal{X}}^{-\frac{\alpha}{2},-\alpha}([0,T] \times \mathbb{T}^2)} \\
\leq & C\left(\left\|\rho_0^n-\rho_0^m\right\|_{C_{\mathcal{X}}^{-(1+\alpha)}(\mathbb{T}^2)} + \left\|f^n-f^m\right\|_{L^1\left([0,T];W_{\mathcal{X}}^{-1,\infty}(\mathbb{T}^2)\right)}\right.\\
& \left. + \left\|\operatorname{div}_{\mathcal{X}}(\rho^m(b^n-b^m))\right\|_{L^1\left([0,T];W_{\mathcal{X}}^{-1,\infty}(\mathbb{T}^2)\right)}\right).
\end{aligned}
\end{equation}
The last term can be further estimated as
\begin{align*}
\left\|\operatorname{div}_{\mathcal{X}}(\rho^m(b^n-b^m))\right\|_{L^1\left([0,T];W_{\mathcal{X}}^{-1,\infty}(\mathbb{T}^2)\right)} \leq & C\int_{0}^{T}\int_{\mathbb{T}^2}\left|\rho^m(b^n-b^m)\right|dxdt \\
\leq & C\left\|b^n-b^m\right\|_{C_{\mathcal{X}}^{\frac{\alpha}{2},\alpha}([0,T] \times \mathbb{T}^2)},
\end{align*}
as $\rho^m$ is uniformly bounded in $C_{\mathcal{X}}^{-\frac{\alpha}{2},-\alpha}([0,T] \times \mathbb{T}^2)$ by \eqref{approx. sol. regularity}. Hence, the right hand side of \eqref{Cauchy seq. regularity} tend to $0$ when $n,m \rightarrow \infty$, that is, $\rho^n$ is a Cauchy sequence. According to the completeness and Cauchy criterion for uniform convergence, there exists a $\rho \in C\left([0,T];C_\mathcal{X}^{-(1+\alpha)}(\mathbb{T}^2)\right) \bigcap C_{\mathcal{X}}^{-\frac{\alpha}{2},-\alpha}([0,T] \times \mathbb{T}^2)$ such that $\rho^n \rightarrow \rho$ in $C\left([0,T];C_\mathcal{X}^{-(1+\alpha)}(\mathbb{T}^2)\right)$ and in $C_{\mathcal{X}}^{-\frac{\alpha}{2},-\alpha}([0,T] \times \mathbb{T}^2)$. Moreover, $\rho$ satisfies \eqref{general KFP sol. regularity}.

We next prove that $\rho$ is a solution of equation \eqref{general KFP} in the sense of Definition \ref{Def_general KFP weak sol.}.

Let $\phi$ and $\phi^n$ be the solutions of equation \eqref{dual KFP} associated to $b$ and $b^n$ respectively. We have the weak formulation for $\rho^n$ as follows:
$$
\langle\rho^n(t),\psi\rangle+\int_{0}^{t} \langle \rho^n(s),\xi(s,\cdot)\rangle ds=\left\langle\rho_0^n,\phi^n(0,\cdot)\right\rangle+\int_{0}^{t}\langle f^n(s),\phi^n(s,\cdot)\rangle ds.
$$
We only need to show that $\phi^n$ converges to $\phi$, since the desired conclusion can be obtained by directly taking the limit for both sides of the equality as above. Actually, for any $n \in \mathbb{N}$, the function $\bar{\phi}^n:=\phi^n-\phi$ satisfies
\begin{equation*}
\begin{cases}
-\partial_t\bar{\phi}^n-\Delta_\mathcal{X}\bar{\phi}^n+b^n D_{\mathcal{X}}\bar{\phi}^n=-(b^n-b)D_{\mathcal{X}}\phi, \\
\bar{\phi}^n(t)=0.
\end{cases}
\end{equation*}
We have
\begin{align*}
\left\|\bar{\phi}^n\right\|_{C_{\mathcal{X}}^{\frac{\alpha}{2},1+\alpha}([0,T] \times \mathbb{T}^2)} \leq & C\left\|(b^n-b)D_{\mathcal{X}}\phi\right\|_{C_{\mathcal{X}}^{\frac{\alpha}{2},\alpha}([0,T] \times \mathbb{T}^2)} \\
\leq & C\left\|D_{\mathcal{X}}\phi\right\|_{C_{\mathcal{X}}^{\frac{\alpha}{2},\alpha}([0,T] \times \mathbb{T}^2)} \left\|b^n-b\right\|_{C_{\mathcal{X}}^{\frac{\alpha}{2},\alpha}([0,T] \times \mathbb{T}^2)} \rightarrow 0.
\end{align*}
Hence $\phi^n \rightarrow \phi$ in $C_{\mathcal{X}}^{\frac{\alpha}{2},1+\alpha}([0,T] \times \mathbb{T}^2)$, and so that $\rho$ is a weak solution of equation \eqref{general KFP}. This completes the proof of existence.

\emph{Step 2: Uniqueness.} Consider two weak solutions $\rho_1$ and $\rho_2$ of equation \eqref{general KFP}. Then the function $\rho:=\rho_1-\rho_2$ is a weak solution of
\begin{equation*}
\begin{cases}
\partial_t\rho-\Delta_\mathcal{X}\rho+\operatorname{div}_{\mathcal{X}}(\rho b)=0, \\
\rho(0)=0,
\end{cases}
\end{equation*}
The weak formulation implies, for any $\xi \in C\left([0,t];C_{\mathcal{X}^{1+\alpha}}(\mathbb{T}^2)\right)$ and $\psi \in C_{\mathcal{X}}^{1+\alpha}(\mathbb{T}^2)$,
$$
\langle\rho(t),\psi\rangle+\int_{0}^{t} \langle \rho(s),\xi(s,\cdot) \rangle ds=0,
$$
which leads to
$$
\sup_{t \in [0,T]}\left\|\rho(t)\right\|_{C_{\mathcal{X}}^{-(1+\alpha)}(\mathbb{T}^2)}=0.
$$
This completes the uniqueness part.

\emph{Step 3: Stability.} The stability can be easily derived from the estimates obtained previously. Since $f^k \rightarrow f$, $\rho_0^k \rightarrow \rho$ and $b^k \rightarrow b$, then the function $\bar{\rho}^k:=\rho^k-\rho$ satisfies equation \eqref{general KFP} with $b$, $\rho_0$ and $f$ replaced by $b^k$, $\rho_0^k-\rho_0$ and $f^k-f+\operatorname{div}_{\mathcal{X}}(\rho(b^k-b))$. Using \eqref{general KFP sol. regularity} again, we have
\begin{align*}
& \sup_{t \in [0,T]}\left\|\bar{\rho}^k(t)\right\|_{C_{\mathcal{X}}^{-(1+\alpha)}(\mathbb{T}^2)} + \left\|\bar{\rho}^k\right\|_{C_{\mathcal{X}}^{-\frac{\alpha}{2},-\alpha}([0,T] \times \mathbb{T}^2)} \\
\leq & C\left(\left\|\rho_0^k-\rho_0\right\|_{C_{\mathcal{X}}^{-(1+\alpha)}(\mathbb{T}^2)} + \left\|f^k-f\right\|_{L^1\left([0,T];W_{\mathcal{X}}^{-1,\infty}\right)}\right.\\
& \left. + \left\|\operatorname{div}_{\mathcal{X}}(\rho(b^k-b))\right\|_{L^1\left([0,T];W_{\mathcal{X}}^{-1,\infty}\right)}\right).
\end{align*}
Same as before, we also get $\rho^k \rightarrow \rho$ in $C\left([0,T];C_\mathcal{X}^{-(1+\alpha)}(\mathbb{T}^2)\right)$ when $k \rightarrow \infty$. Thus we complete the proof of the lemma.
\end{proof}
We next state a lower regularity estimate of the solution to equation \ref{general KFP} when $b$ is more regular. And it is useful for improving the regularity of $\frac{\delta U}{\delta m}$ with respect to $y$.
\begin{Cor}\label{Cor_(L)KFP lower regularity}
Let $b \in C_{\mathcal{X}}^{\frac{\alpha}{2},\alpha}([0,T] \times \mathbb{T}^2)$, $f \in L^1([0,T];W_{\mathcal{X}}^{-1,\infty}(\mathbb{T}^2))$ and $\rho_0 \in C_{\mathcal{X}}^{-1}(\mathbb{T}^2)$. Then the unique solution $\rho$ of equation \eqref{general KFP} satisfies
\begin{equation*}
\sup_{t \in [0,T]}\left\|\rho(t)\right\|_{-(2+\alpha)} \leq C\left(\left\|\rho_0\right\|_{-(2+\alpha)} + \left\|f\right\|_{L^1\left([0,T];W_{\mathcal{X}}^{-1,\infty}(\mathbb{T}^2)\right)}\right).
\end{equation*}
\end{Cor}
\begin{proof}
The proof is similar to the previous one by duality, except taking $\phi$, the solution of equation \eqref{dual KFP}, with $\psi \in C_{\mathcal{X}}^{2+\alpha}(\mathbb{T}^2)$ and $\xi=0$, as a test function.
\end{proof}

\section{Preliminary regularities and Lipschitz continuity of U}\label{Sec_4}
\begin{proof}[Proof of Proposition \ref{Prop_MFG system wellposed regularity}]
We prove the existence by using Schauder fixed point theorem. Let
$$
E:=\left\{m \in C\left([t_0,T];\mathcal{P}\left(\mathbb{T}^2\right)\right)\text{ s.t. }d_1\left(m\left(t_1\right), m\left(t_2\right)\right) \leq C_E|t_1-t_2|^\frac{1}{2}\right\},
$$
where the constant $C_E$ to be determined does not depend on $m(t)$. Then it is easy to verify that E is a convex compact set.

The mapping $\Phi: E \rightarrow E$ is constructed as follows.

Fix $\mu \in E$, we consider the following HJB equation
\begin{equation}\label{SFPT_HJB}
\begin{cases}
-\partial_t u-\Delta_{\mathcal{X}}u+\frac{1}{2}\left|D_{\mathcal{X}} u\right|^2=F(x, \mu(t)), & \text{in } [t_0,T] \times \mathbb{T}^2, \\
u(T,x)=G(x, \mu(T)), & \text{in } \mathbb{T}^2.
\end{cases}
\end{equation}
From hypotheses H\ref{hyp1}) and H\ref{hyp2}), we know $F\left(\cdot,\mu(\cdot)\right) \in C_{\mathcal{X}}^{\frac{\alpha}{2},\alpha}\left([t_0,T]\times\mathbb{T}^2\right)$ and $G\left(\cdot,\mu(T)\right) \in C_{\mathcal{X}}^{2+\alpha}\left(\mathbb{T}^2\right)$ with the corresponding norms independent of $\mu$. Then applying the same method of proof from \cite[Theorem 1.1]{23JRWX} to the space $\mathbb{T}^2$, we can get that there exists a unique solution $u \in C_{\mathcal{X}}^{1+\frac{\alpha}{2},2+\alpha}\left([t_0,T]\times\mathbb{T}^2\right)$ to this equation.

We use the Hopf transform to turn equation \eqref{SFPT_HJB} equivalently into a linear form satisfied by $w:=\exp\left(-\frac{u}{2}\right)$, that is
\begin{equation*}
\begin{cases}
-\partial_t w-\Delta_{\mathcal{X}}w+\frac{1}{2}F(x, \mu(t))w=0, & \text{in } [t_0,T] \times \mathbb{T}^2, \\
w(T,x)=\exp\left(-\frac{G(x, \mu(T))}{2}\right), & \text{in } \mathbb{T}^2.
\end{cases}
\end{equation*}
Thanks to the weak maximum principle, we can know that $w>0$ (cf. \cite[pp. 26-27]{23JRWX}). Then we can get from \cite[Theorem 1.1]{07BB} and from \eqref{MaxE} in Lemma \ref{Lem_LHJB prior estimate} to be proved later that
\begin{equation}\label{HJB_Schauder estimate}
\begin{aligned}
\left\|u\right\|_{1+\frac{\alpha}{2},2+\alpha} \leq \left\|w\right\|_{1+\frac{\alpha}{2},2+\alpha} \leq& C \left(\left\|G\left(\cdot,\mu(T)\right)\right\|_{2+\alpha} +\left\|w\right\|_\infty\right) \\
\leq & C \left\|G\left(\cdot,\mu(T)\right)\right\|_{2+\alpha},
\end{aligned}
\end{equation}
where the constant $C$ only depends on $\left\|F\left(\cdot,\mu(\cdot)\right)\right\|_{\frac{\alpha}{2},\alpha}$, which does not depend on $\mu$, $t_0$ and $m_0$.

Let $m$ be the weak solution to the KFP equation
\begin{equation}\label{SFPT_KFP}
\begin{cases}
\partial_t m-\Delta_{\mathcal{X}} m-\operatorname{div}_{\mathcal{X}}\left(m D_{\mathcal{X}}u\right)=0, & \text{in } (t_0,T] \times \mathbb{T}^2, \\
m(t_0)=m_0, & \text{in } \mathbb{T}^2,
\end{cases}
\end{equation}
then we can define $\Phi(\mu)=m$. Since $m_0 \in \mathcal{P}(\mathbb{T}^2) \subset C_{\mathcal{X}}^{-(1+\alpha)}(\mathbb{T}^2)$, the existence and uniqueness of the distributional solution $m(t) \in C\left([t_0,T];C_{\mathcal{X}}^{-(1+\alpha)}(\mathbb{T}^2)\right)$ to this equation will be given in Lemma \ref{Lem_general KFP wellposed regularity}. Next we check that $m \in E$.

For any Borel set $A \subset \mathbb{T}^2$, there exists a smooth sequence $\{f_n\}_n \subset C_{\mathcal{X}}^{1}\left(\mathbb{T}^2\right)$ converges to $1_A$. Then by the control convergence theorem, we have
$$
\int_{\mathbb{T}^2}1_A dm(t)(x)=\lim_{n \rightarrow\infty} \int_{\mathbb{T}^2}f_n dm(t)(x) < \infty.
$$
Hence $m(t) \in C\left([t_0,T];\mathcal{P}(\mathbb{T}^2)\right)$. As for the H\"{o}lder continuity, let $\left(\Omega,\mathcal{F},\mathbb{P}\right)$ be a standard probability space, and $\xi_t$ be the solution to the stochastic differential equation, writing in the form of Stratonovich integral,
\begin{equation*}
\begin{cases}
d \xi_t=-\sum\limits_{i=1}^{2}X_i u(t,\xi_t)X_i(\xi_t)dt+\sum\limits_{i=1}^{2}\sqrt{2}X_i(\xi_t)\circ dB_t^i, & t \in (t_0,T], \\
\mathcal{L}\left(\xi_{t_0}\right)=m_0.
\end{cases}
\end{equation*}
Then the law of $\xi_t$ is $m(t)$ for any $t$, i.e. $\mathcal{L}\left(\xi(t)\right)=m(t)$ (cf. \cite[Lemma 4.1]{18Ry}). By the definition of the $d_{cc}$ distance and note that the diffusion term propagates with speed $\sqrt{t}$ in the direction of the vector fields $X_i$ (cf. \cite{19BCP}), we have
\begin{align*}
E\left[d_{cc}\left(\xi_{t_1},\xi_{t_2}\right)\right]\leq & E\left[\int_{t_1}^{t_2}\left|D_{\mathcal{X}}u(t,\xi_t)\right|dt +\left(\int_{t_1}^{t_2}4dt\right)^\frac{1}{2}\right] \\
\leq & \left\|D_{\mathcal{X}}u\right\|_\infty|t_1-t_2|+2\sqrt{|t_1-t_2|}.
\end{align*}
Since $d_1\left(m(t_1),m(t_2)\right) \leq E\left[d_{cc}\left(\xi_{t_1},\xi_{t_2}\right)\right]$, we get therefore
\begin{equation}\label{KFP_Horlder ctn.}
\sup_{t_1 \neq t_2}\frac{d_1(m(t_1),m(t_2))}{|t_1-t_2|^{\frac{1}{2}}} \leq \|D_{\mathcal{X}}u\|_\infty T+2=C_E.
\end{equation}

It remains for us to prove that $\Phi$ is continuous. Let any $\mu_n \rightarrow \mu$, then let $u_n$ and $m_n$ be the corresponding solutions to the equation \eqref{SFPT_HJB} and \eqref{SFPT_KFP}. For any $n,m \in \mathbb{N}$, $u_n-u_m$ satisfies the linear degenerate PDE as follows
\begin{equation*}
\begin{cases}
-\partial_t z-\Delta_{\mathcal{X}}z+\frac{1}{2}V(t,x)\cdot D_{\mathcal{X}}z=F(x, \mu_n(t))-F(x, \mu_m(t)), & \text{in } [t_0,T] \times \mathbb{T}^2, \\ z(T,x)=G(x, \mu_n(T))-G(x, \mu_m(T)), & \text{in } \mathbb{T}^2,
\end{cases}
\end{equation*}
where
$$
V(t,x):=D_{\mathcal{X}}\left(u_n+u_m\right)(t,x).
$$
The same reason we can obtain the Schauder estimate
\begin{align*}
& \left\|u_n-u_m\right\|_{1+\frac{\alpha}{2},2+\alpha} \\
& \quad \leq C\left(\left\|F\left(\cdot,\mu_n(\cdot)\right)-F\left(\cdot,\mu_m(\cdot)\right)\right\|_{\frac{\alpha}{2},\alpha} +\left\|G\left(\cdot,\mu_n(T)\right)-G\left(\cdot,\mu_m(T)\right)\right\|_{2+\alpha}\right),
\end{align*}
where $C$ only depends on $\left\|V(t,x)\right\|_{\frac{\alpha}{2},\alpha}$. Let $n,m \rightarrow \infty$, and due to the continuity of $F,G$ with respect to the measure and the fact that $\left\{u_n\right\}_n$ is uniformly bounded in $C_{\mathcal{X}}^{1+\frac{\alpha}{2},2+\alpha}\left([t_0,T]\times\mathbb{T}^2\right)$ by \eqref{HJB_Schauder estimate}, we obtain that $\left\{u_n\right\}_n$ is a Cauchy sequence. Thus we have $u_n \rightarrow u$ in $C_{\mathcal{X}}^{1+\frac{\alpha}{2},2+\alpha}\left([t_0,T]\times\mathbb{T}^2\right)$ since the uniqueness of the solution to equation \eqref{SFPT_HJB}.

The convergence of $\left\{m_n(t)\right\}_n$ can be obtained by the stability of the KFP equation (see Lemma \ref{Lem_general KFP wellposed regularity}), and also the limit $m \in E$ being the unique solution to equation \eqref{SFPT_KFP} related to $\mu(t)$. We would need to show that $m_n \rightarrow m$ in $C\left([t_0,T];\mathcal{P}(\mathbb{T}^2)\right)$. By the definition of $d_1$ distance, we have that for any $\varepsilon>0$, there exists a $d_{cc}$-Lipschitz function $\phi^\varepsilon$ with $Lip(\phi^\varepsilon)=1$ such that
\begin{equation*}
d_1\left(m_n(t),m(t)\right) \leq \int_{\mathbb{T}^2}\phi^{\varepsilon}(x)d\left(m_n(t)-m(t)\right)(x)-\varepsilon, \text{ for any } t \in [t_0,T].
\end{equation*}
We can take a smooth sequence $\left\{\phi_k^{\varepsilon}\right\}_k$ to approximate $\phi^{\varepsilon}$. Since $\left\|\phi_k^{\varepsilon}\right\|_{C_{\mathcal{X}}^{1+\alpha}(\mathbb{T}^2)}$ is bounded uniformly in $\varepsilon$, we have
\begin{align*}
d_1\left(m_n(t),m(t)\right) \leq & \lim_{k\rightarrow\infty}\int_{\mathbb{T}^2}\phi_k^{\varepsilon}(x)d\left(m_n(t)-m(t)\right)(x)-\varepsilon\\
\leq & C\left\|m_n(t)-m(t)\right\|_{C_{\mathcal{X}}^{-(1+\alpha)}(\mathbb{T}^2)}-\varepsilon, \text{ for any } t \in [t_0,T].
\end{align*}
Taking $n \rightarrow \infty$ and $\varepsilon \rightarrow 0$ leads to the desired conclusion. This concludes the proof of continuity.

Applying Schauder fixed point theorem, we obtain a solution
$$
(u,m) \in C_{\mathcal{X}}^{1,2}\left([t_0,T]\times\mathbb{T}^2\right) \times C\left([t_0,T];\mathcal{P}(\mathbb{T}^2)\right)
$$
to the MFG system \eqref{MFG system}. The uniqueness can be derived from the Lasry-Lions monotonicity argument (cf. \cite[Lemma 3.1.2]{19CDLL}), and the regularity results are given by \eqref{HJB_Schauder estimate} and \eqref{KFP_Horlder ctn.}.

Assume that $m_0$ is absolutely continuous with a smooth positive density, since $m$ is the unique solution to the linear degenerate equation
\begin{equation*}
\begin{cases}
\partial_t m-\Delta_{\mathcal{X}} m-D_{\mathcal{X}}u\cdot D_{\mathcal{X}}m-\operatorname{div}_{\mathcal{X}}\left(D_{\mathcal{X}}u\right)m=0, & \text{ in } (t_0,T] \times \mathbb{T}^2, \\
m(t_0)=m_0, & \text{ in } \mathbb{T}^2,
\end{cases}
\end{equation*}
with $C_{\mathcal{X}}^{\frac{\alpha}{2},\alpha}$ coefficients and $C_{\mathcal{X}}^{2+\alpha}$ initial condition, we have
$$
m \in C_{\mathcal{X}}^{1,2} \left((t_0, T] \times \mathbb{T}^{2} \right) \bigcap C([t_0, T] \times \mathbb{T}^2)
$$
by \cite[Theorem 12.1]{10BBLU}. Furthermore, by Schauder estimates (cf. \cite[Theorem 14.4]{10BBLU}) $m \in C_{\mathcal{X}}^{1+\frac{\alpha}{2},2+\alpha}\left([t_0, T] \times \mathbb{T}^{2} \right)$. Moreover, by using the weak maximum principle, we get that $m>0$ whenever $m_0>0$.

The stability of the solution can be obtained by the same method used to prove the continuity of $\Phi$. We have thus completed the proof.
\end{proof}
\begin{proof}[Proof of Proposition \ref{Prop_Lip. ctn. of U}]
We show the result for $t_0=0$ for the sake of simplicity.

\emph{Step 1: Monotonicity argument.} We take advantage of the Lasry-Lions monotonicity argument in \cite[Lemma 3.1.2]{19CDLL}, and since $F,G$ is monotone, $D_{\mathcal{X}}u^1$ and $D_{\mathcal{X}}u^2$ are uniformly bounded, thus
\begin{gather*}
\int_{0}^{T}\int_{\mathbb{T}^2} \left|D_{\mathcal{X}}u^1(t,y)-D_{\mathcal{X}}u^2(t,y)\right|^2\left(m^1(t,y)+m^2(t,y)\right)dydt\\ \leq C\int_{\mathbb{T}^2} \left(u^1(0,y)-u^2(0,y)\right)\left(m_0^1(y)-m_0^2(y)\right)dy.
\end{gather*}
Since
$$
\left|\left(u^1-u^2\right)(0,x)-\left(u^1-u^2\right)(0,y)\right| \leq \left\|D_{\mathcal{X}}\left(u^1-u^2\right)(0,\cdot)\right\|_\infty d_{cc}(x,y),
$$
based on \cite[Proposition 4.2(ii)]{07BB}. Hence by the definition of $d_1$ we eventually have
\begin{gather}\label{Mono. argu.}
\int_{0}^{T}\int_{\mathbb{T}^2} \left|D_{\mathcal{X}}u^1(t,y)-D_{\mathcal{X}}u^2(t,y)\right|^2\left(m^1(t,y)+m^2(t,y)\right)dydt\\ \leq C\left\|D_{\mathcal{X}}\left(u^1-u^2\right)(0,\cdot)\right\|_\infty d_1(m_0^1,m_0^2). \nonumber
\end{gather}

\emph{Step 2: An estimate on $m^1-m^2$.} Let $\bar{m}:=m^1-m^2$. Taking $\phi$ as a test function, the weak formulations of $m^1$ and $m^2$ are as follows
\begin{gather*}
\int_{\mathbb{T}^2}\phi(t,x)m^i(t,x)dx-\int_{\mathbb{T}^2}\phi(s,x)m^i(s,x)dx\\
=\int_{s}^{t}\int_{\mathbb{T}^2} \left(\partial_t\phi+\Delta_{\mathcal{X}}\phi-D_pH\left(x,D_{\mathcal{X}}u^i\right)\cdot D_{\mathcal{X}}\phi\right)m^i(s,x)dxds,\quad i=1,2.
\end{gather*}
Set $s=0$ and subtract the two formulations, we obtain
\begin{gather*}
\int_{\mathbb{T}^2}\phi(t,x)\bar{m}(t,x)dx+\int_{0}^{t}\int_{\mathbb{T}^2} \left(-\partial_t\phi-\Delta_{\mathcal{X}}\phi+D_pH\left(x,D_{\mathcal{X}}u^1\right)\cdot D_{\mathcal{X}}\phi\right)\bar{m}(s,x)dxds\\
+\int_{0}^{t}\int_{\mathbb{T}^2} \left(D_pH\left(x,D_{\mathcal{X}}u^1\right)-D_pH\left(x,D_{\mathcal{X}}u^2\right)\right)\cdot D_{\mathcal{X}}\phi m^2(s,x)dxds\\
=\int_{\mathbb{T}^2} \phi(0,x)\left(m_0^1-m_0^2\right)(x)dx.
\end{gather*}
We choose $\phi$ as the solution of equation \eqref{homo. LHJB} with $V=D_pH(x,D_{\mathcal{X}}u^1)$ and terminal condition $z_T=\psi$ at time $t$, which is $d_{cc}$-Lipschitz with Lipschitz constant bounded by $1$. It follows by Lemma \ref{Lem_homo. LHJB Lipschitz} that $\phi$ is $d_{cc}$-Lipschitz continuous with a constant bounded uniformly. Because of the Lipschitz continuity of $D_pH$ with respect to $p$, we can obtain
$$
\int_{\mathbb{T}^2} \psi(x) \bar{m}(t,x)dx \leq C \int_0^t \int_{\mathbb{T}^2}\left|D_{\mathcal{X}}u^1-D_{\mathcal{X}}u^2\right| m^2(s,x)dxds+C d_1\left(m_0^1, m_0^2\right).
$$

Let us now use Jensen's inequality and \eqref{Mono. argu.} to get
\begin{align*}
\int_{\mathbb{T}^2} \psi(x) \bar{m}(t,x)dx & \leq C\left(\int_0^t \int_{\mathbb{T}^2}\left|D_{\mathcal{X}}u^1-D_{\mathcal{X}}u^2\right|^2 m^2(s,x)dxds\right)^{\frac{1}{2}}+Cd_1 \left(m_0^1,m_0^2\right) \\
& \leq C\left(\left\|D_{\mathcal{X}}\left(u^1-u^2\right)(0,\cdot)\right\|_\infty^{\frac{1}{2}} d_1\left(m_0^1,m_0^2\right)^{\frac{1}{2}}+d_1\left(m_0^1,m_0^2\right)\right),
\end{align*}
taking the sup over the $\psi$ $d_{cc}$-Lipschitz with Lipschitz constant bounded by $1$, and $t \in[0,T]$, we finally get
\begin{equation}\label{m^1-m^2}
\begin{aligned}
& \sup_{t \in [0,T]} d_1\left(m^1(t), m^2(t)\right)\\
& \quad \leq C\left(\left\| D_{\mathcal{X}}\left(u^1-u^2\right)(0,\cdot)\right\|_\infty^{\frac{1}{2}} d_1\left(m_0^1,m_0^2\right)^{\frac{1}{2}}+d_1\left(m_0^1,m_0^2\right)\right).
\end{aligned}
\end{equation}

\emph{Step 3: An estimate on $u^1-u^2$.} We note that $\bar{u}:=u^1-u^2$ satisfies
$$
\begin{cases}
-\partial_t \bar{u}(t,x)-\Delta_{\mathcal{X}} \bar{u}(t,x)+V(t,x) \cdot D_{\mathcal{X}} \bar{u}(t,x)=R_1(t, x),& \text { in }(0,T) \times \mathbb{T}^2, \\
\bar{u}(T,x)=R_T(x),& \text { in }\mathbb{T}^2,
\end{cases}
$$
where, for $(t,x) \in [0,T] \times \mathbb{T}^2$,
$$
V(t,x)=\int_0^1 D_pH\left(x,sD_\mathcal{X} u^1(t,x)+(1-s) D_\mathcal{X} u^2(t,x)\right)ds,
$$
$$
R_1(t,x)=\int_0^1 \int_{\mathbb{T}^2} \frac{\delta F}{\delta m}\left(x,s m^1(t)+(1-s) m^2(t), y\right)\left(m^1(t,y)-m^2(t,y)\right)dyds
$$
and
$$
R_T(t,x)=\int_0^1 \int_{\mathbb{T}^2} \frac{\delta G}{\delta m}\left(x,s m^1(T)+(1-s)m^2(T),y\right)\left(m^1(T,y)-m^2(T,y)\right)dyds.
$$
By hypothesis H\ref{hyp1}) and inequality \eqref{m^1-m^2}, we estimate that, for any $t \in [0,T]$,
\begin{align*}
& \left\|R_1(t,\cdot)\right\|_{C_{\mathcal{X}}^{1+\alpha}(\mathbb{T}^2)}\\
\leq & \int_{0}^{1} \left\|D_\mathcal{X}^y\frac{\delta F}{\delta m}(\cdot,sm^1(t)+(1-s)m^2(t),\cdot)\right\|_{C_{\mathcal{X}}^{1+\alpha}(\mathbb{T}^2) \times L^\infty}ds d_1(m^1(t),m^2(t))\\
\leq & C\left[d_1\left(m_0^1,m_0^2\right)+\|D_\mathcal{X}\bar{u}(0,\cdot)\|_{\infty}^{\frac{1}{2}} d_1\left(m_0^1,m_0^2\right)^{\frac{1}{2}}\right].
\end{align*}
Similarly, using hypothesis H\ref{hyp2}) we have
$$
\left\|R_T\right\|_{C_{\mathcal{X}}^{2+\alpha}(\mathbb{T}^2)} \leqslant C\left[d_1\left(m_0^1,m_0^2\right)+\|D_\mathcal{X}\bar{u}(0,\cdot)\|_{\infty}^{\frac{1}{2}} d_1\left(m_0^1,m_0^2\right)^{\frac{1}{2}}\right].
$$
Apart from that, $V(t,x)$ is bounded in $C_{\mathcal{X}}^{1,1+\alpha}\left([0, T] \times \mathbb{T}^{2} \right)$ owing to the regularity of $u^1$ and $u^2$. Hence using Lemma \ref{Lem_LHJB wellposed regularity} and $\varepsilon$-Cauchy inequality, we have the following estimate
\begin{align*}
\sup_{t \in [0,T]}\|\bar{u}(t,\cdot)\|_{C_{\mathcal{X}}^{2+\alpha}(\mathbb{T}^2)} \leq & C\left(\left\|R_{T}\right\|_{C_{\mathcal{X}}^{2+\alpha}(\mathbb{T}^2)}+\sup_{t \in [0,T]}\left\|R_1(t,\cdot)\right\|_{C_{\mathcal{X}}^{1+\alpha}(\mathbb{T}^2)}\right)\\
\leq & C\left(d_1\left(m_0^1,m_0^2\right)+ \|D_\mathcal{X}\bar{u}(0,\cdot)\|_{\infty}^{\frac{1}{2}}d_{1}\left(m_0^1,m_0^2\right)^{\frac{1}{2}}\right)\\
\leq & C\left(d_1\left(m_0^1,m_0^2\right)+ \varepsilon\left\|D_\mathcal{X}\bar{u}(0,\cdot)\right\|_\infty+\frac{1}{\varepsilon} d_1\left(m_0^1, m_0^2\right)\right).
\end{align*}
Choose $\varepsilon$ small enough, we therefore observe that
$$
\sup_{t \in[0, T]}\|\bar{u}(t,\cdot)\|_{C_{\mathcal{X}}^{2+\alpha}(\mathbb{T}^2)} \leq Cd_1\left(m_0^1, m_0^2\right).
$$
Again, substitute the above inequality back into \eqref{m^1-m^2}, we can find
$$
\sup_{t \in[0,T]} d_1\left(m^1(t), m^2(t)\right) \leq C d_1\left(m_0^1,m_0^2\right).
$$
\end{proof}

\section{Linearized MFG system and differentiability of U}\label{Sec_5}
\noindent
In this section, let us start to work on obtaining some estimates for system \eqref{linear MFG system}. To do this, we consider a more general linearized system of the following form:
\begin{equation}\label{general linear MFG system}
\begin{cases}
-\partial_t z-\Delta_{\mathcal{X}} z+V(t, x) \cdot D_{\mathcal{X}} z=\frac{\delta F}{\delta m}(x,m(t))(\rho(t))+b(t,x),& \text { in }[t_0, T] \times \mathbb{T}^2, \\
\partial_t\rho-\Delta_\mathcal{X}\rho-\operatorname{div}_{\mathcal{X}}(\rho V)-\operatorname{div}_{\mathcal{X}}(m \Gamma D_{\mathcal{X}}z+c(t,x))=0,& \text { in }[t_0, T] \times \mathbb{T}^2, \\
z(T, x)=\frac{\delta G}{\delta m}(x,m(T))(\rho(T))+z_T(x),\quad \rho(t_0)=\rho
_0,& \text { in }\mathbb{T}^2,
\end{cases}
\end{equation}
where $V$ is a given vector field in $C_{\mathcal{X}}^{\frac{\alpha}{2},1+\alpha}\left([t_0, T] \times \mathbb{T}^{2} \right)$, $m \in C([t_0,T];\mathcal{P}(\mathbb{T}^2))$, $\rho_0 \in C_{\mathcal{X}}^{-1}\left(\mathbb{T}^2\right)$, $\Gamma:[t_0,T] \times \mathbb{T}^2 \rightarrow \mathbb{R}^{2 \times 2}$ is continuous which maps into the family of real symmetric matrices, and the maps $b:[t_0,T] \times \mathbb{T}^2 \rightarrow \mathbb{R}$, $c:[t_0,T] \times \mathbb{T}^2 \rightarrow \mathbb{R}^2$ as well as $z_T:\mathbb{T}^2 \rightarrow \mathbb{R}$ are given in corresponding suitable spaces. It is always assumed that there is a constant $\bar{C}>0$ such that
\begin{equation}\label{Assump_linear system}
\begin{array}{c}
d_1(m(t_1),m(t_2)) \leq \bar{C}|t_1-t_2|^{\frac{1}{2}},\quad t_1,t_2 \in [t_0,T]\\
\bar{C}^{-1}I_{d \times d} \leq \Gamma(t,x) \leq \bar{C} I_{d \times d},\quad (t,x) \in [t_0,T]\times\mathbb{T}^2.
\end{array}
\end{equation}

We provide the definition of solution for this system as follows
\begin{Def}\label{Def_linear system weak sol.}
A couple $(z,\rho)$ is said to be a solution to the system \eqref{general linear MFG system} if
\begin{enumerate}
\item $z \in C_{\mathcal{X}}^{1+\frac{\alpha}{2},2+\alpha}([t_0,T] \times \mathbb{T}^2)$ is a classical solution of the first linear equation;
\item $\rho \in C\left([t_0,T];C_\mathcal{X}^{-(1+\alpha)}(\mathbb{T}^2)\right) \bigcap C_{\mathcal{X}}^{-\frac{\alpha}{2},-\alpha}([t_0,T]\times\mathbb{T}^2)$ is a distributional solution of the KFP equation in the sense of Definition \ref{Def_general KFP weak sol.}.
\end{enumerate}
\end{Def}
We now proceed to state the existence, uniqueness and regularity results for the system \eqref{general linear MFG system}.
\begin{Lem}\label{Lem_GL system regularity}
Under assumptions H\ref{hyp1}) and H\ref{hyp2}), system \eqref{general linear MFG system} has a unique solution $(z,\rho)$ in the sense of Definition \ref{Def_linear system weak sol.}, with
\begin{equation}\label{linear system_z regularity}
\sup_{t \in \left[t_{0},T\right]}\|z(t,\cdot)\|_{2+\alpha} + \sup _{\substack{t \neq t'\\t,t' \in [t_0,T']}} \frac{\left\|z\left(t^{\prime}, \cdot\right)-z(t, \cdot)\right\|_{2+\alpha}}{\left|t^{\prime}-t\right|^\beta} \leqslant CM
\end{equation}
and
\begin{equation}\label{linear system_rho regularity}
\sup_{t \in \left[t_{0},T\right]}\|\rho(t)\|_{-(1+\alpha)} + \sup _{\substack{t \neq t'\\t,t' \in [t_0,T]}} \frac{\left\|\rho\left(t^{\prime}\right)-\rho(t)\right\|_{-(1+\alpha)}}{\left|t^{\prime}-t\right|^{\frac{\alpha}{2}}} \leqslant CM,
\end{equation}
where $\beta \in (0,\frac{1}{2})$, $T' \in (t_0,T)$, the constant $C$ depends on $T$, $\alpha$, $\sup\limits_{t \in\left[t_{0},T\right]}\|V(t,\cdot)\|_{1+\alpha}$, the constant $\bar{C}$ in \eqref{Assump_linear system}, $F$ and $G$. Apart from this, $M$ is given by
\begin{align*}
M:= & \left\|z_{T}\right\|_{2+\alpha}+\left\|\rho_{0}\right\|_{-(1+\alpha)}+\sup_{t \in \left[t_{0},T\right]}\left(\|b(t,\cdot)\|_{1+\alpha}+\|c(t)\|_{L^1}\right).
\end{align*}
\end{Lem}
\begin{proof}
Assume $t_{0}=0$ without loss of generality. We use the Leray-Schauder fixed point theorem to prove the existence, where the crucial point is to show the estimates \eqref{linear system_z regularity} and \eqref{linear system_rho regularity}.

\emph{Step 1: Definition of the map $\Psi$.} Set $X:=C\left([0,T];C_{\mathcal{X}}^{-(1+\alpha)}(\mathbb{T}^2)\right)$. For any $\rho \in X$, we define $\Psi(\rho)$ in the following way.

First, we call $z$ the solution to
\begin{equation}\label{linear system_z}
\begin{cases}
-\partial_{t} z-\Delta_{\mathcal{X}}z+V(t,x)\cdot D_{\mathcal{X}}z=\frac{\delta F}{\delta m}(x,m(t))(\rho(t))+b(t,x), & \text { in }[0,T] \times \mathbb{T}^{2}, \\
z(T)=\frac{\delta G}{\delta m}(x, m(T))(\rho(T))+z_{T}, & \text { in } \mathbb{T}^{2}.\end{cases}
\end{equation}
By Lemma \ref{Lem_LHJB wellposed regularity}, there exists a unique solution $z \in C_{\mathcal{X}}^{1,2} \left((0, T) \times \mathbb{T}^{2} \right) \bigcap C([0, T] \times \mathbb{T}^2)$ to the above equation and it satisfies
\begin{equation}\label{z regularity}
\begin{aligned}
& \sup _{t \in[0, T]}\|z(t, \cdot)\|_{2+\alpha} + \sup _{\substack{t \neq t'\\t,t' \in [0,T']}} \frac{\left\|z\left(t^{\prime}, \cdot\right)-z(t, \cdot)\right\|_{2+\alpha}}{\left|t^{\prime}-t\right|^\beta} \\
\leqslant & C\left(\left\|\frac{\delta G}{\delta m}(x, m(T))(\rho(T))\right\|_{2+\alpha}+\left\|z_T\right\|_{2+\alpha}\right.\\
& \left.+ \sup _{t \in[0, T]}\|\frac{\delta F}{\delta m}(\cdot,m(t))(\rho(t))\|_{1+\alpha}+\sup _{t \in[0,T]}\|b(t,\cdot)\|_{1+\alpha}\right) \\
\leqslant & C\left(\left\|z_T\right\|_{2+\alpha} + \sup _{t \in[0,T]}\|\rho(t)\|_{-(1+\alpha)} + \sup _{t \in[0,T]}\|b(t,\cdot)\|_{1+\alpha}\right) \\
\leqslant & C\left(M + \sup _{t \in[0,T]}\|\rho(t)\|_{-(1+\alpha)}\right),
\end{aligned}
\end{equation}
where the constant $C$ depends on $\sup\limits_{t \in \left[0,T\right]}\|V(t,\cdot)\|_{1+\alpha}$, $\alpha$ and the constants in H\ref{hyp1}) and H\ref{hyp2}).

Next we define $\Psi(\rho):=\tilde{\rho}$ as the distributional solution to the KFP equation
\begin{equation*}
\begin{cases}
\partial_t\tilde{\rho}-\Delta_\mathcal{X}\tilde{\rho}-\operatorname{div}_{\mathcal{X}}(\tilde{\rho} V)-\operatorname{div}_{\mathcal{X}}(m \Gamma D_{\mathcal{X}}z+c)=0, & \text{ in } [0,T] \times \mathbb{T}^2,\\
\tilde{\rho}(0)=\rho_{0}, & \text{ in } \mathbb{T}^2.
\end{cases}
\end{equation*}
From Lemma \ref{Lem_general KFP wellposed regularity} we know that $\tilde{\rho} \in X$. We shall prove that the map $\Psi$ is compact and continuous.

As for the compactness, let $\{\rho_n\}_n \subset X$ be a sequence with
$$
\sup\limits_{t \in [0,T]}\left\|\rho_n\right\|_{-(1+\alpha)} \leq C
$$
for a certain constant $C>0$. For each $n$, we consider the corresponding solutions $z_n$ and $\tilde{\rho}_n$. It follows from \eqref{z regularity} that$\left\|z_n\right\|_{C\left([0,T];C_{\mathcal{X}}^{2+\alpha}(\mathbb{T}^2)\right)}$ and $\left\|z_n\right\|_{C^{\beta}\left([0,T'];C_{\mathcal{X}}^{2+\alpha}(\mathbb{T}^2)\right)}$ are uniformly bounded. So we can use Ascoli-Arzel\`{a}'s theorem to obtain that there exists a function $z$ such that $z_n \rightarrow z$ up to a subsequence uniformly at least in $C\left([0,T'];C_{\mathcal{X}}^{1}(\mathbb{T}^2)\right)$. Furthermore, for any $t \in (T',T]$, there exists a subsequence $n_k(t)\rightarrow\infty$, such that $D_{\mathcal{X}}z_{n_k(t)}(t,x)$ converges to $D_{\mathcal{X}}z(t,x)$ uniformly in $x$. Then, combining with the $L^1$ boundedness of $m$ we can estimate
\begin{align*}
& \left\|m\Gamma\left(D_{\mathcal{X}}z_{n_k(t)}-D_{\mathcal{X}}z\right)\right\|_ {L^1\left([0,T]\times\mathbb{T}^2\right)}\\
= & \int_{0}^{T'}\int_{\mathbb{T}^2}\left|m\Gamma\left(D_{\mathcal{X}}z_{n_k(t)}-D_{\mathcal{X}}z\right)\right|dxdt +\int_{T'}^{T}\int_{\mathbb{T}^2}\left|m\Gamma\left(D_{\mathcal{X}}z_{n_k(t)}-D_{\mathcal{X}}z\right)\right|dxdt \\
\leq & C_1\left(\left\|D_{\mathcal{X}}z_{n_k(t)}-D_{\mathcal{X}}z\right\|_ {L^\infty\left([0,T]\times\mathbb{T}^2\right)}+\left|T-T'\right|\right),
\end{align*}
where $C_1$ is independent of $n$ and $T'$. Let $k \rightarrow \infty$ and $T'\rightarrow T$, we have
$$
m\Gamma D_{\mathcal{X}}z_{n_k(T)}+c \rightarrow m\Gamma D_{\mathcal{X}}z+c \text{ in } L^1([0,T] \times \mathbb{T}^2),
$$
which directly implies
$$
\operatorname{div}_{\mathcal{X}}\left(m\Gamma D_{\mathcal{X}}z_{n_k(T)}+c\right) \rightarrow \operatorname{div}_{\mathcal{X}}\left(m\Gamma D_{\mathcal{X}}z+c\right) \text{ in } L^1\left([0,T];W_{\mathcal{X}}^{-1,\infty}(\mathbb{T}^2)\right).
$$
Thus, it turns out that the stability result proved in Lemma \ref{Lem_general KFP wellposed regularity} shows that $\tilde{\rho}_{n_k(T)} \rightarrow \tilde{\rho}$ in $X$, where $\tilde{\rho}$ is the solution associated to $D_{\mathcal{X}}z$. This proves that $\Psi$ is compact on $X$.

The continuity of $\Psi$ can be obtained by using the same method as the compactness.

It remains to check that the other condition of Leray-Schauder theorem holds true. Here we fix $(\rho,\sigma) \in X \times [0,1]$ such that $\rho=\sigma \Psi(\rho)$ and let $z$ be the solution to the equation \eqref{linear system_z}. It can be implied that $(z,\rho)$ satisfies
\begin{equation}\label{linear system_LST}
\begin{cases}
-\partial_t z-\Delta_{\mathcal{X}} z+V \cdot D_{\mathcal{X}} z=\frac{\delta F}{\delta m}(x,m(t))(\rho(t))+b,& \text {in }[0, T] \times \mathbb{T}^2, \\
\partial_t\rho-\Delta_\mathcal{X}\rho-\operatorname{div}_{\mathcal{X}}(\rho V)-\sigma\operatorname{div}_{\mathcal{X}}(m \Gamma D_{\mathcal{X}}z+c)=0,& \text {in }[0, T] \times \mathbb{T}^2, \\
z(T, x)=\frac{\delta G}{\delta m}(x,m(T))(\rho(T))+z_T,\quad \rho(0)=\sigma\rho
_0,& \text {in }\mathbb{T}^2.
\end{cases}
\end{equation}
We need to show that $\rho$ satisfies \eqref{linear system_rho regularity}, and this will be proved in the next step.

\emph{Step 2: Estimate of $\rho$.} It can be noticed that
\begin{gather*}
\frac{\delta F}{\delta m}(x,m(t))(\rho(t))+b(t,x) \in C\left([0,T];C_{\mathcal{X}}^{1+\alpha}(\mathbb{T}^2)\right), \\
\frac{\delta G}{\delta m}(x,m(T))(\rho(T))+z_T(x) \in C_{\mathcal{X}}^{1+\alpha}(\mathbb{T}^2),
\end{gather*}
so we can use $z$ as a test function for the equation of $\rho$ to obtain the weak formulation:
\begin{align*}
& \int_{\mathbb{T}^2}\left(\rho(T,x)z(T,x)-\sigma\rho_{0}(x)z(0,x)\right)dx \\
= & -\sigma \int_{0}^{T} \int_{\mathbb{T}^2} c(t,x) \cdot D_{\mathcal{X}}z(t,x)dxdt \\
& -\int_{0}^{T}\int_{\mathbb{T}^2} \left(\frac{\delta F}{\delta m}(x,m(t))(\rho(t))+b(t,x)\right)\rho(t,x)dxdt \\
& -\sigma \int_{0}^{T} \int_{\mathbb{T}^2}D_{\mathcal{X}}z(t,x) \cdot \left(\Gamma(t,x)D_{\mathcal{X}}z(t,x)\right)m(t,x)dxdt.
\end{align*}
It can be estimated by the terminal condition of $z$ together with the properties of $F$ and $G$ in \eqref{property_F} and \eqref{property_G} that
\begin{equation}\label{energy estimate}
\begin{aligned}
& \sigma \int_{0}^{T} \int_{\mathbb{T}^2}D_{\mathcal{X}}z(t,x) \cdot \left(\Gamma(t,x)D_{\mathcal{X}}z(t,x)\right)m(t,x)dxdt \\
\leq & \sup_{t \in [0,T]}\left\|\rho(t)\right\|_{-(1+\alpha)} \left(\left\|z_T\right\|_{2+\alpha}+\sup_{t \in [0,T]}\left\|b(t,\cdot)\right\|_{1+\alpha}\right) \\
& + \sigma\sup_{t \in [0,T]}\left\|z(t,\cdot)\right\|_{2+\alpha} \left(\left\|\rho_0\right\|_{-(1+\alpha)}+\|c\|_{L^1}\right) \\
\leq & M\left(\sup_{t \in [0,T]}\left\|\rho(t)\right\|_{-(1+\alpha)} + \sigma\sup_{t \in [0,T]}\left\|z(t,\cdot)\right\|_{2+\alpha}\right).
\end{aligned}
\end{equation}
On the other hand, using Lemma \ref{Lem_general KFP wellposed regularity} and Remark \ref{Rem_f=div(c)} we have
\begin{equation}\label{rho space regularity 1}
\sup_{t \in [0,T]}\left\|\rho(t)\right\|_{-(1+\alpha)} \leq C\left(\sigma\left\|m\Gamma D_{\mathcal{X}}z\right\|_{L^1}+\|c\|_{L^1}+\left\|\rho_0\right\|_{-(1+\alpha)}\right).
\end{equation}
As for the first term on the right hand side of the above inequality, since $\Gamma(t,x)$ is a real symmetric matrix and satisfies \eqref{Assump_linear system}, we can use H\"{o}lder's inequality and \eqref{energy estimate} to get
\begin{align*}
\sigma\left\|m\Gamma D_{\mathcal{X}}z\right\|_{L^1}= & \sigma\int_{0}^{T}\int_{\mathbb{T}^2} m \left(D_{\mathcal{X}}z\cdot\left(\Gamma^2 D_{\mathcal{X}}z\right)\right)^{\frac{1}{2}}dxdt \\
\leq & \sigma\left(\int_{0}^{T}\int_{\mathbb{T}^2} m\left(D_{\mathcal{X}}z\cdot\left(\Gamma^2 D_{\mathcal{X}}z\right)\right)dxdt\right)^\frac{1}{2} \left(\int_{0}^{T}\int_{\mathbb{T}^2} mdxdt\right)^\frac{1}{2} \\
\leq & C M^\frac{1}{2} \left(\sup_{t \in [0,T]}\left\|\rho(t)\right\|_{-(1+\alpha)}^{\frac{1}{2}} + \sup_{t \in [0,T]}\left\|z(t,\cdot)\right\|_{2+\alpha}^{\frac{1}{2}}\right).
\end{align*}
Putting the above estimate into \eqref{rho space regularity 1} one has
\begin{equation*}
\sup_{t \in [0,T]}\left\|\rho(t)\right\|_{-(1+\alpha)} \leq C\left(M^\frac{1}{2} \left(\sup_{t \in [0,T]}\left\|\rho(t)\right\|_{-(1+\alpha)}^{\frac{1}{2}} + \sup_{t \in [0,T]}\left\|z(t,\cdot)\right\|_{2+\alpha}^{\frac{1}{2}}\right)+M\right).
\end{equation*}
Using the $\varepsilon$-Cauchy inequality with suitable coefficients we get
\begin{equation*}
\sup_{t \in [0,T]}\left\|\rho(t)\right\|_{-(1+\alpha)} \leq C\left(M^\frac{1}{2} \sup_{t \in [0,T]}\left\|z(t,\cdot)\right\|_{2+\alpha}^{\frac{1}{2}}+M\right).
\end{equation*}
Combining this with the estimate of $z$ in \eqref{z regularity} and using $\varepsilon$-Cauchy inequality again, we finally get a estimate for $\rho$:
\begin{equation}\label{rho space regularity 2}
\sup_{t \in [0,T]}\left\|\rho(t)\right\|_{-(1+\alpha)} \leq CM,
\end{equation}
thus yielding the estimates for $z$. At this point we have completed the proof of existence.

It remains to prove the time regularity of $\rho$, which can also be estimated by duality. Let $t' \in (0,T]$, $\psi \in C_{\mathcal{X}}^{1+\alpha}(\mathbb{T}^2)$ and $\phi$ be the solution to the backward equation
\begin{equation*}
\begin{cases}
-\partial_t\phi-\Delta_\mathcal{X}\phi+V(t,x)\cdot D_{\mathcal{X}}\phi=0, & \text{ in } [0,t') \times \mathbb{T}^2, \\
\phi(t')=\psi, & \text{ in } \mathbb{T}^2.
\end{cases}
\end{equation*}
Lemma \ref{Lem_homo. LHJB Schauder estimate} states that
\begin{equation}\label{test fun. regularity}
\left\|\phi\right\|_{C_{\mathcal{X}}^{\frac{\alpha}{2},1+\alpha}([0,t']\times\mathbb{T}^2)} \leqslant C\left\|\psi\right\|_{C_{\mathcal{X}}^{1+\alpha}(\mathbb{T}^2)},
\end{equation}
where $C$ depends on $\left\|V\right\|_{C_{\mathcal{X}}^{\frac{\alpha}{2},1+\alpha}\left([0,t']\times\mathbb{T}^2\right)}$. Choose $\phi$ as a test function for the equation of $\rho$ in \eqref{linear system_LST}, and use the H\"{o}lder estimate with respect to time in \eqref{test fun. regularity}, then we have, for any $t \in [0,t']$,
\begin{align*}
& \int_{\mathbb{T}^2}\psi(x)(\rho(t',x)-\rho(t,x))dx \\
= & \int_{\mathbb{T}^2}(\phi(t,x)-\phi(t',x))\rho(t,x)dx \\
& -\sigma \int_{t}^{t'}\int_{\mathbb{T}^{2}} D_{\mathcal{X}}\phi(s,x)\cdot\left(\Gamma(s,x)D_{\mathcal{X}}z(s,x)\right)m(s,x) dxds \\
& -\sigma \int_{t}^{t'}\int_{\mathbb{T}^{2}}c(s,x)\cdot D_{\mathcal{X}}\phi(s,x)dxds \\
\leqslant & C(t'-t)^{\frac{\alpha}{2}}\|\psi\|_{1+\alpha}\sup_{t \in [0,T]} \|\rho(t)\|_{-(1+\alpha)} \\
& +\bar{C}^{\frac{1}{2}}(t'-t)^{\frac{1}{2}}\left\|D_{\mathcal{X}}\phi\right\|_{\infty} \left(\int_{0}^{T}\int_{\mathbb{T}^{2}} D_{\mathcal{X}}z(s,x)\cdot\left(\Gamma(s,x)D_{\mathcal{X}}z(s,x)\right)m(s,x)dxds\right)^{\frac{1}{2}} \\
& +(t'-t)\sup_{t \in [0,T]}\|c(t)\|_{L^1}\sup_{t \in [0,T]}\|\phi(t,\cdot)\|_{1+\alpha},
\end{align*}
where the second term on the right hand side of the above inequality is obtained from the H\"{o}lder inequality.

Using \eqref{energy estimate}, \eqref{rho space regularity 2} together with \eqref{test fun. regularity}, we obtain
\begin{align*}
& \int_{\mathbb{T}^2}\psi(x)(\rho(t',x)-\rho(t,x))dx \\
\leq & C(t'-t)^{\frac{\alpha}{2}}\|\psi\|_{1+\alpha}\left(M^\frac{1}{2} \sup_{t \in [0,T]}\left\|z(t,\cdot)\right\|_{2+\alpha}^{\frac{1}{2}}+M\right).
\end{align*}
Dividing both sides by $(t'-t)^{\frac{\alpha}{2}}$, taking the supremum over $\psi$ and combining with the estimate of $z$ yields
$$
\sup _{\substack{t \neq t'\\t,t' \in [0,T]}} \frac{\left\|\rho\left(t^{\prime}\right)-\rho(t)\right\|_{-(1+\alpha)}}{\left|t^{\prime}-t\right|^\frac{\alpha}{2}} \leq C\left(M^\frac{1}{2} \sup_{t \in [0,T]}\left\|z(t,\cdot)\right\|_{2+\alpha}^{\frac{1}{2}}+M\right) \leq CM.
$$

\emph{Step 3: Uniqueness.} Let $\left(z_{1}, \rho_{1}\right)$ and $\left(z_{2}, \rho_{2}\right)$ be two solutions of system \eqref{general linear MFG system}. Therefore, the couple $(\bar{z},\bar{\rho}):=\left(z_{1}-z_{2}, \rho_{1}-\rho_{2}\right)$ satisfies the linear system as follows
\begin{equation*}
\begin{cases}
-\partial_t \bar{z}-\Delta_{\mathcal{X}}\bar{z}+V\cdot D_{\mathcal{X}}\bar{z}=\frac{\delta F}{\delta m}(x,m(t))(\bar{\rho}(t)), & \text{in } [0,T]\times\mathbb{T}^2,\\
\partial_t\bar{\rho}-\Delta_{\mathcal{X}}\bar{\rho}- \operatorname{div}_{\mathcal{X}}\left(\bar{\rho}V\right)-\operatorname{div}_{\mathcal{X}}\left(m\Gamma D_{\mathcal{X}}\bar{z}\right)=0, & \text{in } [0,T]\times\mathbb{T}^2, \\
\bar{z}(T,x)=\frac{\delta G}{\delta m}(x,m(T))(\bar{\rho}(T)),\quad \bar{\rho}\left(t_{0}\right)=0, & \text{in } \mathbb{T}^2.
\end{cases}
\end{equation*}
From the already proved estimates \eqref{linear system_z regularity} and \eqref{linear system_rho regularity}, it follows that
$$
\sup_{t \in \left[0,T\right]}\|\bar{z}(t,\cdot)\|_{2+\alpha} + \sup_{t \in \left[0,T\right]}\|\bar{\rho}(t)\|_{-(1+\alpha)} \leq 0,
$$
hence $\bar{z}=0$, $\bar{\rho}=0$. This concludes the lemma.
\end{proof}
Applying Lemma \ref{Lem_GL system regularity} to the linearized MFG system \eqref{linear MFG system} yields the following important lemma.
\begin{Lem}
Assume that H\ref{hyp1}) and H\ref{hyp2}) hold. If $m_{0} \in \mathcal{P}\left(\mathbb{T}^{2}\right)$ and $\rho_{0} \in C_{\mathcal{X}}^{-1}(\mathbb{T}^2)$, then there is a unique solution $(z,\rho)$ of system \eqref{linear MFG system} and this solution satisfies
\begin{equation}\label{linear MFG system_regularity}
\sup _{t \in \left[t_{0},T\right]}\left\{\|z(t,\cdot)\|_{2+\alpha}+\|\rho(t)\|_{-(1+\alpha)}\right\} \leqslant C\left\|\rho_{0}\right\|_{-(1+\alpha)},
\end{equation}
and
\begin{equation}\label{linear MFG system_lower regularity}
\sup _{t \in \left[t_{0},T\right]}\left\{\|z(t,\cdot)\|_{2+\alpha}+\|\rho(t)\|_{-(2+\alpha)}\right\} \leqslant C\left\|\rho_{0}\right\|_{-(2+\alpha)},
\end{equation}
where the constant $C$ depends on $T$, $H$, $F$ and $G$, but not on $\left(t_{0},m_{0}\right)$.
\end{Lem}
Take note that the map $\rho_{0} \rightarrow (z,\rho)$ is linear and continuous from $C_{\mathcal{X}}^{-(2+\alpha)}(\mathbb{T}^2)$ into $C\left(\left[t_{0},T\right];C_{\mathcal{X}}^{2+\alpha}(\mathbb{T}^2) \times C_{\mathcal{X}}^{-(1+\alpha)}(\mathbb{T}^2)\right)$.
\begin{proof}
This is a direct application of Lemma \ref{Lem_GL system regularity} and Corollary \ref{Cor_(L)KFP lower regularity}, with the coefficients $V(t,x)=D_pH(x,D_{\mathcal{X}}u(t,x))$, $\Gamma(t,x)=D_{pp}^2 H(x,D_{\mathcal{X}}u(t,x))$ and $z_{T}=b=c=0$. Note that in the light of Proposition \ref{Prop_MFG system wellposed regularity}, $V$ belongs to $C_{\mathcal{X}}^{\frac{\alpha}{2},1+\alpha}\left([0,T] \times \mathbb{T}^2\right)$.
\end{proof}
Using the above lemma we can then prove Lemma \ref{Lem_relation of z and rho_0}.
\begin{proof}[Proof of Lemma \ref{Lem_relation of z and rho_0}]
Let us denote $z(t,x;\rho_0)$ as the solution of the first component of system \eqref{linear MFG system} related to $\rho_0$. For any $y \in \mathbb{T}^2$, setting $\rho_0=\delta_y$, the Dirac measure mass at $y$, we can define
$$
K(t_0,x,m_0,y):=z(t_0,x;\delta_y).
$$
Thanks to the linearity of the system, we easily get the representation formula \eqref{linear MFG system_rep.formula} since the density of simple measures. Moreover, for any $y \in \mathbb{T}^2$, let $\gamma(t)$ be the absolutely continuous integral curve of the vector fields $\{X_1,X_2\}$ such that
$$
\begin{cases}
\gamma^{\prime}(t)=X_i(\gamma(t)), \\
\gamma(0)=y.
\end{cases}
$$
Then
\begin{equation*}
\frac{K(t_0,x,m_0,\gamma(t))-K(t_0,x,m_0,y)}{t}=z(t_0,x;\frac{\delta_{\gamma(t)}-\delta_y}{t}).
\end{equation*}
Denoting $\Delta_{t}^y f:=\frac{f(\gamma(t))-f(y)}{t}$ for any map $f$, we need to prove the limit exists when $t \rightarrow 0$, for this we consider the Cauchy sequence. Hence, estimate \eqref{linear MFG system_regularity} and the Lagrange mean value theorem implies that, for any $t_1,t_2 >0$,
\begin{align*}
& \left\|\Delta_{t_1}^y K(t_0,\cdot,m_0,y)-\Delta_{t_2}^y K(t_0,\cdot,m_0,y)\right\|_{2+\alpha}\\
\leq & C\left\|\Delta_{t_1}^y \delta_y-\Delta_{t_2}^y \delta_y\right\|_{-(1+\alpha)} \\
= & C\sup_{\|\psi\|_{1+\alpha}\leq 1}\left(\Delta_{t_1}^y \psi-\Delta_{t_2}^y \psi\right) \\
= & C\sup_{\|\psi\|_{1+\alpha}\leq 1}\left(\frac{d\psi\left(\gamma(t)\right)}{dt}\bigg|_{t=\lambda_1 t_1}-\frac{d\psi\left(\gamma(t)\right)}{dt}\bigg|_{t=\lambda_2 t_2}\right) \\
= & C\sup_{\|\psi\|_{1+\alpha}\leq 1}\left(X_i\psi\left(\gamma(\lambda_1t_1)\right)-X_i\psi\left(\gamma(\lambda_2t_2)\right)\right)\\
\leq & Cd_{cc}\left(\gamma(\lambda_1t_1),\gamma(\lambda_2t_2)\right)^{\alpha} \leq C\left|\lambda_1t_1-\lambda_2t_2\right|^{\alpha},
\end{align*}
where $\lambda_j \in (0,1), j=1,2$. Letting $t_1,t_2 \rightarrow 0$, we get that $\Delta_{t}^y K$ is a Cauchy sequence and the limit exists, that is
$$
X_i K(t_0,x,m_0,y)=z(t_0,x;X_i^* \delta_y).
$$
Replacing $K$ with $X_i K$ and repeating the above steps, by using the estimate \eqref{linear MFG system_lower regularity} we then obtain that
$$
X_j X_i K(t_0,x,m_0,y)=z(t_0,x;X_i^*X_j^* \delta_y).
$$
Consider any $y,y' \in \mathbb{T}^2$, estimate \eqref{linear MFG system_lower regularity} combined with the linearity of system \eqref{linear MFG system} implies that
\begin{align*}
\left\|X_j X_i K(t_0,\cdot,m_0,y)-X_j X_i K(t_0,\cdot,m_0,y')\right\|_{2+\alpha} \leq
 & C\left\|X_i^* X_j^* \delta_y-X_i^* X_j^* \delta_{y'}\right\|_{-(2+\alpha)} \\
\leq & C\left\|\delta_y-\delta_{y'}\right\|_{-\alpha} \\
\leq & Cd_{cc}(y,y')^{\alpha}.
\end{align*}
Consequently we can get
\begin{equation*}
\left\|K\left(t_{0}, \cdot, m_{0}, \cdot\right)\right\|_{(2+\alpha,2+\alpha)} \leqslant C,
\end{equation*}
where $C$ does not depend on $(t_0,m_0)$. And it follows from both the stability of the MFG system \eqref{MFG system} and the linearized system \eqref{linear MFG system}, that $K$ and its derivatives in $(x,y)$ are continuous with respect to $(t_0,m_0)$.
\end{proof}
Let us now prove that $K$ is in fact the derivative of $U$ with respect to $m$.

\begin{proof}[Proof of Proposition \ref{Prop_U C^1 w.r.t. m}]
The method of proof is similar to \cite[Proposition 3.4.3]{19CDLL}, so we only explain the key parts. Set $\nu:=\hat{u}-u-z$ and $\mu:=\hat{m}-m-\rho$, then the pair $(\nu,\mu)$ satisfies the following linear system:
\begin{equation*}
\begin{cases}
-\partial_t \nu-\Delta_{\mathcal{X}} \nu+D_pH(x,D_{\mathcal{X}}u) \cdot D_{\mathcal{X}} \nu=\frac{\delta F}{\delta m}(x,m(t))(\mu(t))+b(t,x), \\
\partial_t\mu-\Delta_\mathcal{X}\mu-\operatorname{div}_{\mathcal{X}}\left(\mu D_pH(x,D_{\mathcal{X}}u)\right)\\
\quad -\operatorname{div}_{\mathcal{X}}\left(m D_{pp}^2H(x,D_{\mathcal{X}}u) D_{\mathcal{X}}\nu\right)-\operatorname{div}_{\mathcal{X}}(c)=0, \\
\nu(T, x)=\frac{\delta G}{\delta m}(x,m(T))(\mu(T))+z_T(x),\quad \mu(t_0)=0,
\end{cases}
\end{equation*}
where
\begin{align*}
b(t, x)= & -\int_{0}^{1}\left(D_{p} H\left(x, s D_{\mathcal{X}} \hat{u}+(1-s) D_{\mathcal{X}} u\right)-D_{p} H\left(x, D_{\mathcal{X}} u\right)\right) \cdot D_{\mathcal{X}}(\hat{u}-u) d s \\
& + \int_{0}^{1} \int_{\mathbb{T}^{2}}\left(\frac{\delta F}{\delta m}(x, s \hat{m}(t)+(1-s) m(t), y)\right. \\
& \left.-\frac{\delta F}{\delta m}(x, m(t), y)\right) d(\hat{m}(t)-m(t))(y) ds,
\end{align*}
\begin{align*}
c(t)= & (\hat{m}-m)(t) D_{pp}^{2} H\left(x, D_{\mathcal{X}} u(t, x)\right)\left(D_{\mathcal{X}}\hat{u}-D_{\mathcal{X}} u\right)(t, x) \\
& + \hat{m} \int_{0}^{1}\left(D_{pp}^{2} H\left(x, s D_{\mathcal{X}} \hat{u}(t, x)+(1-s) D_{\mathcal{X}} u(t, x)\right)\right. \\
& \left.-D_{pp}^{2} H\left(x, D_{\mathcal{X}} u(t, x)\right)\right)\left(D_{\mathcal{X}} \hat{u}-D_{\mathcal{X}} u\right)(t, x) d s
\end{align*}
and
\begin{align*}
z_{T}(x)= & \int_{0}^{1} \int_{\mathbb{T}^{2}}\left(\frac{\delta G}{\delta m}(x, s \hat{m}(T)+(1-s) m(T), y)\right. \\
& \left.-\frac{\delta G}{\delta m}(x, m(T), y)\right) d(\hat{m}(T)-m(T))(y) ds.
\end{align*}
We apply Lemma \ref{Lem_GL system regularity} to obtain an estimate of the solution $(\nu,\mu)$, and then estimate the coefficients based on the hypotheses H\ref{hyp1})-H\ref{hyp2}). We obtain the estimates for $b$ and $z_T$ using methods nearly analogous to \cite{19CDLL}, while $c$ needs to be estimated as follows
\begin{equation*}
\left\|c(t)\right\|_{L^1} \leq C\left(\left\|\left(u-\hat{u}\right)(t,\cdot)\right\|_{C_{\mathcal{X}}^2} d_1\left(m(t),\hat{m}(t)\right) +\left\|\left(u-\hat{u}\right)(t,\cdot)\right\|_{C_{\mathcal{X}}^1}^2\right).
\end{equation*}
By Proposition \ref{Prop_Lip. ctn. of U} we get
\begin{equation*}
\sup_{t \in [t_0,T]} \left\|c(t)\right\|_{L^1} \leq C d_1^2\left(m_0,\hat{m}_0\right).
\end{equation*}
Thus we obtain the final conclusion.
\end{proof}
Last but not least, the Lipschitz continuity of $\frac{\delta F}{\delta m}$ and $\frac{\delta G}{\delta m}$ with respect to $m$ can imply the Lipschitz continuity of $\frac{\delta U}{\delta m}$ with respect to $m$.
\begin{proof}[Proof of Proposition \ref{Prop_derivative Lip. w.r.t. m}]
For any $y \in \mathbb{T}^2$, set $\rho_0=\delta_y$, let us set $(z,\rho):=(z_1-z_2,\rho_1-\rho_2)$, where $(z_1,\rho_1)$ and $(z_2,\rho_2)$ are the solutions of the linear MFG system \eqref{linear MFG system} related to $(u_1,m_1)$ and $(u_2,m_2)$, respectively. The next steps are similar to the proof of Proposition \ref{Prop_U C^1 w.r.t. m}, and we may refer to \cite[Proposition 3.6.1]{19CDLL} for details.

For $i=1,2$, let $(z^i,\rho^i)$ be the solution of the linear MFG system \eqref{linear MFG system} related to $(u^i,m^i)$ and $\rho^i(t_0)=\rho_0$ for any $\rho_0 \in C_{\mathcal{X}}^{-1}\left(\mathbb{T}^2\right)$. Set $(z, \rho):=\left(z^1-z^2,\rho^1-\rho^2\right)$. To simplify the notation, we denote $H_1^{\prime}(t, x)=D_p H\left(x, D_{\mathcal{X}}u^1(t, x)\right)$, $H_1^{\prime \prime}(t, x)=D_{p p}^2 H\left(x, D_{\mathcal{X}}u^1(t, x)\right)$, $F_1^{\prime}(x, \rho)=\int_{\mathbb{T}^2} \frac{\delta F}{\delta m}\left(x, m^1, y\right) \rho(y) d y$, etc... Then $(z, \rho)$ satisfies
\begin{equation*}
\begin{cases}
-\partial_t z-\Delta_{\mathcal{X}} z+H_1^{\prime} \cdot D_{\mathcal{X}}z =F_1^{\prime}(\cdot, \rho)+b,& \text{in }\left[t_0, T\right] \times \mathbb{T}^2, \\
\partial_t \rho-\Delta_{\mathcal{X}}\rho-\operatorname{div}_{\mathcal{X}}\left(\rho H_1^{\prime}\right)-\operatorname{div}_{\mathcal{X}}\left(m^1 H_1^{\prime \prime} D_{\mathcal{X}}z\right)-\operatorname{div}_{\mathcal{X}}(c)=0,& \text {in }\left[t_0, T\right] \times \mathbb{T}^2, \\
z(T)=G_1^{\prime}(\rho(T))+z_T, \rho\left(t_0\right)=0,& \text {in } \mathbb{T}^2,
\end{cases}
\end{equation*}
where
\begin{align*}
& b(t, x):=F_1^{\prime}\left(x, \rho^2(t)\right)-F_2^{\prime}\left(x, \rho^2(t)\right)-\left[\left(H_1^{\prime}-H_2^{\prime}\right) \cdot D_{\mathcal{X}}z^2\right](t, x), \\
& c(t, x):=\rho^2(t, x)\left(H_1^{\prime}-H_2^{\prime}\right)(t, x)+\left[\left(m^1 H_1^{\prime \prime}-m^2 H_2^{\prime \prime}\right) D_{\mathcal{X}}z^2\right](t, x), \\
& z_T(x):=G_1^{\prime}\left(\rho^2(T)\right)-G_2^{\prime}\left(\rho^2(T)\right).
\end{align*}
By applying Lemma \ref{Lem_GL system regularity} with $V=H_1^{\prime}$ and $\Gamma=H_1^{\prime \prime}$, we have
\begin{equation*}
\sup_{t \in \left[t_{0},T\right]}\|z(t,\cdot)\|_{2+\alpha} \leq C\left(\left\|z_{T}\right\|_{2+\alpha}+\sup_{t \in \left[t_{0},T\right]}\left(\|b(t,\cdot)\|_{1+\alpha}+\|c(t)\|_{L^1}\right)\right).
\end{equation*}
Thanks to hypotheses H\ref{hyp1})-H\ref{hyp2}), let us estimate the various terms in the right-hand side as follows
\begin{align*}
& \left\|z_T\right\|_{2+\alpha} \\
\leq &\left\|\int_{\mathbb{T}^2}\left(\frac{\delta G}{\delta m}\left(0, \cdot, m^1(T), y\right)-\frac{\delta G}{\delta m}\left(0, \cdot, m^2(T), y\right)\right) \rho^2(T, y) d y\right\|_{2+\alpha} \\
\leq &\left\|\frac{\delta G}{\delta m}\left(0, \cdot, m^1(T), \cdot\right)-\frac{\delta G}{\delta m}\left(0, \cdot, m^2(T), \cdot\right)\right\|_{(2+\alpha, 2+\alpha)}\left\|\rho^2(T)\right\|_{-(2+\alpha)} \\
\leq & C d_1\left(m_0^1, m_0^2\right)\left\|\rho_0\right\|_{-(2+\alpha)},
\end{align*}
where the last inequality is obtained from Proposition \ref{Prop_Lip. ctn. of U} and estimate \eqref{linear MFG system_lower regularity}. By a similar argument, we have
\begin{align*}
\|b(t, \cdot)\|_{1+\alpha} \leq & \left\|F_1^{\prime}\left(\cdot, \rho^2(t)\right)-F_2^{\prime}\left(\cdot, \rho^2(t)\right)\right\|_{1+\alpha} \\
& +\left\|\left(H_1^{\prime}-H_2^{\prime}\right)(t, \cdot) D_{\mathcal{X}}z^2(t, \cdot)\right\|_{1+\alpha},
\end{align*}
where the first term can be estimated as
$$
\left\|F_1^{\prime}\left(\cdot, \rho^2(t)\right)-F_2^{\prime}\left(\cdot, \rho^2(t)\right)\right\|_{1+\alpha} \leq C d_1\left(m_0^1, m_0^2\right)\left\|\rho_0\right\|_{-(2+\alpha)},
$$
and the second one is bounded by
\begin{align*}
& \left\|\left(H_1^{\prime}-H_2^{\prime}\right)(t, \cdot) D_{\mathcal{X}}z^2(t, \cdot)\right\|_{1+\alpha} \\
= & \left\|\left(D_p H\left(\cdot, D_{\mathcal{X}}u^1(t, \cdot)\right)-D_p H\left(\cdot, D_{\mathcal{X}}u^2(t, \cdot)\right)\right) D_{\mathcal{X}}z^2(t, \cdot)\right\|_{1+\alpha} \\
\leq & \left\|\left(u^1-u^2\right)(t, \cdot)\right\|_{2+\alpha}\left\|z^2(t, \cdot)\right\|_{2+\alpha} \\
\leq & C d_1\left(m_0^1, m_0^2\right)\left\|\rho_0\right\|_{-(2+\alpha)}.
\end{align*}
Moreover,
\begin{align*}
\|c(t)\|_{L^1} \leq & C\left\|\left(u^1-u^2\right)(t, \cdot)\right\|_{2+\alpha}\left\|\rho^2(t,\cdot)\right\|_{-(1+\alpha)} \\
& +C d_1\left(m^1(t), m^2(t)\right)\left\|z^2(t, \cdot)\right\|_{C_{\mathcal{X}}^2}+C\left\|\left(u^1-u^2\right)(t, \cdot)\right\|_{C_{\mathcal{X}}^1}\left\|z^2(t, \cdot)\right\|_{C_{\mathcal{X}}^1} \\
\leq & C d_1\left(m_0^1, m_0^2\right)\left\|\rho_0\right\|_{-(1+\alpha)},
\end{align*}
where the last inequality is based on the estimate \eqref{linear MFG system_regularity}.

Combining the above inequalities yields
\begin{equation*}
\sup _{t \in\left[t_0, T\right]}\|z(t, \cdot)\|_{2+\alpha} \leq C d_1\left(m_0^1, m_0^2\right)\left\|\rho_0\right\|_{-(1+\alpha)}.
\end{equation*}
Since
\begin{equation*}
z\left(t_0, x\right)=\int_{\mathbb{T}^2}\left(\frac{\delta U}{\delta m}\left(t_0, x, m_0^1, y\right)-\frac{\delta U}{\delta m}\left(t_0, x, m_0^2, y\right)\right) \rho_0(y) d y,
\end{equation*}
we can prove the final result by respectively choosing $\rho_0=X_i^*\delta_y$ and $\rho_0=X_i^*\delta_y-X_i^*\delta_{y'}$ for any $y,y' \in \mathbb{T}^2$ and $i=1,2$.
\end{proof}

\section{Solvability of the first-order Master Equation}\label{Sec_6}

We have by now obtained the desired properties of the solution to the degenerate MFG system and the $C^1$ differentiability of $U$ with respect to the measure. Those are essential for the proof of the main theorem.
\begin{proof}[Proof of Theorem \ref{Thm_ME wellposed regularity}.]
The proof follows closely from \cite{19CDLL}. We give the details only for sake of completeness.

\emph{Step 1: Existence.}
Assume $(u,m)$ is the solution to the degenerate MFG system \eqref{MFG system} with the initial condition $m(t_0)=m_0 \in \mathcal{P}\left(\mathbb{T}^2\right)$, and let $U(t_0,x,m_0):=u(t_0,x)$. We need to prove that $U$ is a solution to the master equation \eqref{ME}. Taking a smooth sequence $\{m_0^n\}_n$ to approximate $m_0$ and denoting the corresponding solutions to system \eqref{MFG system} as $\left(u^n,m^n\right)$, we compute
\begin{align*}
\partial_t U\left(t_0,x,m_0^n\right) = & \lim_{h \rightarrow 0}\frac{U\left(t_0+h,x,m_0^n\right)-U\left(t_0,x,m_0^n\right)}{h} \\
= & \lim_{h \rightarrow 0}\frac{U\left(t_0+h,x,m_0^n\right)-U\left(t_0+h,x,m^n\left(t_0+h\right)\right)}{h} \\
& +  \lim_{h \rightarrow 0}\frac{U\left(t_0+h,x,m^n\left(t_0+h\right)\right)-U\left(t_0,x,m_0^n\right)}{h}\\
=: & I+II.
\end{align*}
Set $m_s^n=(1-s)m^n\left(t_0\right)+sm^n\left(t_0+h\right)$. Since $U$ is $C^1$ differentiable with respect to the measure and $m^n$ is smooth satisfying the KFP equation in system \eqref{MFG system}, then we have by the continuity of $\frac{\delta U}{\delta m}$ and integrating by parts that
\begin{align*}
I = & \lim_{h \rightarrow 0} \int_0^1 \int_{\mathbb{T}^2} \frac{\delta U}{\delta m}\left(t_0+h, x, m_s^n, y\right)\left(\frac{m^n\left(t_0, y\right)-m^n\left(t_0+h, y\right)}{h}\right) d y d s \\
= & \!\int_{\mathbb{T}^2}\! \bigg[\frac{\delta U}{\delta m}\left(t_0, x, m_0^n, y\right)\left(\!-\!\Delta_{\mathcal{X}} m^n(t_0,y)\!-\!\operatorname{div}_{\mathcal{X}}\left(m^n(t_0,y)D_pH(y, D_{\mathcal{X}}u^n(t_0,y))\right)\right)\!\bigg]\!d y \\
= & \int_{\mathbb{T}^2} -\Delta_{\mathcal{X}}^y \frac{\delta U}{\delta m}\left(t_0, x, m_0^n, y\right) dm_0^n(y) \\
& +\int_{\mathbb{T}^2} D_{\mathcal{X}}^y \frac{\delta U}{\delta m}\left(t_0, x, m_0^n, y\right) \cdot D_p H(y, D_{\mathcal{X}}U(t_0,y,m_0^n)) dm_0^n(y).
\end{align*}
On the other hand,
\begin{align*}
II=& \partial_t u^n(t_0,x)=-\Delta_{\mathcal{X}}u^n(t_0,x)+H\left(x, D_{\mathcal{X}} u^n(t_0,x)\right)-F(x, m_0^n)\\
= & -\Delta_{\mathcal{X}}U(t_0,x,m_0^n)+H\left(x, D_{\mathcal{X}} U(t_0,x,m_0^n)\right)-F(x, m_0^n).
\end{align*}
Combining the above two equalities, we can then let $n \rightarrow \infty$ thanks to the continuity of both sides of the equality, thus obtaining that $\partial_t U\left(t_0,x,m_0\right)$ exists, namely for any $(t_0,x,m_0) \in [0,T] \times \mathbb{T}^2 \times \mathcal{P}(\mathbb{T}^2)$,
\begin{align*}
\partial_t U\left(t_0,x,m_0\right)= & -\Delta_{\mathcal{X}}U(t_0,x,m_0)+H\left(x, D_{\mathcal{X}} U(t_0,x,m_0)\right)-F(x, m_0) \\
& -\int_{\mathbb{T}^2} \Delta_{\mathcal{X}}^y \frac{\delta U}{\delta m}\left(t_0, x, m_0, y\right) dm_0(y) \\
& +\int_{\mathbb{T}^2} D_{\mathcal{X}}^y \frac{\delta U}{\delta m}\left(t_0, x, m_0, y\right) \cdot D_p H(y, D_{\mathcal{X}}U(t_0,y,m_0)) dm_0(y).
\end{align*}
This concludes the existence part.

\emph{Step 2: Uniqueness.} Let $V(t,x,m)$ be another solution to the master equation \eqref{ME} in the sense of Definition \ref{Def_ME sol.}. For any fixed $t_0$ and smooth $m_0$, suppose $\tilde{m}$ is the $C_{\mathcal{X}}^{1,2}$ solution to the KFP equation
\begin{equation*}
\begin{cases}
\partial_t \tilde{m}-\Delta_{\mathcal{X}} \tilde{m}-\operatorname{div}_{\mathcal{X}}\left(\tilde{m}D_pH\left(x,
D_{\mathcal{X}}V(t,x,\tilde{m}(t))\right)\right)=0, & \text{in } [t_0,T] \times \mathbb{T}^2, \\
\tilde{m}(t_0)=m_0, & \text{in } \mathbb{T}^2.
\end{cases}
\end{equation*}
Note that this solution is well defined since $D_{\mathcal{X}}V(t,x,\tilde{m}(t)) \in C_{\mathcal{X}}^{0,1}\left([t_0,T]\times\mathbb{T}^2\right)$ thanks to the Lipschitz continuity of $D_{\mathcal{X}}V$ with respect to the measure.

Let us define $\tilde{u}(t,x):=V(t,x,\tilde{m}(t))$. Due to the regularity properties of $V$, $\tilde{u}$ is at least of class $C_{\mathcal{X}}^{1,2}$ with
\begin{align*}
\partial_t \tilde{u}(t, x) = & \partial_t V(t, x, \tilde{m}(t))+\int_{\mathbb{T}^2}\frac{\delta V}{\delta m}(t, x, \tilde{m}(t),y) \partial_t\tilde{m}(t,y)dy \\
= & \partial_t V(t, x, \tilde{m}(t))+\int_{\mathbb{T}^2}\frac{\delta V}{\delta m}(t, x, \tilde{m}(t), y)\left(\Delta_{\mathcal{X}}\tilde{m}\right. \\
& +\operatorname{div}_{\mathcal{X}}\left(\tilde{m} D_pH\left(y,D_{\mathcal{X}}V(t, y, \tilde{m}(t))\right)\right)dy \\
= & \partial_t V(t, x, \tilde{m}(t))+\int_{\mathbb{T}^2} \Delta_{\mathcal{X}}^y\frac{\delta V}{\delta m}(t, x, \tilde{m}(t), y) d \tilde{m}(t)(y) \\
& -\int_{\mathbb{T}^2}D_{\mathcal{X}}^y \frac{\delta V}{\delta m}(t, x, \tilde{m}(t), y) \cdot D_p H\left(y, D_{\mathcal{X}} V(t, y, \tilde{m}(t))\right) d \tilde{m}(t)(y) \\
= & -\Delta_{\mathcal{X}}V(t, x, \tilde{m}(t))+H\left(x, D_{\mathcal{X}}V(t, x, \tilde{m}(t))\right)-F(x, \tilde{m}(t)),
\end{align*}
where the last equality is attained by the master equation. Thus we obtain that $\tilde{u}$ satisfies the HJB equation
\begin{equation*}
\partial_t \tilde{u}(t, x) = -\Delta_{\mathcal{X}}\tilde{u}(t, x)+H\left(x, D_{\mathcal{X}}\tilde{u}(t, x)\right)-F(x, \tilde{m}(t))
\end{equation*}
with terminal condition $\tilde{u}(T,x)=V(T,x,\tilde{m}(T))=G(x,\tilde{m}(T))$. Therefore $(\tilde{u},\tilde{m})$ is a solution to the MFG system \eqref{MFG system}. According to the uniqueness of the solution to the system, we have $(\tilde{u},\tilde{m})=(u,m)$ and in turn $V(t_0,x,m_0)=U(t_0,x,m_0)$.

In view of the stability of the solution to the MFG system we can generalize to the case whenever $m_0 \in \mathcal{P}(\mathbb{T}^2)$. Hence the uniqueness is proved.

The further regularity properties of the derivative $\frac{\delta U}{\delta m}$ is given by Proposition \ref{Prop_U C^1 w.r.t. m} and Proposition \ref{Prop_derivative Lip. w.r.t. m}.
\end{proof}
\noindent
\textbf{Acknowledgments}
\noindent
The authors thank sincerely to Professor Wilfrid Gangbo for his valuable suggestions. The authors are grateful to the referees for their careful reading and thoughtful comments.


\vspace{0.4cm}
\begin{thebibliography}{44}
	
\bibitem{02Al} G.K. Alexopoulos, Sub-laplacians with drift on lie groups of polynomial volume growth, Mem. Amer. Math. Soc., 2002, 155(739), 101.
\bibitem{19BCP} V. Bally, L. Caramellino, P. Pigato, Tube estimates for diffusions under a local strong H\"{o}rmander condition, Ann. Inst. H. Poincar\'{e} Probab. Statist., 2019, 55(4)2320-2369.
\bibitem{05BB} M. Bramanti, L. Brandolini, Estimates of BMO type for singular integrals on spaces of homogeneous type and applications to hypoelliptic PDEs, Rev. Mat. Iberoamericana, 2005, 21(2)511-556.
\bibitem{07BB} M. Bramanti, L. Brandolini, Schauder estimates for parabolic nondivergence operators of Hormander type, J. Differ. Equ., 2007, 234, 177-245.
\bibitem{10BBLU} M. Bramanti, L. Brandolini, E. Lanconeli, F. Uguzzoni, Non-divergence equations stnuctured on Hormander vector fields: heat kernels and Harnack inequalities, Mem. Amer. Math. Soc., 2010, 204(961).
\bibitem{19BCCD} E. Bayraktar, A. Cecchin, A. Cohen, F. Delarue, Finite state mean field games with Wright-Fisher common noise, Journal de Math\'{e}matiques Pures et Appliqu\'{e}es, 2019.
\bibitem{21Be_C} C. Bertucci, Monotone solutions for mean field games master equations: continuous state space and common noise, 2021, arXiv preprint arXiv: 2107.09531.
\bibitem{21Be_F} C. Bertucci, Monotone solutions for mean field games master equations: finite state space and optimal stopping, Journal de l'\'{E}cole polytechnique-Math\'{e}matiques, 2021, 8, 1099-1132.
\bibitem{15BFY} A. Bensoussan, J. Frehse, S. C. P. Yam, The master equation in mean field theory, J. Math. Pures et Appliqu\'{e}es., 2015, 103(6): 1441-1474.
\bibitem{17BFY} A. Bensoussan, J. Frehse, S. C. P. Yam, On the interpretation of the master equation, Stoc. Proc. App., 2017, 127(7): 2093-2137.
\bibitem{17BLPR} R. Buckdahn, J. Li, S. Peng, C. Rainer, Mean-field stochastic differential equations and associated PDEs, Ann. Probab., 2017, 45(2): 824-878.
\bibitem{24BM} M. Bansil, A.R. M\'{e}sz\'{a}ros, Hidden monotonicity and canonical transformations for mean field games and master equations, 2024, arXiv preprint arXiv: 2403.05426.
\bibitem{23BMM} M. Bansil, A.R. M\'{e}sz\'{a}ros, C. Mou, Global Well-Posedness of Displacement Monotone Degenerate Mean Field Games Master Equations, 2023, arXiv preprint arXiv: 2308.16167.
\bibitem{13Br} M. Bramanti, An Invitation to Hypoelliptic Operators and H\"{o}rmander's Vector Fields, 2013.
\bibitem{14CCD} J. Chassagneux, D. Crisan, F. Delarue, A Probabilistic Approach to Classical Solutions of the Master Equation for Large Population Equilibria, Memoirs of the American Mathematical Society, 2014.
\bibitem{14CD} R. Carmona, F. Delarue, The Master Equation for large population equilibriums, In: D. Crisan, B. Hambly, T. Zariphopoulou, eds, Stochastic Analysis and Applications, France: Springer, 2014.
\bibitem{18CD} R. Carmona, F. Delarue, Probabilistic Theory of Mean Field Games with Applications I: Mean Field FBSDEs, Control, and Games, 2018.
\bibitem{19CDLL} P. Cardaliaguet, F. Delarue, J. Lasry, P. Lions, The Master Equation and the Convergence Problem in Mean Field Games, Princeton University Press., 2019.
\bibitem{15CGPT} P. Cardaliaguet, P. J. Graber, A. Porretta, D. Tonon, Second order mean field games with degenerate diffusion and local coupling, Nonlinear Differ. Equ. Appl. NoDEA, 2015, 22: 12871317.
\bibitem{21CS} P. Cardaliaguet, P.E. Souganidis, Weak solutions of the master equation for Mean Field Games with no idiosyncratic noise, 2021, arXiv preprint arXiv: 2109.14911.
\bibitem{22CSS} P. Cardaliaguet, B. Seeger, P. Souganidis, Mean field games with common noise and degenerate idiosyncratic noise, 2022, arXiv preprint arXiv: 2207.10209.
\bibitem{18DF} F. Dragoni, E. Feleqi, Ergodic Mean Field Games with H\"{o}rmander diffusions, Calc. Var. Partial Differential Equations, 2018, 57: 1-22.
\bibitem{20FGT} E. Feleqi, D. Gomes, T. Tada, Hypoelliptic mean field games--a case study, Minimax Theory Appl., 2020, 5(2): 305-326.
\bibitem{21FGT} R. Ferreira, D. Gomes, T. Tada, Existence of weak solutions to time-dependent mean-field games, Nonlinear Anal., 2021, 212: 112470.
\bibitem{77Ga} B. Gaveau, Principe de moindre action, propagation de la chaleur et estim\'{e}es sous-elliptiques sur certains groupes nilpotents, Acta Math., 1977, 139, 95-153.
\bibitem{70Gr} V. V. Gru\v{s}hin, On a class of hypoelliptic operators, Math. USSR-Sb., 1970, 12(3): 458-476.
\bibitem{22GMMZ} W. Gangbo, A. R. M\'{e}sz\'{a}ros, C. Mou, and J. Zhang. Mean field games master equations with nonseparable Hamiltonians and displacement monotonicity. Ann. Probab., 2022, 50(6):2178-2217.
\bibitem{15GS} W. Gangbo, A. \'{S}wi\c{e}ch, Existence of a solution to an equation arising from the theory of mean field games, J. Differ. Equ., 2015, 259(11): 6573-6643.
\bibitem{06HCM} M. Huang, P. E. Caines, R. P. Malham\'{e}, Large population stochastic dynamic games: closed-loop McKean-Vlasov systems and the Nash certainty equivalence principle, Comm. Inf. Syst., 2006, 6: 221-251.
\bibitem{76Hu} A. Hulanicki, The distribution of energy in the Brownian motion in the Gaussian field and analytic hypoellipticity of certain subelliptic operators on the Heisenberg group, Studia Math., 1976, 56, 165-173.
\bibitem{23JR} E.R. Jakobsen, A. Rutkowski, The master equation for mean field game systems with fractional and nonlocal diffusions, 2023, arXiv preprint arXiv: 2305.18867.
\bibitem{23JRWX} Y. Jiang, J. Ren, Y. Wei, J. Xue, Degenerate Mean Field Games with H\"ormander diffusion, 2023, arXiv preprint arXiv: 2308.10434.
\bibitem{09Li} H.Q. Li, Fonctions maximales centr\'ees de Hardy-Littlewood sur les groupes de Heisenberg, Studia Math., 2009, 191, 89-100.
\bibitem{Li} P.-L. Lions, Cours au Coll\`{e}ge de France, Available at: www.college-de-france.fr.
\bibitem{07LL} J.-M. Lasry, P.-L. Lions, Mean field games, Jpn. J. Math., 2007, 2(1):229-260.
\bibitem{06LL_I} J.-M. Lasry, P.-L. Lions, Jeux \`{a} champ moyen. I. Le cas stationnaire, C. R. Math. Acad. Sci. Paris., 2006, 343(9): 619-625.
\bibitem{06LL_II} J.-M. Lasry, P.-L. Lions, Jeux \`{a} champ moyen. II. Horizon fini et contr\^{o} le optimal, C. R. Math. Acad. Sci. Paris., 2006, 343:679-684.
\bibitem{08LLG} J.-M. Lasry, P.-L. Lions, O. Gu\`{e}ant, Application of mean field games to growth theory, 2008, hal-00348376.
\bibitem{22LS} H.Q. Li, P. Sjgren, Estimates for Operators Related to the Sub-Laplacian with Drift in Heisenberg Groups, Journal of Fourier Analysis and Applications., 2022, 28(1), 1-29.
\bibitem{03Lu} F. Lust-Piquard, A simple-minded computation of heat kernels on Heisenberg groups, Colloq. Math., 2003, 97, 233-249.
\bibitem{21Mi} N. Mimikos-Stamatopoulos, Weak and renormalized solutions to a hypoelliptic Mean Field Games system, 2021, arXiv preprint arXiv: 2105.05777.
\bibitem{06Mo} R. Montgomery, A Tour of Subriemannian Geometries, Their Geodesics and Applications, 2006.
\bibitem{21Ri} M. Ricciardi, The Master Equation in a bounded domain with Neumann conditions, Communications in Partial Differential Equations, 2021, 47, 912-947.
\bibitem{76RS} L.P. Rothschild, E.M. Stein, Hypoelliptic differential operators and nilpotent groups, Acta Math., 1976, 137, 247-320.
\bibitem{18Ry} L. Ryzhik. Lecture notes, 2018.
\end{thebibliography}
\end{document}